\documentclass[a4paper,12pt]{article}
\usepackage{setspace}
\usepackage{latexsym} 
\usepackage[utf8]{inputenc}
\usepackage[english]{babel}
\usepackage{amsfonts,amsmath,amssymb,amsbsy,amsthm,enumerate}
\usepackage[mathscr]{euscript}
\usepackage{pstricks}
\usepackage{multirow}
\usepackage{verbatim}
\usepackage{color}
\usepackage{graphicx}
\usepackage{float}

\usepackage[sc]{mathpazo}

\theoremstyle{plain}
\newtheorem{theorem}{Theorem}
\newtheorem*{theorem-nn}{Theorem}
\newtheorem{conjecture}{Conjecture}
\newtheorem{corollary}{Corollary}[theorem]
\newtheorem{defi}{Definition}
\newtheorem{remark}{Remark}
\newtheorem{claim}{Claim}[theorem]
\newtheorem{lemma}[theorem]{Lemma}
\newtheorem{proposition}[theorem]{Proposition}

\newcommand{\N}{\mathbb{N}} 

\newcommand{\D}{{\mathcal{D}}}
\newcommand{\cv}{{\rm{cv}}}
\newcommand{\aux}{{\rm{aux}}}
\newcommand{\tr}{{\rm{tr}}}
\newcommand{\hang}{{\rm{hang}}}
\newcommand{\bcv}{{\mathbf\rm{cv}}}
\newcommand{\btr}{{\mathbf\rm{tr}}} 

\usepackage{tikz}
\usepackage{subcaption}
\usepackage[subrefformat=parens,labelformat=parens]{subcaption}
\usetikzlibrary{calc,positioning,decorations.pathmorphing,decorations.pathreplacing}
\tikzset{emp/.style={double distance = 0.3ex}}

\tikzset{M edge/.style={line width=1.3pt,double distance=1.1pt}}
\tikzset{F1 edge/.style={line width=1.3,color=red,->}}
\tikzset{F2 edge/.style={line width=1.3,color=blue,->}}
\tikzset{E edge/.style={line width=1.3,color=black,-}}

\tikzset{squared black vertex/.style={draw,minimum size=2mm,inner sep=0pt,outer sep=3pt,fill=black, color=black}}
\tikzset{red vertex/.style={circle,draw,minimum size=2mm,inner sep=0pt,outer sep=4pt,fill=red, color=red}}
\tikzset{blue vertex/.style={circle,draw,minimum size=2mm,inner sep=0pt,outer sep=4pt,fill=blue, color=blue}}
\tikzset{black vertex/.style={circle,draw,minimum size=2mm,inner sep=0pt,outer sep=3pt,fill=black, color=black}}
\tikzset{white vertex/.style={circle,draw,minimum size=2mm,inner sep=0pt,outer sep=3pt, color=black}}

\tikzset{fatpath/.style={line width=9pt,rounded corners=.1mm}}

\tikzstyle{edge}=[line width=1.3]

\tikzstyle{color1}=[color=blue] 
\tikzstyle{color2}=[color=red]
\tikzstyle{color3}=[color=green] 
\tikzstyle{color4}=[fill=yellow]
\tikzstyle{color5}=[ dashed] 

\tikzstyle{backcolor1}=[color=gray!55!white] 
\tikzstyle{backcolor2}=[color=blue!35!white] 

\tikzstyle{fucking loosely dotted}=[dash pattern=on \pgflinewidth off 6pt]
\tikzstyle{nonedge}=[color=red,fucking loosely dotted]

\tikzstyle{possible edge}=[edge, dashed]

\tikzstyle{snake}=[decorate, decoration=snake, segment length=1cm]
\tikzstyle{short snake}=[decorate, decoration=snake, segment length=7mm]
\tikzstyle{long snake}=[decorate, decoration=snake, segment length=11mm]

\colorlet{setfilling}{green!5!white}
\colorlet{setborder}{gray}

\usepackage{xcolor}

\author{Fábio Botler \hspace{1cm} Luiz Hoffmann\\
	{\footnotesize Programa de Engenharia de Sistemas e Computação}\vspace{-.2cm}\\
	{\footnotesize Universidade Federal do Rio de Janeiro}
	\footnote{{E-mails: fbotler@cos.ufrj.br (F. Botler), hoffmann@cos.ufrj.br (L. Hoffmann).}}}

\title{Decomposition of $(2k+1)$-regular graphs containing special spanning \(2k\)-regular Cayley graphs into paths of length~$2k+1$}

\begin{document}
	\maketitle

\begin{abstract}
		A \(P_\ell\)-decomposition of a graph \(G\) is a set of paths with \(\ell\) edges in \(G\) 
		that cover the edge set of \(G\).
		Favaron, Genest, and Kouider (2010) 
		conjectured that every \((2k+1)\)-regular graph that contains a perfect matching
		admits a {\(P_{2k+1}\)-decomposition}.
		They also verified this conjecture for \(5\)-regular graphs without cycles of length \(4\).
		In 2015, Botler, Mota, and Wakabayashi verified this conjecture 
		for \(5\)-regular graphs without triangles.
		In this paper, we verify it for 
		{\((2k+1)\)-regular} graphs that contain the \(k\)th power of a spanning cycle;
		and for \(5\)-regular graphs that contain
		special spanning \(4\)-regular Cayley graphs.
		
		\smallskip
		\noindent\textbf{keywords:} Decomposition, Path, Regular graph, Cayley graph.
\end{abstract}

\section{Introduction}\label{sec:introduction}
	All graphs in this paper are simple, i.e., 
	have no loops nor multiple edges.
	A \emph{decomposition} of a graph \(G\) is a set \(\D\) 
	of edge-disjoint subgraphs of \(G\) that cover its edge set.
	If every element of \(\D\) is isomorphic to a fixed graph \(H\),
	then we say that~\(\D\) is an \emph{\(H\)-decomposition}.
	In this paper, we focus on the case \(H\) 
	is the simple path with \(2k+1\) edges, which we denote by \(P_{2k+1}\).
	Note that this notation is not standard.
	In 1957, Kotzig~\cite{Ko57} (see also~\cite{BoFo83}) 
	proved that a \(3\)-regular graph \(G\)
	admits a \(P_3\)-decomposition if and only if \(G\) contains a perfect matching.
	In 2010, Favaron, Genest, and Kouider~\cite{FaGeKo10} extended this result by proving 
	that every \(5\)-regular graph that
	contains a perfect matching and no cycles of length \(4\)
	admits a \(P_5\)-decomposition;
	and proposed the following conjecture to 
	extend Kotzig's result.
	\begin{conjecture}[Favaron--Genest--Kouider, 2010]\label{conj:favaron}
		If \(G\) is a \((2k+1)\)-regular graph that contains a perfect matching,
		then \(G\) admits a \(P_{2k+1}\)-decomposition.
	\end{conjecture}
	
	In 2015, Botler, Mota, and Wakabayashi~\cite{BoMoWa15}
	verified Conjecture~\ref{conj:favaron} for triangle-free \hbox{\(5\)-regular},
	and, more recently, Botler, Mota, Oshiro, and Wakabayashi~\cite{BoMoOsWa17}
	generalized this result
	for \((2k+1)\)-regular graphs with girth at least \(2k\).
	
	It is clear that a \(5\)-regular graph contains a perfect matching
	if and only if it contains a spanning \(4\)-regular graph.
	In fact, by using a theorem of Petersen~\cite{Pe1891},
	one can prove that a \((2k+1)\)-regular graph contains a perfect matching
	if and only if it contains a spanning \(2k'\)-regular graph for every \(k'\leq k\).
	\begin{theorem}[Petersen, 1891]\label{theorem:petersen}
		If \(G\) is a \(2k\)-regular graph,
		then \(G\) admits a decomposition into spanning \(2\)-regular graphs.
	\end{theorem}	
	In this paper, we explore Conjecture~\ref{conj:favaron}
	for \((2k+1)\)-regular graphs that contain special spanning \(2k\)-regular graphs as follows.
	Throughout the text, \(\Gamma\) denotes a finite group of order \(n\);
	\(+\) denotes the group operation of~\(\Gamma\); and \(0\) denotes the identity of \(\Gamma\).
	As usual, for each \(x\in\Gamma\), we denote by \(-x\) the \emph{inverse} of $x$,
	i.e., the element \(y\in\Gamma\) for which \(x+y=0\),
	and the operation \(-\) denotes the default binary operation \((x,y)\mapsto x+(-y)\). 
	Let \(S\subseteq\Gamma\) be a set not containing \(0\), and such that \(-x\in S\) for every \(x\in S\) (i.e., \(S\) is closed under taking inverses).
	The \emph{Cayley graph} \(X(\Gamma,S)\) is the graph \(H\) with \(V(H)=\Gamma\),
	and \(E(H) = \big\{xy\colon y-x\in S\big\}\) (see~\cite{godsil13}).
	In this paper, we allow \(S\) to be a set not generating~\(\Gamma\),
	and hence \(X(\Gamma,S)\) is not necessarily connected.	
	We say that \(H\) is \emph{simply commutative}
	if (i) \(x + y = y + x\) for every \(x,y\in S\), 
	and (ii) \(-x\neq x\) for every \(x\in S\).
	Condition (ii) implies that \(H\) has no multiple edges, and, since $0\notin S$, $H$ is simple.
	It is not hard to check that, in such a graph, 
	the neighborhood of a vertex \(v\in\Gamma\)
	is \(N(v) = \{v+x\colon x\in S\}\).
	Although the definition of Cayley graphs can be extended to multigraphs 
	and directed graphs, 
	Conjecture~\ref{conj:favaron} considers only simple graphs.
	In fact, we explore some structure of the colored directed Cayley graph (see \cite{Ca78}) in which the edge set consists of the pairs $(x,x+s)$ with color $s\in S$.
	
	We present two results regarding Conjecture~\ref{conj:favaron}.
	We verify it for \((2k+1)\)-regular graphs
	that contain the \(k\)th power of a spanning cycle (see Section~\ref{sec:powers-of-cycles});
	and for {\(5\)-regular} graphs that contain 
	spanning simply commutative \(4\)-regular Cayley graphs 
	(see Section~\ref{sec:5-regular}).
	Since the graphs in these classes may contain cycles of lengths \(3\) and \(4\),
	these results extend the family of graphs for which 
	Conjecture~\ref{conj:favaron} is known to hold.
	
	We believe that, due to the underlying group structure, the techniques developed here can be extended 
	for dealing with \((2k+1)\)-regular graphs that contain 
	more general spanning Cayley graphs, 
	and also \((2k+1)\)-regular graphs that contain 
	special spanning Schreier graphs,
	which could give us significant insight with respect to the general case of Conjecture~\ref{conj:favaron} (see Section~\ref{sec:concluding}).
	
	\bigskip	\noindent
	\textbf{Notation.}
	A graph $T$ is a \emph{trail} if there is a sequence $x_0,\ldots, x_{\ell}$ of its vertices for which $E(T)=\{x_{i}x_{i+1}\colon 0\leq i \leq \ell-1\}$ and $x_{i}x_{i+1}\neq x_{j}x_{j+1}$, for every $i\neq j$. Moreover,
	if $x_i\neq x_j$ for every $i \neq j$, we say that $T$ is a \emph{path}.
	A subgraph $F$ of a graph $G$ is a \emph{factor} of $G$ or
	a \emph{spanning subgraph} of $G$, if $V(F)=V(G)$. 
	If, additionally, \(F\) is $r$-regular, then we say that $F$ is an \emph{$r$-factor}.
	In particular, a perfect matching is the edge set of a \(1\)-factor.
	Moreover, we say that $F$ is an \emph{$H$-factor} if $F$ is a factor that consists of vertex-disjoints copies of $H$.
	The reader may refer to \cite{bondy1976graph} for standard definitions of graph theory.

	\section{Regular graphs that contain powers of cycles}\label{sec:powers-of-cycles}

	Given a perfect matching $M$ in a graph $G$, we say that a $P_{\ell}$-decomposition $\D$ of a graph $G$ is
	$M$-\emph{centered} if for every  
	$P=a_0a_1\cdots a_{\ell-1}a_{\ell} \in \D$,
	we have $a_{i}a_{i+1} \in M$ for $i=(\ell-1)/2$,
	i.e., if the middle edges of the paths in \(\D\) are precisely the edges of \(M\).
	The next results are examples of $M$-centered decomposition
	that are used in the proof of Theorems~\ref{theorem:no-g2r2} and~\ref{theorem:2g2r!=0}.
	
	\begin{proposition}\label{proposition:K44}
		If \(G\) is a \(5\)-regular graph that contains a spanning copy \(K\) of $K_{4,4}$,
		and \({M = E(G)\setminus E(K)}\), then \(G\) admits an $M$-centered \(P_5\)-decomposition.
	\end{proposition}
	
	\begin{proof}
		Let \(G\), \(K\), and \(M\) be as in the statement.
		Let \((R,L)\) be the bipartition of \(K\), 
		where $R=\{r_1,r_2,r_3,r_4\}$ and $L=\{l_1,l_2,l_3,l_4\}$.
		Since \(K\) is a complete bipartite graph,
		if $xy\in M$, then either $x,y \in R$ or $x,y \in L$.
		Thus, we may assume, without loss of generality, that $M=\{r_1r_2,r_3r_4,l_1l_2,l_3l_4\}$,
		and hence, ${\D=\{l_1r_1l_3l_4r_2l_2, 
		~l_3r_3l_1l_2r_4l_4,~r_1l_2r_3r_4l_1r_2,
		~r_3l_4r_1r_2l_3r_4\}}$ is an $M$-centered decomposition of $G$ as desired (see Figure~\ref{fig:K44}).	
	\end{proof}
	
	\begin{figure}[H]
		\centering
		\scalebox{.8}{\begin{tikzpicture}[scale=1.5]
				\node (r1) [black vertex] at (-1,0) {};
				\node (r2) [black vertex] at (0,0) {};
				\node (r3) [black vertex] at (1,0) {};
				\node (r4) [black vertex] at (2,0) {};
				
				\node (l1) [black vertex] at (-1,-1.5) {};
				\node (l2) [black vertex] at (0,-1.5) {};
				\node (l3) [black vertex] at (1,-1.5) {};
				\node (l4) [black vertex] at (2,-1.5) {};

				\node () [] at ($(r1)+(90:.25)$) {$r_1$};
				\node () [] at ($(r2)+(90:.25)$) {$r_2$};
				\node () [] at ($(r3)+(90:.25)$) {$r_3$};
				\node () [] at ($(r4)+(90:.25)$) {$r_4$};

				\node () [] at ($(l1)+(270:.3)$) {$l_1$};
				\node () [] at ($(l2)+(270:.3)$) {$l_2$};
				\node () [] at ($(l3)+(270:.3)$) {$l_3$};
				\node () [] at ($(l4)+(270:.3)$) {$l_4$};

				\draw[line width=1.3pt,M edge,color=red] (r1) -- (r2);
				\draw[line width=1.3pt,M edge,color=blue] (l1) -- (l2);
				\draw[line width=1.3pt,M edge,color=green] (r3) -- (r4);
				\draw[line width=1.3pt,M edge,color=black] (l3) -- (l4);
				
				\draw[line width=1.3pt,color=green] (r1) -- (l2);
				\draw[line width=1.3pt,color=green] (l2) -- (r3);
				\draw[line width=1.3pt,color=green] (r4) -- (l1);
				\draw[line width=1.3pt,color=green] (l1) -- (r2);
				\draw[line width=1.3pt,color=blue] (l3) -- (r3);
				\draw[line width=1.3pt,color=blue] (r3) -- (l1);
				\draw[line width=1.3pt,color=blue] (l2) -- (r4);
				\draw[line width=1.3pt,color=blue] (r4) -- (l4);
				\draw[line width=1.3pt,color=black] (l1) -- (r1);
				\draw[line width=1.3pt,color=black] (r1) -- (l3);
				\draw[line width=1.3pt,color=black] (l4) -- (r2);
				\draw[line width=1.3pt,color=black] (r2) -- (l2);
				\draw[line width=1.3pt,color=red] (r3) -- (l4);
				\draw[line width=1.3pt,color=red] (l4) -- (r1);
				\draw[line width=1.3pt,color=red] (r2) -- (l3);
				\draw[line width=1.3pt,color=red] (l3) -- (r4);
		\end{tikzpicture}}
		\caption{$P_5$-decomposition of a \(5\)-regular graph that contains a spanning copy of a \(K_{4,4}\).}
		\label{fig:K44}
	\end{figure}
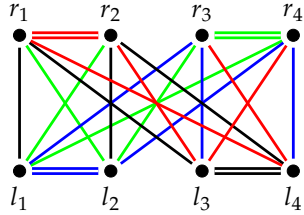

	Given positive integers \(k\) and \(n\),
	the \(k\)th power of the cycle on \(n\) vertices,
	denoted by \(C_n^k\),
	is the graph on the vertex set \(\{0,\ldots,n-1\}\)
	and such that, for every vertex \(v\), 
	we have \(x\in N(v)\) if and only if \(x = v+r\pmod{n}\), 
	where \(r\in\{-k,\ldots,-1\}\cup\{1,\ldots,k\}\).	
	
	\begin{proposition}\label{proposition:powers-of-cycles}
		Let \(n\) and \(k\) be positive integers for which \(k < n/2\).
		If \(G\) is a simple \((2k+1)\)-regular graph on \(n\) vertices
		that contains a copy \(C\) of \(C_n^k\),
		and \(M = E(G) \setminus E(C)\),
		then \(G\) admits an \(M\)-centered \(P_{2k+1}\)-decomposition.
	\end{proposition}

	\begin{proof}
		Let \(G\), \(C\), and \(M\) be as in the statement,
		and let \(V(C) = \{0,\ldots, n-1\}\) as above.
		The operations on the vertices of $C$ are taken modulo $n$.
		Since \(C\) is a \(2k\)-regular graph,
		\(M\) is a perfect matching of \(G\).
		Given \(i\in V(C)\),
		let \(Q_i\) be the path \(v_0v_1\cdots v_k\) 
		in which \(v_0=i\); and, for \(j=1,\ldots,k\),
		we have
		\(v_j = v_{j-1}+j\) if \(j\) is odd; and
		\(v_j = v_{j-1}-j\) if \(j\) is even (see Figure~\ref{fig:PC}).
		Note that for every \(j=1,\ldots,k\),
		the path \(Q_i\) contains an edge \(xy\) such that \(|x-y|=j\).
		Also, we have
		\(V(Q_i) = \big\{i+r\colon r\in\{-\lfloor k/2\rfloor,
		-\lfloor k/2\rfloor+1,\ldots, \lceil k/2\rceil\}\big\}\).	
		It is not hard to check that 
		the set \(\mathcal{Q}=\big\{Q_i\colon i\in V(C)\big\}\) 
		is a \(P_k\)-decomposition of \(C\).
		
		Given an edge \(e=ij\in M\), let \(P_e = Q_i \cup \{ij\} \cup Q_j\).
		Since \(Q_i\) and \(Q_j\) have, 
		respectively, \(i\) and \(j\) as end vertices, 
		and \(E(Q_i)\cap E(Q_j)=\emptyset\),
		the graph \(P_e\) is a trail of length \(2k+1\).
		Thus, since \(\mathcal{Q}\) is a \(P_k\)-decomposition of \(C\),
		and \(M\) is a perfect matching of \(G\),
		the set \(\D=\{P_e\colon e\in M\}\) is a decomposition of \(G\) 
		into trails of length \(2k+1\).
		We claim that, in fact, \(\D\) 
		is a \(P_{2k+1}\)-decomposition of \(G\).
		For that, we prove that if \(ij\in M\),
		then \(V(Q_i)\cap V(Q_j)=\emptyset\).
		Indeed, note that for every \(e=ij\in M\),
		we have \(|i-j|>k\), otherwise we have \(ij\in E(C)\).
		Now, suppose that there is a vertex \(v\) in \(V(Q_i)\cap V(Q_j)\).
		Then, there are \(r_1,r_2\) with
		\(-\lfloor k/2\rfloor \leq r_1,r_2\leq \lceil k/2\rceil\),
		and such that \(i+r_1 = v = j+r_2\).
		Suppose, without loss of generality, that \(i>j\).
		Then, we have \(r_2-r_1 = i-j > k\),
		but \(r_2-r_1 \leq \lceil k/2\rceil+\lfloor k/2\rfloor = k\),
		a contradiction.
	\end{proof}

		\begin{figure}
		\centering
		\scalebox{1.}{\begin{tikzpicture}[scale=1.5]
				
				\node (PC0) [black vertex] at (0,0) {};
				\node (PC1) [black vertex] at (1,0) {};
				\node (PC2) [black vertex] at (1.81,0.59) {};
				\node (PC3) [black vertex] at (2.12,1.54) {};
				\node (PC4) [black vertex] at (1.81,2.49) {};
				\node (PC5) [black vertex] at (1,3.08) {};
				\node (PC6) [black vertex] at (0,3.08) {};
				\node (PC7) [black vertex] at (-0.81,2.49) {};
				\node (PC8) [black vertex] at (-1.12,1.54) {};
				\node (PC9) [black vertex] at (-0.81,0.59) {};
				
				\node () [] at ($(PC0)+(270:.25)$) {$0$};
				\node () [] at ($(PC1)+(270:.25)$) {$1$};
				\node () [] at ($(PC2)+(0:.2)$) {$2$};
				\node () [] at ($(PC3)+(0:.2)$) {$3$};
				\node () [] at ($(PC4)+(0:.2)$) {$4$};
				\node () [] at ($(PC5)+(90:.25)$) {$5$};
				\node () [] at ($(PC6)+(90:.25)$) {$6$};
				\node () [] at ($(PC7)+(180:.2)$) {$7$};
				\node () [] at ($(PC8)+(180:.2)$) {$8$};
				\node () [] at ($(PC9)+(180:.2)$) {$9$};
				
				\draw[line width=1.5pt,color=black] (PC0) -- (PC1);
				\draw[line width=1.3pt,color=black] (PC1) -- (PC2);
				\draw[line width=1.3pt,color=black] (PC4) -- (PC3);
				\draw[line width=1.5pt,color=red] (PC2) -- (PC3);
				\draw[line width=1.5pt,color=red] (PC8) -- (PC9);	
				\draw[line width=1.3pt,color=black] (PC4) -- (PC5);
				\draw[line width=1.5pt,color=black] (PC5) -- (PC6);
				\draw[line width=1.3pt,color=black] (PC6) -- (PC7);
				\draw[line width=1.3pt,color=black] (PC7) -- (PC8);
				\draw[line width=1.3pt,color=black] (PC9) -- (PC0);
				
				\draw[line width=1.3pt,color=black] (PC0) -- (PC2);
				\draw[line width=1.3pt,color=black] (PC4) -- (PC2);
				\draw[line width=1.3pt,color=black] (PC3) -- (PC5);
				\draw[line width=1.5pt,color=black] (PC4) -- (PC6);
				\draw[line width=1.3pt,color=black] (PC5) -- (PC7);
				\draw[line width=1.3pt,color=black] (PC6) -- (PC8);
				\draw[line width=1.5pt,color=red] (PC9) -- (PC7);
				\draw[line width=1.5pt,color=red] (PC3) -- (PC2);	
				\draw[line width=1.5pt,color=red] (PC3) -- (PC1);	
				\draw[line width=1.3pt,color=black] (PC8) -- (PC0);
				\draw[line width=1.5pt,color=black] (PC9) -- (PC1);
				
				\draw[line width=1.3pt,color=black] (PC0) -- (PC3);
				\draw[line width=1.3pt,color=black] (PC2) -- (PC5);
				\draw[line width=1.5pt,color=red] (PC1) -- (PC4);
				\draw[line width=1.5pt,color=red] (PC7) -- (PC0);
				\draw[line width=1.3pt,color=black] (PC3) -- (PC6);
				\draw[line width=1.5pt,color=black] (PC7) -- (PC4);
				\draw[line width=1.3pt,color=black] (PC5) -- (PC8);
				\draw[line width=1.3pt,color=black] (PC6) -- (PC9);
				\draw[line width=1.3pt,color=black] (PC8) -- (PC1);
				\draw[line width=1.5pt,color=black] (PC9) -- (PC2);
				
				\draw[line width=1.3pt,M edge,color=red] (PC2) -- (PC8);
				\draw[line width=1.3pt,M edge,color=black] (PC0) -- (PC5);
				\draw[line width=1.3pt,M edge,color=black] (PC1) -- (PC6);
				\draw[line width=1.3pt,M edge,color=black] (PC3) -- (PC7);
				\draw[line width=1.3pt,M edge] (PC4) -- (PC9);
				
		\end{tikzpicture}}
		\caption{The path $P_{uv}$, with $u=2$ and $v=8,$ in the proof of Proposition~\ref{proposition:powers-of-cycles} for a $7$-regular graph that contains a spanning copy of a $C^3_{10}$.}
		\label{fig:PC}
	\end{figure}
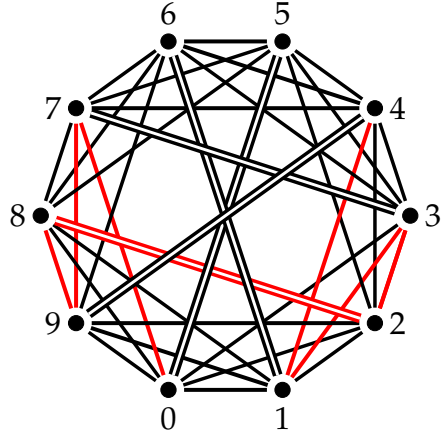
	
	Note that, from the proof of Proposition~\ref{proposition:powers-of-cycles}
	we also obtain a construction for the Hamilton path decomposition 
	of complete graphs of even order.
	
	\begin{corollary}
		Let \(\ell\) be odd.
		The complete graph \(K_{\ell+1}\) admits a \(P_\ell\)-decomposition.
	\end{corollary}
	
	A slight variation of the proof of Proposition~\ref{proposition:powers-of-cycles} also provides the following result.
	
	\begin{proposition}\label{lemma:cycle-factorization}
		Let \(\ell\) be odd, and
		let \(G\) be an \(\ell\)-regular graph with a perfect matching \(M\).
		If each component of \(G\setminus M\)  is the \(\frac{\ell-1}{2}\)-th power of a cycle,
		then \(G\) admits an \(M\)-centered \(P_\ell\)-decomposition.
	\end{proposition}
	
	Let \(G_1\) and \(G_2\) be disjoint graphs with perfect matchings \(M_1\) and \(M_2\), respectively.
	Let \(a_1b_1,\ldots,a_kb_k \in M_1\) and \(x_1y_1,\ldots, x_ky_k\in M_2\) be distinct edges,
	and let \(G\) be the graph obtained from the disjoint union \(G_1\cup G_2\) by removing \(a_1b_1,\ldots,a_kb_k,x_1y_1,\ldots, x_ky_k\)
	and adding the edges \(a_1x_1,b_1y_1,\ldots,a_kx_k,b_ky_k\).
	We say that \(G\) is a \emph{collage} of \(G_1\) and \(G_2\) \emph{over edges} of \(M_1\) and \(M_2\),
	and denote by \(M_G\) the perfect matching \({\big(M_1\cup M_2 \cup \{a_1x_1,b_1y_1,\ldots,a_kx_k,b_ky_k\}\big) \setminus \{a_1b_1,\ldots,a_kb_k,x_1y_1,\ldots, x_ky_k\}}\). 
	When \(M_1\) and \(M_2\) are clear from the context, we say simply that \(G\) is a \emph{collage} of \(G_1\) and \(G_2\).
	Note that  \(G\) is \(\ell\)-regular if and only if \(G_1\) and \(G_2\) are \(\ell\)-regular.
	
	Let \(G\) be an \(\ell\)-regular graph, where \(\ell\) is an odd positive integer, and let \(M\) be a perfect matching of \(G\).
	We say that \(G\) is \emph{\(M\)-constructable} if either \(G\) admits an \(M\)-centered \(P_\ell\)-decomposition,
	or \(G\) is the collage of an \(M_1\)-constructable graph and an \(M_2\)-constructable graph over edges of \(M_1\) and \(M_2\) and \(M=M_G\).
	The next straightforward result is a useful tool in the proof of Theorem~\ref{theorem:2g2r!=0}.
	
	\begin{lemma}\label{lemma:collage}
		Let \(\ell\) be odd, and \(G\) be an \(\ell\)-regular graph.
		If \(G\) is \(M\)-constructable, then~\(G\) admits an \(M\)-centered \(P_\ell\)-decomposition.
	\end{lemma}
	
	\begin{proof}
		The proof follows by induction on \(|V(G)|\).
		By the definition of \(M\)-constructable, 
		we may assume that \(G\) is the collage of an \(M_1\)-constructable graph \(G_1\)
		and an \(M_2\)-constructable graph \(G_2\) over edges of \(M_1\) and \(M_2\).
		By the induction hypothesis, \(G_i\) admits an \(M_i\)-centered \(P_\ell\)-decomposition \(\D_i\), for \(i=1,2\).
		Let \(a_i,b_i,x_i,y_i\), for \(i=1,\ldots,k\) be such that
		\(G\) is the graph obtained from \(G_1\cup G_2\) by removing \(a_1b_1,\ldots,a_kb_k,x_1y_1,\ldots, x_ky_k\)
		and adding \(a_1x_1,b_1y_1,\ldots,a_kx_k,b_ky_k\) as above.
		For  \({i=1,\ldots,k}\), let \(P_i\in\D_1\) and \(Q_i\in\D_2\) be the paths containing the edges \(a_ib_i\) and \(x_iy_i\), respectively.
		By the def\-i\-ni\-tion of \(M_1\)- and \(M_2\)-centered \(P_\ell\)-decomposition, for \(i=1,\ldots,k\),
		we may write \(P_i = P_{i,1}a_ib_i P_{i,2}\) and \(Q_i = Q_{i,1}x_iy_iQ_{i,2}\),
		where \(P_{i,1},P_{i,2},Q_{i,1}\) and \(Q_{i,2}\) are paths of length \((\ell-1)/2\).
		Since \(G_1\) and \(G_2\) are disjoint, \(V(P_{i,j})\cap V(Q_{i,j}) = \emptyset\) for \(i=1,\ldots,k\) and \(j=1,2\).
		Let \(R_{i,1} = P_{i,1}\cup \{a_ix_i\}\cup Q_{i,1}\) and \(R_{i,2} = P_{i,2}\cup \{b_iy_i\} \cup Q_{i,2}\),
		and note that \(\D = \big(\D_1\setminus\{P_1,\ldots,P_k\}\big)\cup\big(\D_2\setminus\{Q_1,\ldots,Q_k\}\big)\cup\{R_{i,j}\colon i=1,\ldots,k\text{ and }j=1,2\}\)
		is an \(M_G\)-centered \(P_\ell\)-decomposition of \(G\) as desired.
	\end{proof}
	
	By Proposition~\ref{proposition:K44}, if $G$ contains a spanning copy $K$ of $K_{4,4}$ and $M=E(G)\setminus E(K)$, then $G$ is $M$-constructable. Therefore, Lemma~\ref{lemma:collage} yields the following result.
	
	\begin{corollary}\label{corollary:K44}
		If $G$ is a 5-regular graph that contains a $K_{4,4}$-factor $K$ and ${M=E(G)\setminus E(K)}$, then $G$ admits an $M$-centered $P_5$-decomposition.
	\end{corollary}

	\section{$5$-regular graphs that contain Cayley graphs}\label{sec:5-regular}

	In this section, we explore \(5\)-regular graphs that contain spanning simply 
	commutative \(4\)-regular Cayley graphs.
	Botler, Mota, and Wakabayashi~\cite{BoMoWa15} showed that every triangle-free 5-regular graph $G$ that has a perfect matching admits a $P_5$-decomposition. For that, they applied the following strategy: i) to find an initial decomposition of $G$ into paths and trails; 
	and ii) to perform exchanges of edges between the elements of $\D$, preserving a special invariant, while minimizing the number of trails that are not paths.
	
	The proof of our main result (Theorem~\ref{thm:main}) 
	consists of four steps.
	First, we deal with a somehow degenerate case (Theorem~\ref{theorem:no-g2r2}).
	After that, we follow the framework used by Botler, Mota, Wakabayashi~\cite{BoMoWa15},
	i.e, from the structure of Cayley graphs, we find an initial decomposition into trails,
	not necessarily paths (Proposition~\ref{proposition:initial-decomposition}), and then we exchange edges
	between the elements of the decomposition in order to reduce the number of \emph{bad elements}
	(the trails that are not paths). 
	For that, we first show that the bad elements are distributed in a circular fashion (Lemma~\ref{lemma:P5-decomposition}),
	and then we show how to deal with these ``cycles of bad elements'' 
	(Theorem~\ref{theorem:2g2r!=0}). 
	The invariants preserved by the operations in the proofs of Lemma~\ref{lemma:P5-decomposition} and Theorem~\ref{theorem:2g2r!=0} are presented, respectively, in Definitions~\ref{def:complete-commutative}
	 and~\ref{def:semi-complete-commutative}.
	
	The following lemma is used often throughout the text.

	\begin{lemma}\label{lemma:no-cycles}
		Let \(\ell\) be odd, and \(G\) be an \(\ell\)-regular graph.
		If \(\D\) is a decomposition of~\(G\) into trails of length \(\ell\),
		then each vertex of \(G\) is the end vertex of precisely one element of~\(\D\).
	\end{lemma}
	
	\begin{proof}
		Let \(k\), \(G\) and \(\D\) be as in the statement.
		Let \(n=|V(G)|\).
		Given an element \(T\in\D\), we denote by \(o(T)\) the number of vertices $v$ in $T$ for which $d_T(v)$ is odd,
		and given a vertex \(v\in V(G)\),
		we denote by \(\D(v)\) the number of trails in \(\D\) for which $d_T(v)$ is odd.
		Clearly, \(\sum_{T\in \D} o(T) = \sum_{v\in V(G)}\D(v)\).
		Moreover, for every trail $T$, we have $o(T)\leq 2$.
		Also, since every element of \(\D\) has \(\ell\) edges,
		we have \(|\D| = \frac{1}{\ell}|E(G)| 
		= \frac{1}{\ell}\frac{1}{2}\ell n = \frac{1}{2}n\).
		Thus, we have  \(\sum_{T\in\D}o(T) \leq 2|\D| = n\).
		Now, since \(v\in V(G)\) has odd degree (in $G$), 
		\(v\) must have odd degree in at least one element of \(\D\),
		and hence \(\D(v) \geq 1\).
		Thus, we have \(\sum_{v\in V(G)}\D(v)\geq n\),
		and hence \(n \leq \sum_{v\in V(G)}D(v) = \sum_{T\in\D} o(T) \leq n\).
		This implies that \(\D(v) = 1\) for every \(v\in V(G)\), as desired.	
	\end{proof}
	
	Recall that \(\Gamma\) is a finite group of order \(n\) and operation \(+\).
	Fix two elements \(g,r\) of~\(\Gamma\),
	we say that \(\{g,r\}\) is a \emph{simple commutative generator} (SCG)
	if (a) \(0\notin \{g, r, 2g, 2r\}\)\label{remark:no-multiple-edges};
	(b) \(g\notin\{r,-r\}\);
	and (c) \(g+r = r+g\).
	Let \(\{g,r\}\) be an SCG, put \(S = \{g, -g, r, -r\}\), and consider the Cayley graph \(X=X(\Gamma,S)\).
	By construction, \(X\) is a simply commutative Cayley graph (see Section~\ref{sec:introduction}).
	Conditions (a) and (b) guarantee that \(X\) is a simple graph,
	while condition (c) introduces the main restriction explored in this paper.
	In this case, we say that~\(X\) is the graph \emph{generated} by \(\{g,r\}\),
	and that \(\{g,r\}\) is the \emph{generator} of \(X\).
	Finally, we say that a simple \(5\)-regular graph \(G\) with vertex set \(\Gamma\)
	is a \emph{\(\{g,r\}\)-graph}
	if \(G\) contains a spanning Cayley graph \(X\) generated by \(\{g,r\}\).
	We say that \(G\) is a \emph{simply commutative generated graph} or, for short, \emph{SCG-graph} if~\(G\) 
	is a \(\{g,r\}\)-graph for some SCG \(\{g,r\}\).
	In this section, we verify Conjecture~\ref{conj:favaron} for SCG-graphs. In particular,
	Proposition~\ref{lemma:cycle-factorization} implies that every $\{g, r\}$-graph for which $2g = r$ admits an $M_{g,r}$-centered decomposition;
	and as a consequence of Corollary~\ref{corollary:K44}, we obtain the following result, 
	which is also a special case of our main result.
	
	\begin{theorem}\label{theorem:no-g2r2}
		Every $\{g,r\}$-graph for which $2g+2r=0$ and \(2g-2r=0\) admits an $M_{g,r}$-centered 
		decomposition.
	\end{theorem}
	
	\begin{proof}	
		Let \(G\) be a \(\{g,r\}\)-graph for which \(2g+2r=0\) and  \(2g-2r=0\) and put \({M=M_{g,r}}\).
		Note that we also have \(4g=4r=0\).
		Let \(u\) be a vertex of \(G\),
		and let \(H\) be the component of \(G\setminus E(M)\) that contains \(u\).
		In what follows, we prove that \(H\) is a copy of \(K_{4,4}\).
		Since $g$ and $r$ commute, if \(v\in V(H)\), we have \(v = u + ig + jr\), where \(i,j\in\mathbb{N}\).
		Since \(4g=4r=0\), we may assume \(i,j\in\{0,1,2,3\}\).
		Moreover, since \(2g - 2r = 0\) (and hence \(2g=2r\)), we may assume \(j\in\{0,1\}\).
		Therefore, there are at most eight vertices in \(H\), namely, \(V(H) = \{u, u+g,u+2g,u+3g, u+r,u+2g+r,u+3g+r\}\).
		We claim that \(H\) is bipartite.
		Indeed, suppose that there is an odd cycle \(C\) in \(H\).
		Then, there is an element $x \in V(C)$ such that $x+ig+jr=x$, where $i,j \in \mathbb{N}$. 
		Note that \(i+j\) can be obtained from the length of \(C\) by ignoring pairs of edges with the same color and different directions.
		Since $C$ is odd, precisely one between $i$ and $j$ is odd.
		Suppose, without loss of generality, that $i$ is odd, and hence \(j\) is even.
		Note that, since \(2g=2r\), we have \(jr = jg\).
		Therefore, $(i+j)g = ig + jr = 0$. 
		Let \(s\in\{1,3\}\) be such that \(i+j = 4q + s\) for some \(q\in\N\).
		Then we have \(0 = (i+j)g = 4qg + sg\), which implies \(sg = 0\).
		Thus, if \(s = 1\), then \(g=0\);
		and if \(s = 3\), then \(g = 4g-sg = 0\),
		a contradiction to the definition of SCG.
		Thus, since \(H\) is \(4\)-regular, \(H\) is a copy of \(K_{4,4}\).
		Now, since every component of \(G\setminus E(M)\) is isomorphic to \(K_{4,4}\).
		Therefore, $G$ is a $5$-regular graph that contains a $K_{4,4}$-factor,
		and hence by Corollary~\ref{corollary:K44}, \(G\) admits an \(M\)-centered decomposition as desired.
	\end{proof}

	If \(X\) is the graph generated by an SCG \(\{g,r\}\),
	and \(x\in\{g,r\}\),
	then we denote by \(F_x\) the \(2\)-factor of \(X\) 
	with edge set \(E(F_x) = \{v+x\colon v\in\Gamma\}\).
	If \(G\) is a \(\{g,r\}\)-graph, then we denote by \(M_{g,r}\) the perfect matching \(G\setminus E(F_g \cup F_r)\),
	and the triple \(\{M_{g,r},F_g,F_r\}\) is called the \emph{base factorization} of \(G\).
	Although \(G\) is a simple graph, for ease of notation, 
	we refer to an edge \(uv\in F_x\), with \(x\in\{g,r\}\), as a \emph{green} (resp. \emph{red}) \emph{out edge} of \(u\) and \emph{in edge} of \(v\) if \(v = u+x\) and \(x=g\) (resp. \(x=r\)).
	In the figures throughout the text,
	the edges in \(F_g\), \(F_r\), \(M_{g,r}\)
	are illustrated, respectively, in dotted green, dashed red, 
	and double black patterns, 
	while edges without specific affiliation are illustrated in straight gray pattern.
	Moreover, if such an edge has a specific direction (i.e., in edge or out edge), it is illustrated accordingly.
	Note that each vertex \(u\) has precisely one edge of each type (green in edge, green out edge, red in edge, red out edge),
	and is incident to precisely one edge of \(M_{g,r}\).
	In particular, the group structure overcomes Theorem~\ref{theorem:petersen} by giving a decomposition of \(X\) into \(2\)-factors in terms of the elements \(g\) and \(r\).
	
	In the rest of the paper we deal with the case $2g+2r\neq 0$.
	For that, we characterize the elements of the forthcoming decompositions.
	
	\begin{defi}\label{def:types}
		We say that a trail \(T\) in a \(\{g,r\}\)-graph is of type~A, B, C, or D 
		if \(T\) can be written as \(a_0a_1a_2a_3a_4a_5\), 
		where \(a_0,a_1,a_2,a_3,a_4\) are distinct vertices, 
		as follows.
		\begin{enumerate}[type~A:]
			\item
			\(a_2 = a_5\),
			\(a_2a_3\in M_{g,r}\),
			\(a_2a_1,a_3a_4\in F_g\), 
			\(a_4a_5 \in F_r\), and
			\(a_1a_0\in F_g\cup F_r\cup M_{g,r}\),
			i.e., \(a_1a_0\) is an out edge of \(a_1\), or \(a_1a_0\in M_{g,r}\)
			(see Figure~\ref{fig:types}\subref{fig:typeA}).
			In this case, we say that \(a_3\) 
			is the \emph{primary connection vertex} of~\(T\), 
			\(a_2\) is the \emph{secondary connection vertex}
			of \(T\);
			\(a_1\) is the \emph{auxiliary vertex} of \(T\);
			and \(a_4\) is the \emph{tricky vertex} of \(T\).
			We denote these vertices, respectively, by
			\(\cv_1(T)\), \(\cv_2(T)\), \(\aux(T)\), and \(\tr(T)\);
			
			\item
			$a_5\notin\{a_0,a_1,a_2,a_3,a_4\}$,
			\(a_2a_3\in M_{g,r}\)
			\(a_2a_1,a_3a_4\in F_g\), 
			\(a_1a_0,a_4a_5\in F_g\cup F_r\cup M_{g,r}\)
			(see Figure~\ref{fig:types}\subref{fig:typeB});
			
			\item
			$a_5\notin\{a_0,a_1,a_2,a_3,a_4\}$,
			\(a_2a_1,a_4a_3\in F_g\),
			\(a_3a_2,a_4a_5\in F_r\),
			\(a_1a_0\in F_g\cup F_r\cup M_{g,r}\),
			and, moreover, we have \(a_2a_4\in E(G)\) and \(a_2a_4 \in M_{g,r}\)
			(see Figure~\ref{fig:types}\subref{fig:typeC});
			
			\item
			$a_5\notin\{a_0,a_1,a_2,a_3,a_4\}$,
			\(a_1a_0,a_4a_5\in F_r\), 
			\(a_1a_2,a_3a_4\in M_{g,r}\), and \(a_3a_2\in F_g\)
			(see Figure~\ref{fig:types}\subref{fig:typeD}).
		\end{enumerate}
	\end{defi}

	\begin{figure}[H]
		\centering
		\begin{subfigure}{.24\textwidth}
			\centering
			\scalebox{.7}{\begin{tikzpicture}[scale=1.5]
					\node (0) [black vertex] at ($(0,1)+(120:1)$) {};
					\node (1) [black vertex] at (120:1) {};
					\node (2) [black vertex] at (0,0) {};
					\node (3) [black vertex] at (0:1) {};
					\node (4) [black vertex] at (60:1) {};
					\node (x) [] at ($(1,0)+(60:1)$) {};
					\node (y) [] at ($(1,1)+(60:1)$) {};
					
					\node () [] at ($(0)+(180:.3)$) {$a_0$};
					\node () [] at ($(1)+(180:.3)$) {$a_1$};
					\node () [] at ($(2)+(-90:.3)$) {$a_2=a_5$};
					\node () [] at ($(3)+(0:.3)$) {$a_3$};
					\node () [] at ($(4)+(90:.3)$) {$a_4$};
					\node () [] at ($(x)+(0:.3)$) {\color{white}$a_4$};
					\node () [] at ($(y)+(0:.3)$) {\color{white}$a_5$};
					
					\draw[line width=1.5pt,color=gray,->] (1) -- (0);
					\draw[line width=1.5pt,dotted,color=green,->] (2) -- (1);
					\draw[line width=1.5pt,dotted,color=green,->] (3) -- (4);
					\draw[line width=1.5pt,dashed,color=red,->] (4) -- (2);
					\draw[line width=1.2pt,M edge] (2) -- (3);
			\end{tikzpicture}}
			\caption{Type A.}
			\label{fig:typeA}
		\end{subfigure}
		\begin{subfigure}{.24\textwidth}
			\centering
			\scalebox{.7}{\begin{tikzpicture}[scale=1.5]
					\node (0) [black vertex] at ($(0,1)+(120:1)$) {};
					\node (1) [black vertex] at (120:1) {};
					\node (2) [black vertex] at (0,0) {};
					\node (3) [black vertex] at (1,0) {};
					\node (4) [black vertex] at ($(1,0)+(60:1)$) {};
					\node (5) [black vertex] at ($(1,1)+(60:1)$) {};
					
					\node () [] at ($(0)+(180:.3)$) {$a_0$};
					\node () [] at ($(1)+(180:.3)$) {$a_1$};
					\node () [] at ($(2)+(-90:.3)$) {$a_2$};
					\node () [] at ($(3)+(-90:.3)$) {$a_3$};
					\node () [] at ($(4)+(0:.3)$) {$a_4$};
					\node () [] at ($(5)+(0:.3)$) {$a_5$};
					
					\draw[line width=1.3pt,color=gray,->] (1) -- (0);
					\draw[line width=1.3pt,color=gray,->] (4) -- (5);
					\draw[line width=1.5pt,dotted,color=green,->] (2) -- (1);
					\draw[line width=1.5pt,dotted,color=green,->] (3) -- (4);
					\draw[line width=1.3pt,M edge] (2) -- (3);
			\end{tikzpicture}}
			\caption{Type B.}
			\label{fig:typeB}
		\end{subfigure}	
		\begin{subfigure}{.24\textwidth}
			\centering
			\scalebox{.7}{\begin{tikzpicture}[scale=1.5]
					\node (0) [black vertex] at ($(0,1)+(120:1)$) {};
					\node (1) [black vertex] at (120:1) {};
					\node (2) [black vertex] at (0,0) {};
					\node (3) [black vertex] at (0:1) {};
					\node (4) [black vertex] at (60:1) {};
					\node (5) [black vertex] at ($(3)+(60:1)$) {};
					
					\node () [] at ($(0)+(180:.3)$) {$a_0$};
					\node () [] at ($(1)+(180:.3)$) {$a_1$};
					\node () [] at ($(2)+(-90:.3)$) {$a_2$};
					\node () [] at ($(3)+(-90:.3)$) {$a_4$};
					\node () [] at ($(4)+(90:.3)$) {$a_3$};
					\node () [] at ($(5)+(90:.3)$) {$a_5$};
					
					\draw[line width=1.3pt,color=gray,->] (1) -- (0);
					\draw[line width=1.5pt,dotted,color=green,->] (2) -- (1);
					\draw[line width=1.5pt,dotted,color=green,->] (3) -- (4);
					\draw[line width=1.5pt,dashed,color=red,->] (4) -- (2);
					\draw[line width=1.5pt,dashed,color=red,->] (3) -- (5);
			\end{tikzpicture}}
			\caption{Type C.}
			\label{fig:typeC}
		\end{subfigure}
		\begin{subfigure}{.24\textwidth}
			\centering
			\scalebox{.7}{\begin{tikzpicture}[scale=1.5]
					\node (0) [black vertex] at ($(0,1)+(120:1)$) {};
					\node (1) [black vertex] at (120:1) {};
					\node (2) [black vertex] at ($(1,0)+(60:1)$) {};
					\node (3) [black vertex] at (1,0) {};
					\node (4) [black vertex] at (0,0) {};
					\node (5) [black vertex] at ($(1,1)+(60:1)$) {};
					
					\node () [] at ($(0)+(180:.3)$) {$a_0$};
					\node () [] at ($(1)+(180:.3)$) {$a_1$};
					\node () [] at ($(2)+(0:.3)$) {$a_4$};
					\node () [] at ($(3)+(-90:.3)$) {$a_3$};
					\node () [] at ($(4)+(-90:.3)$) {$a_2$};
					\node () [] at ($(5)+(0:.3)$) {$a_5$};
					
					\draw[line width=1.5pt,dashed,color=red,->] (1) -- (0);
					\draw[line width=1.3pt,M edge] (4) -- (1);
					\draw[line width=1.5pt,dashed,color=red,->] (2) -- (5);
					\draw[line width=1.5pt,dotted,color=green,->] (3) -- (4);
					\draw[line width=1.3pt,M edge] (2) -- (3);
				\end{tikzpicture}
			}
			\caption{Type D.}
			\label{fig:typeD}
		\end{subfigure}
		\caption{The main types of trails.}
		\label{fig:types}
	\end{figure}
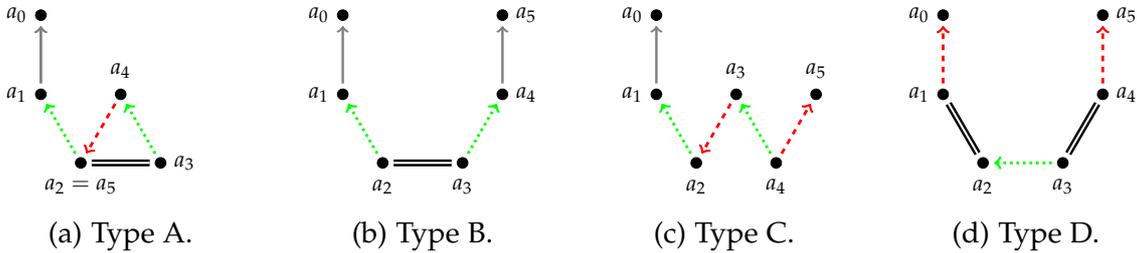
	
	We remark that elements of type~A are not paths,
	while elements of type~B, C, and D are paths.
	Moreover, the connection vertices are defined only for elements of type~A,
	and the connection vertices of an element \(T\) are always incident to an edge of \(M_{g,r}\) in \(T\),
	and hence, no vertex of a \(\{g,r\}\)-graph is a connection vertex of two edge-disjoint elements of type~A in a graph.

	Given a trail (not necessarily a path) 
	\(T= a_0a_1a_2a_3a_4a_5\) in a decomposition \(\D\) of a \(\{g,r\}\)-graph \(G\),
	we say that the edge \(a_1a_0\) (resp.  \(a_4a_5\)) is a \emph{hanging edge}
	at \(a_1\) (resp. \(a_4\)) if \(a_1a_0\in M_{g,r}\cup F_g\cup F_r\) (resp. \(a_4a_5\in M_{g,r}\cup F_g\cup F_r\)),
	i.e., the hanging edges of \(T\) are the end edges of \(T\) that are in \(M_{g,r}\) or that are in edges of its end vertices.
	By Definition~\ref{def:types}, all end edges of elements of type~A, B, C, or D are hanging edges.
	Note that if \(T\) is an element of type~A where \(a_5 = a_2\), then \(a_1a_0\), \(a_2a_3\) and \(a_4a_2\) are hanging edges 
	of \(T\) at, respectively, \(a_1\), \(a_3\), and \(a_4\).
	Given a trail decomposition \(\D\) of a graph \(G\) and a vertex \(u\in V(G)\),
	we denote by \(\hang_\D(u)\) the number of edges of \(G\) that are hanging edges at \(u\).
	
	The next lemma presents a consequence of the exchange of hanging edges at primary connection vertices.

	\begin{lemma}\label{lemma:T'1-is-path}
		If \(T= a_0a_1a_2a_3a_4a_5\) is an element of type~A in a decomposition of a \(\{g,r\}\)-graph \(G\) into trails of length \(5\),
		where \(a_5=a_2\) and \(a_2a_3\in M_{g,r}\),
		and \(u\in V(G)\) is such that \(a_3u\) is a hanging edge at \(a_3 = \cv_1(T)\),
		then \(T'= a_0a_1a_2a_4a_3u\) is of type~C.
	\end{lemma}
	
	\begin{proof}
		Let \(T\), \(u\), and \(T'\) be as in the statement.
		Since \(a_3a_4\) is a green out edge of \(a_3\) and \(a_2a_3\) is an edge of \(M_{g,r}\) 
		incident to \(a_3\), we conclude that \(a_3u\) is a red out edge of~\(a_3\), and hence \(u = a_3 + r\).
		Now, since \(G\) is simple, we have \(u\notin\{a_2,a_3,a_4\}\);
		if \(u=a_1\), then we have \(a_3 + r= u = a_1 = a_3 + g + r + g\), 
		which implies \(2g=0\), a contradiction to the definition of SCG.
		Finally, by Lemma~\ref{lemma:no-cycles} we have \(u\neq a_0\).
		Thus, \(T'\) is a path.
		Since \(a_3u\in F_r\), \(T'\) is of type~C.
	\end{proof}
	
	\begin{figure}[H]
		\centering
		\begin{subfigure}{.45\textwidth}
			\centering
			\scalebox{.8}{\begin{tikzpicture}[scale=1.5]
					
					\draw[fatpath,backcolor1] (120:1)-- ($(-1,0)+(120:1)$) -- (120:1) -- (0,0) -- (0:1) -- (60:1) -- (0,0) -- (60:1);

					\node (0) [black vertex] at ($(-1,0)+(120:1)$) {};
					\node (1) [black vertex] at (120:1) {};
					\node (2) [black vertex] at (0,0) {};
					\node (3) [black vertex] at (0:1) {};
					\node (4) [black vertex] at (60:1) {};
					\node () []		 at ($(2,0)+(60:1)$) {};

					\node (4') [black vertex] at (1,0) {};
					\node (5') [black vertex] at ($(4')+(1,0)$) {};

					\node () [] at ($(0)+(90:.2)$) {$a_0$};
					\node () [] at ($(1)+(90:.2)$) {$a_1$};
					\node () [] at ($(2)+(-90:.2)$) {$a_2$};
					\node () [] at ($(3)+(45:.3)$) {$a_3$};
					\node () [] at ($(4)+(90:.2)$) {$a_4$};

					\node () [] at ($(4')+(-45:.3)$) {};
					\node () [] at ($(5')+(90:.2)$) {$u$};

					\draw[line width=1.3pt,color=gray,<-] (0) -- (1);
					\draw[line width=1.5pt,dotted,color=green,->] (2) -- (1);
					\draw[line width=1.5pt,dotted,color=green,->] (3) -- (4);
					\draw[line width=1.5pt,dashed,color=red,->] (4) -- (2);
					\draw[line width=1.3pt,M edge] (2) -- (3);

					\draw[line width=1.3pt,dashed,color=red,->] (4') -- (5');
					
				\end{tikzpicture}
			}
		\end{subfigure}
		\begin{subfigure}{.45\textwidth}
			\centering
			\scalebox{.8}{\begin{tikzpicture}[scale=1.5]
	
					\draw[fatpath,backcolor1] (120:1) -- ($(-1,0)+(120:1)$) -- (120:1) -- (0,0)  (60:1) -- (0,0) -- (60:1) -- (0:1) -- ($(2,0)$) -- (0:1);

					\node (0) [black vertex] at ($(-1,0)+(120:1)$) {};
					\node (1) [black vertex] at (120:1) {};
					\node (2) [black vertex] at (0,0) {};
					\node (3) [black vertex] at (0:1) {};
					\node (4) [black vertex] at (60:1) {};
					\node () []		 at ($(2,0)+(60:1)$) {};

					\node (4') [black vertex] at (1,0) {};
					\node (5') [black vertex] at ($(4')+(1,0)$) {};

					\node () [] at ($(0)+(90:.2)$) {$a_0$};
					\node () [] at ($(1)+(90:.2)$) {$a_1$};
					\node () [] at ($(2)+(-90:.2)$) {$a_2$};
					\node () [] at ($(3)+(45:.3)$) {$a_3$};
					\node () [] at ($(4)+(90:.2)$) {$a_4$};

					\node () [] at ($(4')+(-45:.3)$) {};
					\node () [] at ($(5')+(90:.2)$) {$u$};

					\draw[line width=1.3pt,color=gray,<-] (0) -- (1);
					\draw[line width=1.5pt,dotted,color=green,->] (2) -- (1);
					\draw[line width=1.5pt,dotted,color=green,->] (3) -- (4);
					\draw[line width=1.5pt,dashed,color=red,->] (4) -- (2);
					\draw[line width=1.3pt,M edge] (2) -- (3);

					\draw[line width=1.3pt,dashed,color=red,->] (4') -- (5');
				\end{tikzpicture}
			}
		\end{subfigure}
		\caption{Exchange of edges performed in the proof of Lemma~\ref{lemma:T'1-is-path}.}
		\label{fig:T'1-is-path}
	\end{figure}
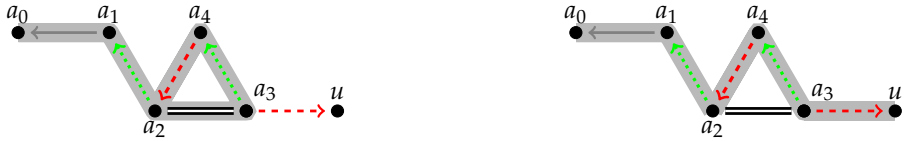

	The following lemma shows how two elements of type~A may be connected.
	
	\begin{lemma}\label{lemma:type12}
		If \(T_1\) and \(T_2\) are two edge-disjoint elements of type~A in a \(\{g,r\}\)-graph \(G\)
		for which \(\tr(T_2) = \cv_1(T_1)\),
		then \(\aux(T_2) = \cv_2(T_1)\).
	\end{lemma}
	
	\begin{proof}
		Let \(T_1 = a_0a_1a_2a_3a_4a_5\)
		and \(T_2 = b_0b_1b_2b_3b_4b_5\), where \(a_5 = a_2\) and \(b_5 = b_2\)
		and \(a_2a_3,b_2b_3\in M_{g,r}\).
		If \(\cv_1(T_1)=\tr(T_2)\),
		then \(a_3=b_4\).
		Since \(b_1 = b_4 + r+g\) and \(a_2 = a_3 + g + r\).
		Thus, \(\aux(T_2) = b_1 = b_4 + r + g = a_3 + r + g = a_2 = \cv_2(T_1)\),
		as desired (see Figure~\ref{fig:type12}).
	\end{proof}

	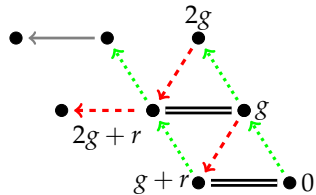
\begin{figure}[H]
		\centering
		\scalebox{.8}{\begin{tikzpicture}[scale=1.5]

				\node (0) [black vertex] at (-1.5,0.8) {};
				\node (1) [black vertex] at (-0.5,0.8) {};
				\node (2) [black vertex] at (0,0) {};
				\node (3) [black vertex] at (1,0) {};
				\node (4) [black vertex] at (0.5,0.8) {};
				\node () []		 at ($(0,0)+(60:1)$) {};

				\node (4') [black vertex] at ($(1,0)$) {};
				\node (3') [black vertex] at ($(3)+(0.5,-0.8)$) {};	
				\node (2') [black vertex] at ($(2)+(0.5,-0.8)$) {};
				\node (1') [black vertex] at ($(0,0)$) {};
				\node (0') [black vertex] at ($(0,0)+(-1,0)$) {};
				\node (5') [] at ($(2,0)$) {};

				\node () [] at ($(0)+(180:.2)$) {};
				\node () [] at ($(1)+(90:.2)$) {};
				\node () [] at ($(2)+(90:.3)$) {};
				\node () [] at ($(3)+(0:.2)$) {$g$};
				\node () [] at ($(4)+(90:.22)$) {$2g$};

				\node () [] at ($(0')+(180:.2)$) {};
				\node () [] at ($(1')-(0.5,0)+(-90:.3)$) {$2g+r$};
				\node () [] at ($(2')+(180:.4)$) {$g+r$};
				\node () [] at ($(3')+(0:.2)$) {$0$};
				\node () [] at ($(4')+(-90:.3)$) {};

				\draw[line width=1.3pt,color=gray,->] (1) -- (0);
				\draw[line width=1.5pt,dotted,color=green,->] (2) -- (1);
				\draw[line width=1.5pt,dotted,color=green,->] (3) -- (4);
				\draw[line width=1.5pt,dashed,color=red,->] (4) -- (2);
				\draw[line width=1.3pt,M edge] (2) -- (3);

				\draw[line width=1.5pt,dashed,color=red,<-] (0') -- (1');
				\draw[line width=1.5pt,dashed,color=red,->] (4') -- (2');
				\draw[line width=1.5pt,dotted,color=green,->] (2') -- (1');
				\draw[line width=1.5pt,dotted,color=green,->] (3') -- (4');
				\draw[line width=1.3pt,M edge] (2') -- (3');

		\end{tikzpicture}}
		\caption{Identities given by Lemma~\ref{lemma:type12} when $b_3 = 0$.}
		\label{fig:type12}
	\end{figure}
	
	\subsection{Complete decompositions}
	
	The following definition consists of two properties that are invariant under a series of operations performed throughout the proof of Lemma~\ref{lemma:P5-decomposition}.
	
	\begin{defi}\label{def:complete-commutative}
	
		A decomposition \(\D\) of a \(\{g,r\}\)-graph \(G\) 
		into trails of length \(5\) 
		is \emph{complete} if the following hold for every \(T\in\D\).
		\begin{enumerate}[(a)]
			\item\label{def:complete-commutative-types}
			\(T\) is of type~A, B, C or D;
			\item\label{def:complete-commutative-hanging-edge}
			If \(T\) is of type~A, then \(\hang_\D\big(\cv_1(T)\big)\geq 2\) and \(\hang_\D\big(\cv_2(T)\big)\geq 1\).
		\end{enumerate}
	\end{defi}
	
	The first step of our proof is given by the next proposition,
	which presents an initial decomposition for the graphs studied.

	\begin{proposition}\label{proposition:initial-decomposition}
		If \(G\) is a $\{g,r\}$-graph for which $2g+2r\neq 0$,
		then \(G\) admits a complete decomposition.
	\end{proposition}
	
	\begin{proof}
		Let \(\{M_{g,r},F_g,F_r\}\) be the base factorization of \(G\).
		For each \(e=xy\in M_{g,r}\), let \(P_e=a_0a_1a_2a_3a_4a_5\),
		where \(a_1a_0,a_4a_5\in F_r\), \(a_2a_1,a_3a_4\in F_g\), \(a_2=x\), 
		and \(a_3 = y\).
		We claim that \(\D = \{P_e\colon e\in M_{g,r}\}\) is complete.
		Clearly, \(P_e\) is an element of type~A or B, for every \(e\in M_{g,r}\),
		and hence \(\D\) satisfies
		Definition~\ref{def:complete-commutative}\eqref{def:complete-commutative-types}.
		Moreover, note that \(a_0a_1\) (resp. \(a_4a_5\)) is a hanging edge 
		at \(a_1\) (resp. \(a_4\)).
		Thus, given \(z\in V(G)\), let \(e' = xy\in M_{g,r}\) be such that \(x= z-g\),
		then \(P_{e'}\) contains a hanging edge at \(z\),
		namely, the red out edge of \(z\),
		and hence there is a hanging edge at every vertex of \(G\).
		Moreover, if \(z=\cv_1(T)\) for some element \(T\in\D\) of type~A, and \(e\in M_{g,r}\cap E(T)\),
		then \(e\) is a second hanging edge at \(z\).
		This proves Definition~\ref{def:complete-commutative}\eqref{def:complete-commutative-hanging-edge}.
	\end{proof}

	We say that an element \(T\) of type~A in a decomposition \(\D\) is \emph{free}
	if \(\tr(T)\neq \cv_i(T')\) for every element \(T'\in\D\) of type~A and \(i\in\{1,2\}\).
	An \emph{A-chain} is a sequence \(T_0, T_1, \ldots, T_{s-1}\) 
	of elements of type~A such that 
	for each $j\in\{0,\ldots,s-1\}$, 
	we have \(\tr(T_j) = \cv_i(T_{j-1})\), for some $i \in\{1,2\}$ 
	(subtraction on the indexes are taken modulo \(s\)). 
	Note that A-chains do not consider the auxiliary vertex when allowing two elements to be consecutive.
	Thus, elements, say \(T\) and \(T'\), of type~A that are not consecutive in an A-chain, 
	or that are in different A-chains, may still share a vertex \(u\) for which \(\cv_i(T) = u = \aux(T')\).
	
	Given a decomposition $\D$ of a graph $G$ into trails of length $5$, denote by
	$\tau(\D)$ the number of elements that are not paths.
	By exchanging edges between the elements of a decomposition 
	given by Proposition~\ref{proposition:initial-decomposition},
	we can show that a complete decomposition that minimizes $\tau(\D)$
	has no free element, and hence its elements of type~A are partitioned 
	into A-chains.

	\begin{lemma}\label{lemma:P5-decomposition}
		Every \(\{g,r\}\)-graph for which $2g+2r\neq 0$ admits a complete decomposition
		in which the elements of type~A are partitioned into A-chains.
	\end{lemma}
	
	\begin{proof}
		Let $g$ and $r$ be as in the statement, let \(G\) be a \(\{g,r\}\)-graph, and put \({M=M_{g,r}}\).
		By Proposition~\ref{proposition:initial-decomposition},
		\(G\) admits a complete decomposition.
		Let \(\D\) be a complete decomposition of~\(G\) that minimizes \(\tau(\D)\).
		In what follows, we prove that \(\D\) contains no free element of type~A.
		For that, we prove three claims regarding the relation between some elements of~$\D$.
		In the proof of each such claim, we exchange edges between some elements of \(\D\) 
		and obtain a complete decomposition \(\D'\) into trails of length~5 
		such that \(\tau(\D') <\tau(\D)\), which is a contradiction to the minimality of \(\D\).
		To check that \(\D'\) is a complete decomposition, we observe the two following items:
		(i) the vertices $u$ for which $\hang_D(u) >\hang_{\D'}(u)$ are vertices 
		that are not connection vertices of \(\D'\),
		for example, tricky vertices of free elements of type~A,
		or connection vertices of elements of type~A in \(\D\) that become paths in \(\D'\).
		Hence, Definition~\ref{def:complete-commutative}\eqref{def:complete-commutative-hanging-edge} holds for \(\D'\); and
		(ii) every element of \(\D'\) that is not an element of \(\D\), i.e., 
		the elements involved in the exchange of edges, are of type~A, B, C, or~D,
		and hence \ref{def:complete-commutative}\eqref{def:complete-commutative-types} 
		holds for \(\D'\).

		\begin{claim} \label{claim:AB,AC}
			No element of type~B or C has a hanging edge at the primary connection vertex 
			of a free element of type~A.
		\end{claim}
		
		\begin{proof}
			Let \(T_1 = a_0a_1a_2a_3a_4a_5\in\D\), where \(a_5=a_2\) and \(a_2a_3\in M\), be a free element of type~A,
			and let \(T_2= b_0b_1b_2b_3b_4b_5\in\D\) be an element of type~B or C that contains a hanging edge at \(\cv_1(T_1)\).
			We divide the proof depending on whether $T_2$ is of type~B or C.
			
			\smallskip	\noindent
			\textbf{$\mathbf{T_2}$ is of type~B.}
			Suppose, for a contradiction, that \(b_4=\cv_1(T_1)=a_3\).			
			Put \({T_1' = a_0a_1a_2a_4a_3b_5}\),
			\(T_2' = b_0b_1b_2b_3b_4a_2\) (see Figure~\ref{fig:case1}),
			and let \(\D' = \big(\D\setminus\{T_1,T_2\}\big)\cup\{T_1',T_2'\}\).
			Note that \(\D'\) is a decomposition of \(G\) into trails of length~\(5\). 
			By Lemma~\ref{lemma:T'1-is-path}, \(T_1'\) is an element of
			Type C.
			In what follows, we prove that \(T_2'\) is of type~B,
			i.e., \(a_2\notin \{b_0,b_1,b_2,b_3,b_4\}\).
			Indeed, since \(G\) has no loops or multiple edges, 
			\(a_2\notin \{b_3,b_4\}\).
			Since \(M\) is a matching, \(a_2\neq b_2\). 
			If \(a_2 = b_1\), then \(b_2 = b_1-g = a_2 -g = a_3 + g +r -g = b_5\),
			and hence \(T_2\) is of type~A, a contradiction.
			Finally, by Lemma~\ref{lemma:no-cycles}, \(a_2 \neq b_0\).
			Thus, \(T'_2\) is an element of type~B, 
			and hence Definition~\ref{def:complete-commutative}\eqref{def:complete-commutative-types} holds for \(\D'\).
			Note that \(\hang_{\D'}(v) \geq \hang_{\D}(v)\) 
			for every \(v\in V(G)\setminus\{a_4\}\). 
			Since $a_4$ is not a connection vertex of \(\D'\), 
			and by Definition~\ref{def:complete-commutative}\eqref{def:complete-commutative-hanging-edge},
			we have \(\hang_{\D}\big(\cv_1(T)\big)\geq 2\) and \(\hang_{\D}\big(\cv_2(T)\big)\geq 1\)
			for every \(T\in\D\),
			we have 
			\(\hang_{\D'}\big(\cv_1(T)\big)\geq 2\) and \(\hang_{\D'}\big(\cv_2(T)\big)\geq 1\)
			for every \(T\in\D'\),
			Thus Definition~\ref{def:complete-commutative}\eqref{def:complete-commutative-hanging-edge} holds for \(\D'\). 
			Therefore, $\D'$ is a complete decomposition such that
			\(\tau(\D')=\tau(\D)-1<\tau(\D)\), a contradiction
			to the minimality of \(\D\).
			
			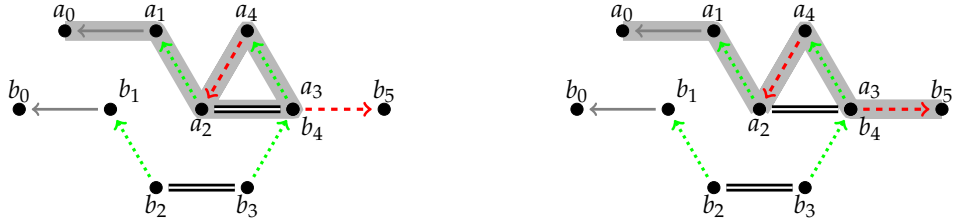
\begin{figure}[H]
				\centering
				\begin{subfigure}{.45\textwidth}
					\centering
					\scalebox{.8}{\begin{tikzpicture}[scale=1.5]
						
							\draw[fatpath,backcolor1] (120:1)-- ($(-1,0)+(120:1)$) -- (120:1) -- (0,0) -- (0:1) -- (60:1) -- (0,0) -- (60:1);
							
							\node (0) [black vertex] at ($(-1,0)+(120:1)$) {};
							\node (1) [black vertex] at (120:1) {};
							\node (2) [black vertex] at (0,0) {};
							\node (3) [black vertex] at (0:1) {};
							\node (4) [black vertex] at (60:1) {};
							\node () []		 at ($(2,0)+(60:1)$) {};

							\node (4') [black vertex] at (1,0) {};
							\node (3') [black vertex] at ($(4')-(60:1)$) {};	
							\node (2') [black vertex] at ($(3')-(1,0)$) {};
							\node (1') [black vertex] at ($(2')+(120:1)$) {};
							\node (0') [black vertex] at ($(2')+(-1,0)+(120:1)$) {};
							\node (5') [black vertex] at ($(4')+(1,0)$) {};

							\node () [] at ($(0)+(90:.2)$) {$a_0$};
							\node () [] at ($(1)+(90:.2)$) {$a_1$};
							\node () [] at ($(2)+(-90:.2)$) {$a_2$};
							\node () [] at ($(3)+(45:.3)$) {$a_3$};
							\node () [] at ($(4)+(90:.2)$) {$a_4$};

							\node () [] at ($(0')+(90:.2)$) {$b_0$};
							\node () [] at ($(1')+(45:.3)$) {$b_1$};
							\node () [] at ($(2')+(-90:.2)$) {$b_2$};
							\node () [] at ($(3')+(-90:.2)$) {$b_3$};
							\node () [] at ($(4')+(-45:.3)$) {$b_4$};
							\node () [] at ($(5')+(90:.2)$) {$b_5$};

							\draw[line width=1.3pt,color=gray,<-] (0) -- (1);
							\draw[line width=1.5pt,dotted,color=green,->] (2) -- (1);
							\draw[line width=1.5pt,dotted,color=green,->] (3) -- (4);
							\draw[line width=1.5pt,dashed,color=red,->] (4) -- (2);
							\draw[line width=1.3pt,M edge] (2) -- (3);

							\draw[line width=1.3pt,color=gray,<-] (0') -- (1');
							\draw[line width=1.5pt,dashed,color=red,->] (4') -- (5');
							\draw[line width=1.5pt,dotted,color=green,->] (2') -- (1');
							\draw[line width=1.5pt,dotted,color=green,->] (3') -- (4');
							\draw[line width=1.3pt,M edge] (2') -- (3');
						\end{tikzpicture}
					}
				\end{subfigure}
				\begin{subfigure}{.45\textwidth}
					\centering
					\scalebox{.8}{\begin{tikzpicture}[scale=1.5]
							
							\draw[fatpath,backcolor1] (120:1) -- ($(-1,0)+(120:1)$) -- (120:1) -- (0,0)  (60:1) -- (0,0) -- (60:1) -- (0:1) -- ($(2,0)$) -- (0:1);
							
							\node (0) [black vertex] at ($(-1,0)+(120:1)$) {};
							\node (1) [black vertex] at (120:1) {};
							\node (2) [black vertex] at (0,0) {};
							\node (3) [black vertex] at (0:1) {};
							\node (4) [black vertex] at (60:1) {};
							\node () []		 at ($(2,0)+(60:1)$) {};

							\node (4') [black vertex] at (1,0) {};
							\node (3') [black vertex] at ($(4')-(60:1)$) {};	
							\node (2') [black vertex] at ($(3')-(1,0)$) {};
							\node (1') [black vertex] at ($(2')+(120:1)$) {};
							\node (0') [black vertex] at ($(2')+(-1,0)+(120:1)$) {};
							\node (5') [black vertex] at ($(4')+(1,0)$) {};

							\node () [] at ($(0)+(90:.2)$) {$a_0$};
							\node () [] at ($(1)+(90:.2)$) {$a_1$};
							\node () [] at ($(2)+(-90:.2)$) {$a_2$};
							\node () [] at ($(3)+(45:.3)$) {$a_3$};
							\node () [] at ($(4)+(90:.2)$) {$a_4$};

							\node () [] at ($(0')+(90:.2)$) {$b_0$};
							\node () [] at ($(1')+(45:.3)$) {$b_1$};
							\node () [] at ($(2')+(-90:.2)$) {$b_2$};
							\node () [] at ($(3')+(-90:.2)$) {$b_3$};
							\node () [] at ($(4')+(-45:.3)$) {$b_4$};
							\node () [] at ($(5')+(90:.2)$) {$b_5$};

							\draw[line width=1.3pt,color=gray,<-] (0) -- (1);
							\draw[line width=1.5pt,dotted,color=green,->] (2) -- (1);
							\draw[line width=1.5pt,dotted,color=green,->] (3) -- (4);
							\draw[line width=1.5pt,dashed,color=red,->] (4) -- (2);
							\draw[line width=1.3pt,M edge] (2) -- (3);

							\draw[line width=1.3pt,color=gray,<-] (0') -- (1');
							\draw[line width=1.5pt,dashed,color=red,->] (4') -- (5');
							\draw[line width=1.5pt,dotted,color=green,->] (2') -- (1');
							\draw[line width=1.5pt,dotted,color=green,->] (3') -- (4');
							\draw[line width=1.3pt,M edge] (2') -- (3');
						\end{tikzpicture}
					}
				\end{subfigure}
				\caption{Exchange of edges between elements of type~A
					and B in the proof of Claim~\ref{claim:AB,AC}.}
				\label{fig:case1}
			\end{figure}
			
			\smallskip \noindent
			\textbf{$\mathbf{T_2}$ is of type~C.}
			We may assume \(b_3b_2\in F_r\).
			In this case we have \(b_4b_3\in F_g\).
			Since \(T_2\) contains a hanging edge at \(\cv_1(T_1)\),
			we have \(a_3=\cv_1(T_1)\in\{b_1,b_4\}\).
			If $b_4=a_3$, then there are two green out edges at \(a_3\), namely \(a_3a_4, b_4a_3\), a contradiction.
			Thus, we may assume that \(a_3 = b_1\).
			Put \(T_1' = a_0a_1a_2a_4a_3b_0\),
			\(T_2' = a_2b_1b_2b_4b_3b_5\) (see Figure~\ref{fig:case2}),
			and let \(\D' = \big(\D\setminus\{T_1,T_2\}\big)\cup\{T_1',T_2'\}\).
			Note that \(\D'\) is a decomposition of \(G\) into trails of length~\(5\).
			By Lemma~\ref{lemma:T'1-is-path}, \(T_1'\) is an element of
			Type C. 
			In what follows we prove that \(T_2'\) is a path. For that,
			we prove that \(a_2\notin \{b_0,b_1,b_2,b_3,b_4\}\). Indeed,
			since \(G\) is simple, \({a_2\notin \{b_1,b_2\}}\).
			If \(a_2 = b_4\), then \(a_2a_1\) and \(b_4b_3\) are two green out edges at \(a_2\), a contradiction.
			By Lemma~\ref{lemma:no-cycles}, $a_2\neq b_5$. 
			Finally, \(a_3 + g + r = a_2\) and \({b_1 = b_3 + r + g}\),
			if \(a_2 = b_3\), then we have \(a_3 + g + r = a_2 = b_3 = b_1 - g - r = a_3 - g - r\),
			which implies \(2g+2r = 0\), a contradiction.
			Thus, \(T'_2\) is an element of type~C, 
			and hence Definition~\ref{def:complete-commutative}\eqref{def:complete-commutative-types} holds for \(\D'\).	
			Analogously to the case above 
			$\D'$ is a complete decomposition of \(G\) such that
			\(\tau(\D')=\tau(\D)-1<\tau(\D)\), a contradiction
			to the minimality of~\(\D\).		
		\end{proof}
		
		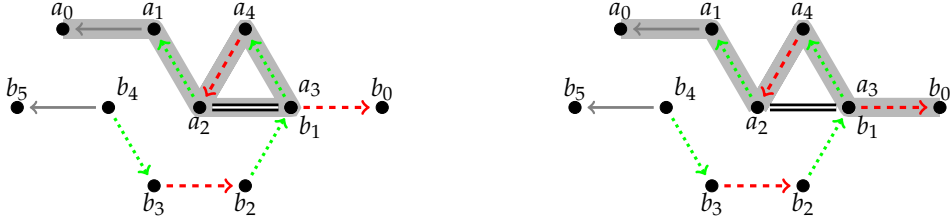
\begin{figure}[H]
			\centering
			\begin{subfigure}{.45\textwidth}
				\centering
				\scalebox{.8}{\begin{tikzpicture}[scale=1.5]
					
						\draw[fatpath,backcolor1] (120:1)-- ($(-1,0)+(120:1)$) -- (120:1) -- (0,0) -- (0:1) -- (60:1) -- (0,0) -- (60:1);
						
						\node (0) [black vertex] at ($(-1,0)+(120:1)$) {};
						\node (1) [black vertex] at (120:1) {};
						\node (2) [black vertex] at (0,0) {};
						\node (3) [black vertex] at (0:1) {};
						\node (4) [black vertex] at (60:1) {};
						\node () []		 at ($(2,0)+(60:1)$) {};

						\node (4') [black vertex] at (1,0) {};
						\node (3') [black vertex] at ($(4')-(60:1)$) {};	
						\node (2') [black vertex] at ($(3')-(1,0)$) {};
						\node (1') [black vertex] at ($(2')+(120:1)$) {};
						\node (0') [black vertex] at ($(2')+(-1,0)+(120:1)$) {};
						\node (5') [black vertex] at ($(4')+(1,0)$) {};

						\node () [] at ($(0)+(90:.2)$) {$a_0$};
						\node () [] at ($(1)+(90:.2)$) {$a_1$};
						\node () [] at ($(2)+(-90:.2)$) {$a_2$};
						\node () [] at ($(3)+(45:.3)$) {$a_3$};
						\node () [] at ($(4)+(90:.2)$) {$a_4$};

						\node () [] at ($(0')+(90:.2)$) {$b_5$};
						\node () [] at ($(1')+(45:.3)$) {$b_4$};
						\node () [] at ($(2')+(-90:.2)$) {$b_3$};
						\node () [] at ($(3')+(-90:.2)$) {$b_2$};
						\node () [] at ($(4')+(-45:.3)$) {$b_1$};
						\node () [] at ($(5')+(90:.2)$) {$b_0$};

						\draw[line width=1.3pt,color=gray,<-] (0) -- (1);
						\draw[line width=1.5pt,dotted,color=green,->] (2) -- (1);
						\draw[line width=1.5pt,dotted,color=green,->] (3) -- (4);
						\draw[line width=1.5pt,dashed,color=red,->] (4) -- (2);
						\draw[line width=1.3pt,M edge] (2) -- (3);

						\draw[line width=1.3pt,color=gray,->] (1') -- (0');
						\draw[line width=1.5pt,dashed,color=red,->] (4') -- (5');
						\draw[line width=1.5pt,dotted,color=green,->] (1') -- (2');
						\draw[line width=1.5pt,dotted,color=green,->] (3') -- (4');
						\draw[line width=1.5pt,dashed,red,->] (2') -- (3');
					\end{tikzpicture}
				}
			\end{subfigure}
			\begin{subfigure}{.45\textwidth}
				\centering
				\scalebox{.8}{\begin{tikzpicture}[scale=1.5]
					
						\draw[fatpath,backcolor1] (120:1) -- ($(-1,0)+(120:1)$) -- (120:1) -- (0,0)  (60:1) -- (0,0) -- (60:1) -- (0:1) -- ($(2,0)$) -- (0:1);

						\node (0) [black vertex] at ($(-1,0)+(120:1)$) {};
						\node (1) [black vertex] at (120:1) {};
						\node (2) [black vertex] at (0,0) {};
						\node (3) [black vertex] at (0:1) {};
						\node (4) [black vertex] at (60:1) {};
						\node () []		 at ($(2,0)+(60:1)$) {};

						\node (4') [black vertex] at (1,0) {};
						\node (3') [black vertex] at ($(4')-(60:1)$) {};	
						\node (2') [black vertex] at ($(3')-(1,0)$) {};
						\node (1') [black vertex] at ($(2')+(120:1)$) {};
						\node (0') [black vertex] at ($(2')+(-1,0)+(120:1)$) {};
						\node (5') [black vertex] at ($(4')+(1,0)$) {};

						\node () [] at ($(0)+(90:.2)$) {$a_0$};
						\node () [] at ($(1)+(90:.2)$) {$a_1$};
						\node () [] at ($(2)+(-90:.2)$) {$a_2$};
						\node () [] at ($(3)+(45:.3)$) {$a_3$};
						\node () [] at ($(4)+(90:.2)$) {$a_4$};

						\node () [] at ($(0')+(90:.2)$) {$b_5$};
						\node () [] at ($(1')+(45:.3)$) {$b_4$};
						\node () [] at ($(2')+(-90:.2)$) {$b_3$};
						\node () [] at ($(3')+(-90:.2)$) {$b_2$};
						\node () [] at ($(4')+(-45:.3)$) {$b_1$};
						\node () [] at ($(5')+(90:.2)$) {$b_0$};

						\draw[line width=1.3pt,color=gray,<-] (0) -- (1);
						\draw[line width=1.5pt,dotted,color=green,->] (2) -- (1);
						\draw[line width=1.5pt,dotted,color=green,->] (3) -- (4);
						\draw[line width=1.5pt,dashed,color=red,->] (4) -- (2);
						\draw[line width=1.3pt,M edge] (2) -- (3);

						\draw[line width=1.3pt,color=gray,->] (1') -- (0');
						\draw[line width=1.5pt,dashed,color=red,->] (4') -- (5');
						\draw[line width=1.5pt,dotted,color=green,->] (1') -- (2');
						\draw[line width=1.5pt,dotted,color=green,->] (3') -- (4');
						\draw[line width=1.5pt,dashed,red,->] (2') -- (3');
					\end{tikzpicture}
				}
			\end{subfigure}
			\caption{Exchange of edges between elements of type~A
				and C in the proof of Claim~\ref{claim:AB,AC}.}
			\label{fig:case2}
		\end{figure}

		\begin{claim}\label{claim:AAB,AAC}
			Let \(T_1\) and \(T_2\) be two elements of type~A in \(\D\).
			If \(T_1\) is free and \(T_2\) contains a hanging 
			edge on \(\cv_1(T_1)\), then no element of type~A, B, or C in \(\D\setminus\{T_1,T_2\}\) 
			contains a hanging edge at \(\cv_2(T_2)\).	
		\end{claim}  
		
		\begin{proof}
			Let \(T_1 = a_0a_1a_2a_3a_4a_5\) and \(T_2 = b_0b_1b_2b_3b_4b_5\) be two elements of \(\D\), 
			where \({a_5=a_2}\) and \(b_5=b_2\) and \(a_2a_3,b_2b_3\in M\).
			First, we prove that \(\cv_1(T_1)=\tr(T_2)\),
			and hence, by Lemma~\ref{lemma:type12}, we have \(\cv_2(T_1) = \aux(T_2)\).
			Suppose, for contradiction, that \(\cv_1(T_1)\neq \tr(T_2)\). 
			Since \(b_2b_3\in M\), we must have $b_1=\cv_1(T_1)=a_3$. 
			Now, put \({T_1' = a_0a_1a_2a_4a_3b_0}\),
			\(T_2' = a_2b_1b_2b_3b_4b_2\) (see Figure~\ref{fig:case3}) and let
			\(\D' = \big(\D\setminus\{T_1,T_2\}\big)\cup\{T_1',T_2'\}\).
			By Lemma~\ref{lemma:T'1-is-path}, \(T_1'\) is an element of type~C.
			We claim that \(T_2'\) is an element of type~A.
			For that we prove that \(a_2 \notin \{b_1,b_2,b_3,b_4\}\).
			Again, since \(G\) is a simple graph, we have \(a_2\notin\{b_1,b_2\}\).
			Since every vertex is incident to precisely one edge of \(M\), we have \(a_2\neq b_3\),
			and if \(a_2 = b_4\), then we have \(a_3 + g + r = a_2 = b_4 = a_3 - g - r\), which implies \(2g+2r=0\), a contradiction.
			Thus, Definition~\ref{def:complete-commutative}\eqref{def:complete-commutative-types} holds for \(\D'\).
			Analogously to the cases above, 
			$\D'$ is a complete decomposition such that
			\(\tau(\D')=\tau(\D)-1<\tau(\D)\), a contradiction
			to the minimality of \(\D\).
			Finally, by Lemma~\ref{lemma:type12}, we have \(\cv_2(T_1) = \aux(T_2)\).
			
			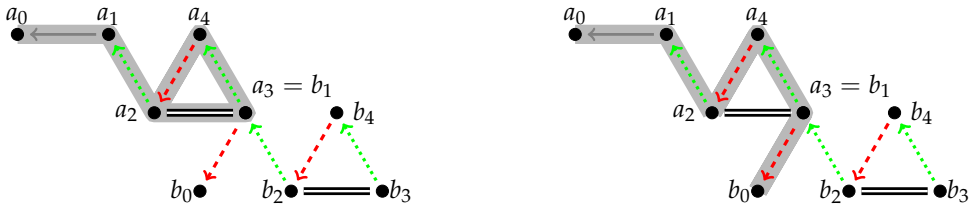
\begin{figure}[H]
				\centering
				\begin{subfigure}{.45\textwidth}
					\centering
					\scalebox{.8}{\begin{tikzpicture}[scale=1.5]
						
							\draw[fatpath,backcolor1] (120:1)-- ($(-1,0)+(120:1)$) -- (120:1) -- (0,0) -- (0:1) -- (60:1) -- (0,0) -- (60:1);
							
							\node (0) [black vertex] at ($(-1,0)+(120:1)$) {};
							\node (1) [black vertex] at (120:1) {};
							\node (2) [black vertex] at (0,0) {};
							\node (3) [black vertex] at (0:1) {};
							\node (4) [black vertex] at (60:1) {};
							\node () []		 at ($(2,0)+(60:1)$) {};

							\node (4') [black vertex] at ($(2,0)$) {};
							\node (3') [black vertex] at ($(3,0)-(60:1)$) {};	
							\node (2') [black vertex] at ($(3')-(1,0)$) {};
							\node (1') [black vertex] at ($(2')+(120:1)$) {};
							\node (0') [black vertex] at ($(2')+(180:1)$) {};
							\node (5') [] at ($(4,0)$) {};

							\node () [] at ($(0)+(90:.2)$) {$a_0$};
							\node () [] at ($(1)+(90:.2)$) {$a_1$};
							\node () [] at ($(2)+(180:.3)$) {$a_2$};
							\node () [] at ($(3)+(30:.6)$) {$a_3=b_1$};
							\node () [] at ($(4)+(90:.2)$) {$a_4$};

							\node () [] at ($(0')+(180:.2)$) {$b_0$};
							\node () [] at ($(1')+(-90:.3)$) {};
							\node () [] at ($(2')+(180:.2)$) {$b_2$};
							\node () [] at ($(3')+(0:.2)$) {$b_3$};
							\node () [] at ($(4')+(0:.3)$) {$b_4$};

							\draw[line width=1.3pt,color=gray,<-] (0) -- (1);
							\draw[line width=1.5pt,dotted,color=green,->] (2) -- (1);
							\draw[line width=1.5pt,dotted,color=green,->] (3) -- (4);
							\draw[line width=1.5pt,dashed,color=red,->] (4) -- (2);
							\draw[line width=1.3pt,M edge] (2) -- (3);

							\draw[line width=1.5pt,dashed,color=red,<-] (0') -- (1');
							\draw[line width=1.5pt,dashed,color=red,->] (4') -- (2');
							\draw[line width=1.5pt,dotted,color=green,->] (2') -- (1');
							\draw[line width=1.5pt,dotted,color=green,->] (3') -- (4');
							\draw[line width=1.3pt,M edge] (2') -- (3');
					\end{tikzpicture}}
				\end{subfigure}
				\begin{subfigure}{.45\textwidth}
					\centering
					\scalebox{.8}{\begin{tikzpicture}[scale=1.5]
						
							\draw[fatpath,backcolor1] (120:1) -- ($(-1,0)+(120:1)$) -- (120:1) -- (0,0)  (60:1) -- (0,0) -- (60:1) -- (0:1) -- ($(2,0)-(60:1)-(0:1)$) -- (0:1);
						
							\node (0) [black vertex] at ($(-1,0)+(120:1)$) {};
							\node (1) [black vertex] at (120:1) {};
							\node (2) [black vertex] at (0,0) {};
							\node (3) [black vertex] at (0:1) {};
							\node (4) [black vertex] at (60:1) {};
							\node () []		 at ($(2,0)+(60:1)$) {};

							\node (4') [black vertex] at ($(2,0)$) {};
							\node (3') [black vertex] at ($(3,0)-(60:1)$) {};	
							\node (2') [black vertex] at ($(3')-(1,0)$) {};
							\node (1') [black vertex] at ($(2')+(120:1)$) {};
							\node (0') [black vertex] at ($(2')+(180:1)$) {};
							\node (5') [] at ($(4,0)$) {};

							\node () [] at ($(0)+(90:.2)$) {$a_0$};
							\node () [] at ($(1)+(90:.2)$) {$a_1$};
							\node () [] at ($(2)+(180:.3)$) {$a_2$};
							\node () [] at ($(3)+(30:.6)$) {$a_3=b_1$};
							\node () [] at ($(4)+(90:.2)$) {$a_4$};

							\node () [] at ($(0')+(180:.2)$) {$b_0$};
							\node () [] at ($(1')+(-90:.3)$) {};
							\node () [] at ($(2')+(180:.2)$) {$b_2$};
							\node () [] at ($(3')+(0:.2)$) {$b_3$};
							\node () [] at ($(4')+(0:.3)$) {$b_4$};

							\draw[line width=1.3pt,color=gray,<-] (0) -- (1);
							\draw[line width=1.5pt,dotted,color=green,->] (2) -- (1);
							\draw[line width=1.5pt,dotted,color=green,->] (3) -- (4);
							\draw[line width=1.5pt,dashed,color=red,->] (4) -- (2);
							\draw[line width=1.3pt,M edge] (2) -- (3);

							\draw[line width=1.5pt,dashed,color=red,<-] (0') -- (1');
							\draw[line width=1.5pt,dashed,color=red,->] (4') -- (2');
							\draw[line width=1.5pt,dotted,color=green,->] (2') -- (1');
							\draw[line width=1.5pt,dotted,color=green,->] (3') -- (4');
							\draw[line width=1.3pt,M edge] (2') -- (3');
					\end{tikzpicture}}
				\end{subfigure}
				\caption{Exchange of edges between two elements of type~A
					in the proof of Claim~\ref{claim:AAB,AAC}.}
				\label{fig:case3}
			\end{figure}

			Now, let \(T_3\in\D\setminus\{T_1,T_2\}\) be an element of type~A, B, or C,
			and suppose, for a contradiction, that \(T_3\) 
			contains a hanging edge at \(\cv_2(T_2)\). 
			Since \(\cv_1(T_1)=\tr(T_2)\) and \(\cv_2(T_1) = \aux(T_2)\), 
			we have \(a_3=b_4\), \(a_5=a_2=b_1\) and \(b_5=b_2\).
			In what follows, we divide the proof according to the type of \(T_3\).
			
			\smallskip	\noindent
			\textbf{$\mathbf{T_3}$ is of type~A.}
			Let \(T_3 = c_0c_1c_2c_3c_4c_5\), where \(c_2=c_5\) and \(c_2c_3\in M\). 
			Since each vertex is incident to precisely one edge of \(M\)
			we have \(c_3\neq b_2 = \cv_2(T_2)\).
			Therefore, we have \(\cv_2(T_2)\in\{c_1,c_4\}\).
			Suppose that \(\cv_2(T_2)=c_4\).
			Thus, we have \(b_5 = b_2 = c_4\).
			Put \(T_1' = a_0a_1a_2a_4a_3b_5\),
			\(T_2' = b_0b_1b_4b_3b_2c_2\),
			\(T_3' = b_1c_4c_3c_2c_1c_0\)  
			(see Figure~\ref{fig:case5}),
			and let
			\(\D' = \big(\D\setminus\{T_1,T_2,T_3\}\big)\cup\{T_1',T_2',T_3'\}\).
			In what follows, we prove that \(T_1'\), \(T_2'\) and \(T_3'\) are paths.
			By Lemma~\ref{lemma:T'1-is-path}, \(T_1'\) is an element of type~C.
			Since \(G\) is simple, we have \(c_2\notin \{b_1,b_2,b_3,b_4\}\) and \(b_1\notin \{c_2,c_3,c_4\}\).
			By Lemma~\ref{lemma:no-cycles}, \(c_2 \neq b_0,
			~b_1 \neq c_0\). Therefore, $T_2'$ is an element of type~D.						
			If \(b_1 = c_1\), then \(b_2b_1\) and \(c_2c_1\) are two green in edges at \(c_1\), a contradiction.
			Thus, \(T'_3\) is an element of type~B, 
			and hence Definition~\ref{def:complete-commutative}\eqref{def:complete-commutative-types} holds for \(\D'\).
			Analogously to the cases above, we have \(\hang_{\D'}(v) \geq \hang_{\D}(v)\geq 0\) 
			for every \(v\in V(G)\setminus\{a_3,a_4,b_3,c_3\}\).
			Since $a_3,a_4,b_3,c_3$ are not connection vertices in $\D'$,
			Definition~\ref{def:complete-commutative}\eqref{def:complete-commutative-hanging-edge} holds for \(\D'\). 
			Thus, $\D'$ is a complete decomposition such that
			\(\tau(\D')=\tau(\D)-3<\tau(\D)\), a contradiction to the minimality of~\(\D\).
			
			\begin{figure}
				\centering
				\begin{subfigure}{.45\textwidth}
					\centering
					\scalebox{.8}{\begin{tikzpicture}[scale=1.5]
						
							\draw[fatpath,backcolor1] (-0.5,0.8) -- (-1.5,0.8) -- (-0.5,0.8) -- (0,0) -- (0,0) -- (1,0) -- (0.5,0.8) -- (0,0);
							\draw[fatpath,backcolor2] (-0.5,-0.8) -- (-1.5,-0.8)-- (-0.5,-0.8) --(0,-1.6) -- (1,-1.6) -- (0.5,-0.8) -- (0,-1.6);
						
							\node (0) [black vertex] at (-1.5,0.8) {};
							\node (1) [black vertex] at (-0.5,0.8) {};
							\node (2) [black vertex] at (0,0) {};
							\node (3) [black vertex] at (1,0) {};
							\node (4) [black vertex] at (0.5,0.8) {};
							\node () []		 at ($(0,0)+(60:1)$) {};

							\node (4') [black vertex] at ($(1,0)$) {};
							\node (3') [black vertex] at ($(3)+(0.5,-0.8)$) {};	
							\node (2') [black vertex] at ($(2)+(0.5,-0.8)$) {};
							\node (1') [black vertex] at ($(0,0)$) {};
							\node (0') [black vertex] at ($(2')-(1,0)$) {};
							\node (5') [] at ($(2,0)$) {};

							\node (4'') [black vertex] at ($(2')$) {};
							\node (3'') [black vertex] at ($(3)-(0,1.6)$) {};	
							\node (2'') [black vertex] at ($(2)-(0,1.6)$) {};
							\node (1'') [black vertex] at ($(4'')-(1,0)$) {};
							\node (0'') [black vertex] at ($(1'')+(-1,0)$) {};
							\node (5'') [] at ($(2,0)$) {};

							\node () [] at ($(0)+(180:.2)$) {$a_0$};
							\node () [] at ($(1)+(90:.2)$) {$a_1$};
							\node () [] at ($(2)+(90:.3)$) {};
							\node () [] at ($(3)+(0:.6)$) {$a_3=b_4$};
							\node () [] at ($(4)+(90:.2)$) {$a_4$};

							\node () [] at ($(0')+(-90:.3)-(0.35,0)$) {$b_0=c_1$};
							\node () [] at ($(1')+(180:.4)-(0.3,0)$) {$b_1=a_2$};
							\node () [] at ($(2')+(-90:.25)+(0.7,0)$) {$b_2=c_4$};
							\node () [] at ($(3')+(0:.3)$) {$b_3$};
							\node () [] at ($(4')+(-90:.3)$) {};

							\node () [] at ($(0'')+(180:.2)$) {$c_0$};
							\node () [] at ($(1'')+(225:.4)$) {};
							\node () [] at ($(2'')+(180:.2)$) {$c_2$};
							\node () [] at ($(3'')+(0:.3)$) {$c_3$};
							\node () [] at ($(4'')+(90:.3)$) {};

							\draw[line width=1.3pt,color=gray,<-] (0) -- (1);
							\draw[line width=1.5pt,dotted,color=green,->] (2) -- (1);
							\draw[line width=1.5pt,dotted,color=green,->] (3) -- (4);
							\draw[line width=1.5pt,dashed,color=red,->] (4) -- (2);
							\draw[line width=1.3pt,M edge] (2) -- (3);

							\draw[line width=1.5pt,dashed,color=red,<-] (0') -- (1');
							\draw[line width=1.5pt,dashed,color=red,->] (4') -- (2');
							\draw[line width=1.5pt,dotted,color=green,->] (2') -- (1');
							\draw[line width=1.5pt,dotted,color=green,->] (3') -- (4');
							\draw[line width=1.3pt,M edge] (2') -- (3');

							\draw[line width=1.3pt,color=gray,<-] (0'') -- (1'');
							\draw[line width=1.5pt,dashed,color=red,->] (4'') -- (2'');
							\draw[line width=1.5pt,dotted,color=green,->] (2'') -- (1'');
							\draw[line width=1.5pt,dotted,color=green,->] (3'') -- (4'');
							\draw[line width=1.3pt,M edge] (2'') -- (3'');
					\end{tikzpicture}}
				\end{subfigure}
				\begin{subfigure}{.45\textwidth}
					\centering
					\scalebox{.8}{\begin{tikzpicture}[scale=1.5]
						
							\draw[fatpath,backcolor1] (-0.5,0.8) -- (-1.5,0.8) -- (-0.5,0.8) -- (0,0) -- (0.5,0.8) -- (1,0) -- (0.5,-0.8);
							\draw[fatpath,backcolor2] (-0.5,-0.8) -- (-1.5,-0.8)-- (-0.5,-0.8) --(0,-1.6) -- (1,-1.6) -- (0.5,-0.8) -- (0,0);
						
							\node (0) [black vertex] at (-1.5,0.8) {};
							\node (1) [black vertex] at (-0.5,0.8) {};
							\node (2) [black vertex] at (0,0) {};
							\node (3) [black vertex] at (1,0) {};
							\node (4) [black vertex] at (0.5,0.8) {};
							\node () []		 at ($(0,0)+(60:1)$) {};

							\node (4') [black vertex] at ($(1,0)$) {};
							\node (3') [black vertex] at ($(3)+(0.5,-0.8)$) {};	
							\node (2') [black vertex] at ($(2)+(0.5,-0.8)$) {};
							\node (1') [black vertex] at ($(0,0)$) {};
							\node (0') [black vertex] at ($(2')-(1,0)$) {};
							\node (5') [] at ($(2,0)$) {};

							\node (4'') [black vertex] at ($(2')$) {};
							\node (3'') [black vertex] at ($(3)-(0,1.6)$) {};	
							\node (2'') [black vertex] at ($(2)-(0,1.6)$) {};
							\node (1'') [black vertex] at ($(4'')-(1,0)$) {};
							\node (0'') [black vertex] at ($(1'')+(-1,0)$) {};
							\node (5'') [] at ($(2,0)$) {};

							\node () [] at ($(0)+(180:.2)$) {$a_0$};
							\node () [] at ($(1)+(90:.2)$) {$a_1$};
							\node () [] at ($(2)+(90:.3)$) {};
							\node () [] at ($(3)+(0:.6)$) {$a_3=b_4$};
							\node () [] at ($(4)+(90:.2)$) {$a_4$};

							\node () [] at ($(0')+(-90:.3)-(0.35,0)$) {$b_0=c_1$};
							\node () [] at ($(1')+(180:.4)-(0.3,0)$) {$b_1=a_2$};
							\node () [] at ($(2')+(-90:.25)+(0.7,0)$) {$b_2=c_4$};
							\node () [] at ($(3')+(0:.3)$) {$b_3$};
							\node () [] at ($(4')+(-90:.3)$) {};

							\node () [] at ($(0'')+(180:.2)$) {$c_0$};
							\node () [] at ($(1'')+(225:.4)$) {};
							\node () [] at ($(2'')+(180:.2)$) {$c_2$};
							\node () [] at ($(3'')+(0:.3)$) {$c_3$};
							\node () [] at ($(4'')+(90:.3)$) {};

							\draw[line width=1.3pt,color=gray,<-] (0) -- (1);
							\draw[line width=1.5pt,dotted,color=green,->] (2) -- (1);
							\draw[line width=1.5pt,dotted,color=green,->] (3) -- (4);
							\draw[line width=1.5pt,dashed,color=red,->] (4) -- (2);
							\draw[line width=1.3pt,M edge] (2) -- (3);

							\draw[line width=1.5pt,dashed,color=red,<-] (0') -- (1');
							\draw[line width=1.5pt,dashed,color=red,->] (4') -- (2');
							\draw[line width=1.5pt,dotted,color=green,->] (2') -- (1');
							\draw[line width=1.5pt,dotted,color=green,->] (3') -- (4');
							\draw[line width=1.3pt,M edge] (2') -- (3');

							\draw[line width=1.3pt,color=gray,<-] (0'') -- (1'');
							\draw[line width=1.5pt,dashed,color=red,->] (4'') -- (2'');
							\draw[line width=1.5pt,dotted,color=green,->] (2'') -- (1'');
							\draw[line width=1.5pt,dotted,color=green,->] (3'') -- (4'');
							\draw[line width=1.3pt,M edge] (2'') -- (3'');
					\end{tikzpicture}}
				\end{subfigure}
				\caption{Exchange of edges between three elements of type~A
					in the proof of Claim~\ref{claim:AAB,AAC}.}
				\label{fig:case5}
			\end{figure}
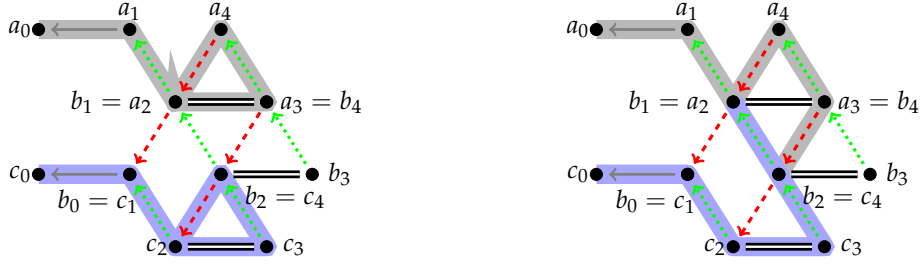
			
			Thus, we may assume \(\cv_2(T_2) = c_1\).
			This implies that \(b_5=b_2=c_1\),
			and hence we have \(b_3 = b_5 - r - g = c_1 - g-r = c_4\).
			Put \(T_1' = a_0a_1a_2a_4a_3b_2\),
			\(T_2' = b_0b_1b_4b_3b_2c_0\),
			\(T_3' = b_1c_1c_2c_3c_4c_2\) (see Figure~\ref{fig:case4}) and
			\(\D' = \big(\D\setminus\{T_1,T_2,T_3\}\big)\cup\{T_1',T_2',T_3'\}\).
			Again, by Lemma~\ref{lemma:T'1-is-path}, \(T_1'\) is an element of type~C.
			We claim that  \(T_2'\), \(T_3'\)  are, respectively, of type~D and A.
			Since $G$ is simple,
			\(c_0\notin \{b_1,b_2,b_3,b_4\}\) and
			\(b_1\notin \{c_1,c_2,c_4\}\).
			By Lemma~\ref{lemma:no-cycles} we have \(c_0 \neq b_0\). 
			Therefore, $T_2'$ is  of type~D.	
			Finally, if $b_1 = c_3$, then $d(b_1)\geq 7>5$, 
			a contradiction.					
			Thus, \(T'_3\) is an element of type~A, 
			and hence Definition~\ref{def:complete-commutative}\eqref{def:complete-commutative-types} holds for \(\D'\).
			Analogously to the case above, 
			we have $\hang_{\D'}(v) \geq \hang_{\D}(v)\geq 0$ for every $v \in V(G)\setminus\{a_4,a_3,b_3\}$.
			Since $a_4,a_3,b_3$ are not connection vertices in $\D'$, 
			Definition~\ref{def:complete-commutative}\eqref{def:complete-commutative-hanging-edge} holds for \(\D'\).
			Thus, \(\D'\) is a complete decomposition such that \({\tau(\D') = \tau(\D) -2<\tau(\D)}\), a contradiction to the minimality of~\(\D\).

			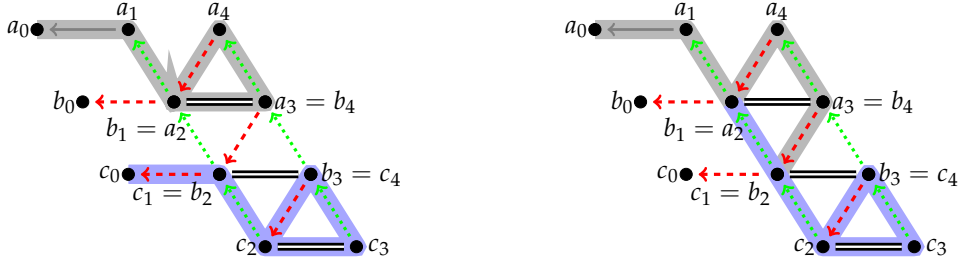
\begin{figure}
				\centering
				\begin{subfigure}{.45\textwidth}
					\centering
					\scalebox{.8}{\begin{tikzpicture}[scale=1.5]
						
							\draw[fatpath,backcolor1] (-0.5,0.8) -- (-1.5,0.8) -- (-0.5,0.8) -- (0,0) -- (0,0) -- (1,0) -- (0.5,0.8) -- (0,0);
							\draw[fatpath,backcolor2] (0.5,-0.8) -- (-0.5,-0.8)-- (0.5,-0.8) --(1,-1.6) -- (2,-1.6) -- (1.5,-0.8) -- (1,-1.6);

							\node (0) [black vertex] at (-1.5,0.8) {};
							\node (1) [black vertex] at (-0.5,0.8) {};
							\node (2) [black vertex] at (0,0) {};
							\node (3) [black vertex] at (1,0) {};
							\node (4) [black vertex] at (0.5,0.8) {};
							\node () []		 at ($(1,0)+(60:1)$) {};

							\node (4') [black vertex] at ($(1,0)$) {};
							\node (3') [black vertex] at ($(3)+(0.5,-0.8)$) {};	
							\node (2') [black vertex] at ($(2)+(0.5,-0.8)$) {};
							\node (1') [black vertex] at ($(0,0)$) {};
							\node (0') [black vertex] at ($(0,0)+(-1,0)$) {};
							\node (5') [] at ($(3,0)$) {};

							\node (4'') [black vertex] at ($(3')$) {};
							\node (3'') [black vertex] at ($(3)+(1,-1.6)$) {};	
							\node (2'') [black vertex] at ($(2)+(1,-1.6)$) {};
							\node (1'') [black vertex] at ($(2')$) {};
							\node (0'') [black vertex] at ($(1'')+(-1,0)$) {};
							\node (5'') [] at ($(3,0)$) {};

							\node () [] at ($(0)+(180:.2)$) {$a_0$};
							\node () [] at ($(1)+(90:.2)$) {$a_1$};
							\node () [] at ($(2)+(90:.3)$) {};
							\node () [] at ($(3)+(0:.55)$) {$a_3=b_4$};
							\node () [] at ($(4)+(90:.2)$) {$a_4$};

							\node () [] at ($(0')+(180:.2)$) {$b_0$};
							\node () [] at ($(1')+(-90:.3)-(0.3,0)$) {$b_1=a_2$};
							\node () [] at ($(2')+(90:.35)$) {};
							\node () [] at ($(3')+(0:.55)$) {$b_3=c_4$};
							\node () [] at ($(4')+(-90:.3)$) {};

							\node () [] at ($(0'')+(180:.2)$) {$c_0$};
							\node () [] at ($(1'')+(225:.3)-(0.3,0)$) {$c_1=b_2$};
							\node () [] at ($(2'')+(180:.2)$) {$c_2$};
							\node () [] at ($(3'')+(0:.25)$) {$c_3$};
							\node () [] at ($(4'')+(270:.35)$) {};

							\draw[line width=1.3pt,color=gray,<-] (0) -- (1);
							\draw[line width=1.5pt,dotted,color=green,->] (2) -- (1);
							\draw[line width=1.5pt,dotted,color=green,->] (3) -- (4);
							\draw[line width=1.5pt,dashed,color=red,->] (4) -- (2);
							\draw[line width=1.3pt,M edge] (2) -- (3);

							\draw[line width=1.5pt,dashed,color=red,<-] (0') -- (1');
							\draw[line width=1.5pt,dashed,color=red,->] (4') -- (2');
							\draw[line width=1.5pt,dotted,color=green,->] (2') -- (1');
							\draw[line width=1.5pt,dotted,color=green,->] (3') -- (4');
							\draw[line width=1.3pt,M edge] (2') -- (3');

							\draw[line width=1.5pt,dashed,color=red,<-] (0'') -- (1'');
							\draw[line width=1.5pt,dashed,color=red,->] (4'') -- (2'');
							\draw[line width=1.5pt,dotted,color=green,->] (2'') -- (1'');
							\draw[line width=1.5pt,dotted,color=green,->] (3'') -- (4'');
							\draw[line width=1.3pt,M edge] (2'') -- (3'');
					\end{tikzpicture}}
				\end{subfigure}
				\begin{subfigure}{.45\textwidth}
					\centering
					\scalebox{.8}{\begin{tikzpicture}[scale=1.5]
						
							\draw[fatpath,backcolor1] (-0.5,0.8) -- (-1.5,0.8) -- (-0.5,0.8) -- (0,0) -- (0.5,0.8) -- (1,0) -- (0.5,-0.8);
							\draw[fatpath,backcolor2] (0,0)-- (0.5,-0.8) --(1,-1.6) -- (2,-1.6) -- (1.5,-0.8) -- (1,-1.6);

							\node (0) [black vertex] at (-1.5,0.8) {};
							\node (1) [black vertex] at (-0.5,0.8) {};
							\node (2) [black vertex] at (0,0) {};
							\node (3) [black vertex] at (1,0) {};
							\node (4) [black vertex] at (0.5,0.8) {};
							\node () []		 at ($(1,0)+(60:1)$) {};

							\node (4') [black vertex] at ($(1,0)$) {};
							\node (3') [black vertex] at ($(3)+(0.5,-0.8)$) {};	
							\node (2') [black vertex] at ($(2)+(0.5,-0.8)$) {};
							\node (1') [black vertex] at ($(0,0)$) {};
							\node (0') [black vertex] at ($(0,0)+(-1,0)$) {};
							\node (5') [] at ($(3,0)$) {};

							\node (4'') [black vertex] at ($(3')$) {};
							\node (3'') [black vertex] at ($(3)+(1,-1.6)$) {};	
							\node (2'') [black vertex] at ($(2)+(1,-1.6)$) {};
							\node (1'') [black vertex] at ($(2')$) {};
							\node (0'') [black vertex] at ($(1'')+(-1,0)$) {};
							\node (5'') [] at ($(3,0)$) {};

							\node () [] at ($(0)+(180:.2)$) {$a_0$};
							\node () [] at ($(1)+(90:.2)$) {$a_1$};
							\node () [] at ($(2)+(90:.3)$) {};
							\node () [] at ($(3)+(0:.55)$) {$a_3=b_4$};
							\node () [] at ($(4)+(90:.2)$) {$a_4$};

							\node () [] at ($(0')+(180:.2)$) {$b_0$};
							\node () [] at ($(1')+(-90:.3)-(0.3,0)$) {$b_1=a_2$};
							\node () [] at ($(2')+(90:.35)$) {};
							\node () [] at ($(3')+(0:.55)$) {$b_3=c_4$};
							\node () [] at ($(4')+(-90:.3)$) {};

							\node () [] at ($(0'')+(180:.2)$) {$c_0$};
							\node () [] at ($(1'')+(225:.3)-(0.3,0)$) {$c_1=b_2$};
							\node () [] at ($(2'')+(180:.2)$) {$c_2$};
							\node () [] at ($(3'')+(0:.25)$) {$c_3$};
							\node () [] at ($(4'')+(270:.35)$) {};

							\draw[line width=1.3pt,color=gray,<-] (0) -- (1);
							\draw[line width=1.5pt,dotted,color=green,->] (2) -- (1);
							\draw[line width=1.5pt,dotted,color=green,->] (3) -- (4);
							\draw[line width=1.5pt,dashed,color=red,->] (4) -- (2);
							\draw[line width=1.3pt,M edge] (2) -- (3);

							\draw[line width=1.5pt,dashed,color=red,<-] (0') -- (1');
							\draw[line width=1.5pt,dashed,color=red,->] (4') -- (2');
							\draw[line width=1.5pt,dotted,color=green,->] (2') -- (1');
							\draw[line width=1.5pt,dotted,color=green,->] (3') -- (4');
							\draw[line width=1.3pt,M edge] (2') -- (3');

							\draw[line width=1.5pt,dashed,color=red,<-] (0'') -- (1'');
							\draw[line width=1.5pt,dashed,color=red,->] (4'') -- (2'');
							\draw[line width=1.5pt,dotted,color=green,->] (2'') -- (1'');
							\draw[line width=1.5pt,dotted,color=green,->] (3'') -- (4'');
							\draw[line width=1.3pt,M edge] (2'') -- (3'');
					\end{tikzpicture}}
				\end{subfigure}
				\caption{Exchange of edges between three elements of type~A
					in the proof of Claim~\ref{claim:AAB,AAC}.}
				\label{fig:case4}
			\end{figure}		
			
			\smallskip	\noindent
			\textbf{$\mathbf{T_3}$ is of type~B.}
			Let \(T_3 = c_0c_1c_2c_3c_4c_5\) be an element of type~B.
			Since \(T_3\) contains a hanging edge on \(\cv_2(T_2)=b_2\), we have \(b_2\in\{c_1,c_4\}\).
			By symmetry we may assume \(b_2=c_1\). 
			Thus, put \(T_1' = a_0a_1a_2a_4a_3b_2\),
			\(T_2' = b_0b_1b_4b_3b_2c_0\),
			\(T_3' = b_1c_1c_2c_3c_4c_5\)  (see
			Figure~\ref{fig:case6})
			and let
			\(\D' = \big(\D\setminus\{T_1,T_2,T_3\}\big)\cup\{T_1',T_2',T_3'\}\).
			Again, by Lemma~\ref{lemma:T'1-is-path}, \(T_1'\) is an element of type~C.
			We prove that \(T_2'\) and \(T_3'\) are, 
			respectively, of type~D and B.
			Since $G$ is simple, 
			we have \(c_0\notin \{b_1,b_2,b_3,b_4\}\) and \(b_1\notin \{c_1,c_2\}\).
			By Lemma~\ref{lemma:no-cycles}, we have \(c_0 \neq b_0\) and \(b_1 \neq c_5\). 
			Therefore, $T_2'$ is an element of type~D.			
			Since $c_4=c_3+g$ and $b_1=c_1+g$, if $c_4=b_1$, then $c_3=c_1$, a
			contradiction.
			If $b_1 \in \{c_3,c_4\}$, then $d(b_1)\geq7>5$, a contradiction.
			Therefore, $T_3'$ is an element of type~B.
			Analogously to the case above, \(\D'\) is a complete decomposition 
			such that \(\tau(\D') = \tau(\D) -2<\tau(\D)\), a contradiction to the minimality of \(\D\).

			\begin{figure}[h]
				\centering
				\begin{subfigure}{.45\textwidth}
					\centering
					\scalebox{.8}{\begin{tikzpicture}[scale=1.5]
							
							\draw[fatpath,backcolor1] (-0.5,0.8) -- (-1.5,0.8) -- (-0.5,0.8) -- (0,0) -- (0,0) -- (1,0) -- (0.5,0.8) -- (0,0);
							\draw[fatpath,backcolor2] (0.5,-0.8) -- (-0.5,-0.8)-- (0.5,-0.8) --(0.5,-1.8) -- (1.5,-1.8) -- (2.5,-1.8) -- (3.5,-1.8) -- (2.5,-1.8);

							\node (0) [black vertex] at (-1.5,0.8) {};
							\node (1) [black vertex] at (-0.5,0.8) {};
							\node (2) [black vertex] at (0,0) {};
							\node (3) [black vertex] at (1,0) {};
							\node (4) [black vertex] at (0.5,0.8) {};
							\node () []		 at ($(2,0)+(60:1)$) {};

							\node (4') [black vertex] at ($(1,0)$) {};
							\node (3') [black vertex] at ($(3)+(0.5,-0.8)$) {};	
							\node (2') [black vertex] at ($(2)+(0.5,-0.8)$) {};
							\node (1') [black vertex] at ($(0,0)$) {};
							\node (0') [black vertex] at ($(0,0)+(-1,0)$) {};
							\node (5') [] at ($(4,0)$) {};

							\node (1'') [black vertex] at ($(2')$) {};
							\node (0'') [black vertex] at ($(1'')+(-1,0)$) {};	
							\node (2'') [black vertex] at ($(1'')-(0,1)$) {};
							\node (3'') [black vertex] at ($(2'')+(1,0)$) {};
							\node (4'') [black vertex] at ($(3'')+(1,0)$) {};
							\node (5'') [black vertex] at ($(4'')+(1,0)$) {};

							\node () [] at ($(0)+(180:.2)$) {$a_0$};
							\node () [] at ($(1)+(90:.2)$) {$a_1$};
							\node () [] at ($(2)+(90:.3)$) {};
							\node () [] at ($(3)+(0:.55)$) {$a_3=b_4$};
							\node () [] at ($(4)+(90:.2)$) {$a_4$};

							\node () [] at ($(0')+(180:.2)$) {$b_0$};
							\node () [] at ($(1')+(-90:.25)-(0.4,0)$) {$b_1=a_2$};
							\node () [] at ($(2')+(90:.35)$) {};
							\node () [] at ($(3')+(0:.25)$) {$b_3$};
							\node () [] at ($(4')+(-90:.3)$) {};

							\node () [] at ($(0'')+(180:.2)$) {$c_0$};
							\node () [] at ($(1'')+(225:.3)+(0.85,0)$) {$c_1=b_2$};
							\node () [] at ($(2'')+(180:.25)$) {$c_2$};
							\node () [] at ($(3'')+(90:.25)$) {$c_3$};
							\node () [] at ($(4'')+(90:.25)$) {$c_4$};
							\node () [] at ($(5'')+(90:.25)$) {$c_5$};

							\draw[line width=1.3pt,color=gray,<-] (0) -- (1);
							\draw[line width=1.5pt,dotted,color=green,->] (2) -- (1);
							\draw[line width=1.5pt,dotted,color=green,->] (3) -- (4);
							\draw[line width=1.5pt,dashed,color=red,->] (4) -- (2);
							\draw[line width=1.3pt,M edge] (2) -- (3);

							\draw[line width=1.5pt,dashed,color=red,<-] (0') -- (1');
							\draw[line width=1.5pt,dashed,color=red,->] (4') -- (2');
							\draw[line width=1.5pt,dotted,color=green,->] (2') -- (1');
							\draw[line width=1.5pt,dotted,color=green,->] (3') -- (4');
							\draw[line width=1.3pt,M edge] (2') -- (3');

							\draw[line width=1.5pt,dashed,color=red,<-] (0'') -- (1'');
							\draw[line width=1.3pt,color=gray,->] (4'') -- (5'');
							\draw[line width=1.5pt,dotted,color=green,->] (2'') -- (1'');
							\draw[line width=1.5pt,dotted,color=green,->] (3'') -- (4'');
							\draw[line width=1.3pt,M edge] (2'') -- (3'');
					\end{tikzpicture}}
				\end{subfigure}
				\begin{subfigure}{.45\textwidth}
					\centering
					\scalebox{.8}{\begin{tikzpicture}[scale=1.5]
						
							\draw[fatpath,backcolor1] (-0.5,0.8) -- (-1.5,0.8) -- (-0.5,0.8) -- (0,0) -- (0.5,0.8) -- (1,0) -- (0.5,-0.8);
							\draw[fatpath,backcolor2] (0,0)-- (0.5,-0.8) --(0.5,-1.8) -- (1.5,-1.8) -- (2.5,-1.8) -- (3.5,-1.8) -- (2.5,-1.8);

							\node (0) [black vertex] at (-1.5,0.8) {};
							\node (1) [black vertex] at (-0.5,0.8) {};
							\node (2) [black vertex] at (0,0) {};
							\node (3) [black vertex] at (1,0) {};
							\node (4) [black vertex] at (0.5,0.8) {};
							\node () []		 at ($(2,0)+(60:1)$) {};

							\node (4') [black vertex] at ($(1,0)$) {};
							\node (3') [black vertex] at ($(3)+(0.5,-0.8)$) {};	
							\node (2') [black vertex] at ($(2)+(0.5,-0.8)$) {};
							\node (1') [black vertex] at ($(0,0)$) {};
							\node (0') [black vertex] at ($(0,0)+(-1,0)$) {};
							\node (5') [] at ($(4,0)$) {};

							\node (1'') [black vertex] at ($(2')$) {};
							\node (0'') [black vertex] at ($(1'')+(-1,0)$) {};	
							\node (2'') [black vertex] at ($(1'')-(0,1)$) {};
							\node (3'') [black vertex] at ($(2'')+(1,0)$) {};
							\node (4'') [black vertex] at ($(3'')+(1,0)$) {};
							\node (5'') [black vertex] at ($(4'')+(1,0)$) {};

							\node () [] at ($(0)+(180:.2)$) {$a_0$};
							\node () [] at ($(1)+(90:.2)$) {$a_1$};
							\node () [] at ($(2)+(90:.3)$) {};
							\node () [] at ($(3)+(0:.55)$) {$a_3=b_4$};
							\node () [] at ($(4)+(90:.2)$) {$a_4$};

							\node () [] at ($(0')+(180:.2)$) {$b_0$};
							\node () [] at ($(1')+(-90:.25)-(0.4,0)$) {$b_1=a_2$};
							\node () [] at ($(2')+(90:.35)$) {};
							\node () [] at ($(3')+(0:.25)$) {$b_3$};
							\node () [] at ($(4')+(-90:.3)$) {};

							\node () [] at ($(0'')+(180:.2)$) {$c_0$};
							\node () [] at ($(1'')+(225:.3)+(0.85,0)$) {$c_1=b_2$};
							\node () [] at ($(2'')+(180:.25)$) {$c_2$};
							\node () [] at ($(3'')+(90:.25)$) {$c_3$};
							\node () [] at ($(4'')+(90:.25)$) {$c_4$};
							\node () [] at ($(5'')+(90:.25)$) {$c_5$};

							\draw[line width=1.3pt,color=gray,<-] (0) -- (1);
							\draw[line width=1.5pt,dotted,color=green,->] (2) -- (1);
							\draw[line width=1.5pt,dotted,color=green,->] (3) -- (4);
							\draw[line width=1.5pt,dashed,color=red,->] (4) -- (2);
							\draw[line width=1.3pt,M edge] (2) -- (3);

							\draw[line width=1.5pt,dashed,color=red,<-] (0') -- (1');
							\draw[line width=1.5pt,dashed,color=red,->] (4') -- (2');
							\draw[line width=1.5pt,dotted,color=green,->] (2') -- (1');
							\draw[line width=1.5pt,dotted,color=green,->] (3') -- (4');
							\draw[line width=1.3pt,M edge] (2') -- (3');

							\draw[line width=1.5pt,dashed,color=red,<-] (0'') -- (1'');
							\draw[line width=1.3pt,color=gray,->] (4'') -- (5'');
							\draw[line width=1.5pt,dotted,color=green,->] (2'') -- (1'');
							\draw[line width=1.5pt,dotted,color=green,->] (3'') -- (4'');
							\draw[line width=1.3pt,M edge] (2'') -- (3'');
					\end{tikzpicture}}
				\end{subfigure}
				\caption{Exchange of edges between two elements of type~A
					and an element of type~B
					in the proof of Claim~\ref{claim:AAB,AAC}.
				}
				\label{fig:case6}
			\end{figure}
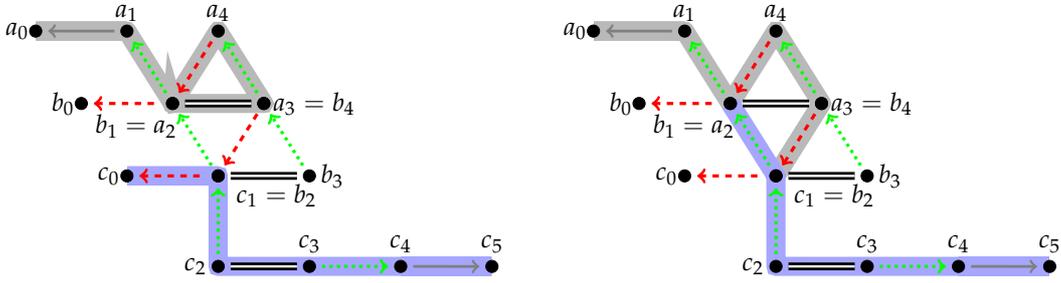
			
			\smallskip	\noindent
			\textbf{$\mathbf{T_3}$ is of type~C.}
			Let \(T_3 = c_0c_1c_2c_3c_4c_5\) be an element of type~C, where \(c_3c_2\in F_r\).
			This implies that \(c_4c_3\in F_g\).
			Since \(T_3\) contains a hanging edge on \(\cv_2(T_2)=b_2\), we have \(b_2\in\{c_1,c_4\}\).
			If \(b_2 = c_4\), then \(c_4c_3\) and \(b_2b_1\) are two green out edges of \(b_2\), a contradiction.
			Thus, we may assume \(b_2 = c_1\).
			Put \(T_1' = a_0a_1a_2a_4a_3b_2\),
			\(T_2' = b_0b_1b_4b_3b_2c_0\),
			\(T_3' = b_1c_1c_2c_3c_4c_5\) (see Figure~\ref{fig:case7})
			and let
			\(\D' = \big(\D\setminus\{T_1,T_2,T_3\}\big)\cup\{T_1',T_2',T_3'\}\).
			Again, by Lemma~\ref{lemma:T'1-is-path}, \(T_1'\) is an element of type~C.
			We prove that \(T_2'\) and \(T_3'\) are, respectively,
			of type~D and C.
			Since $G$ is simple, \(c_0\notin \{b_1,b_2,b_3,b_4\}\) and \(b_1\notin \{c_1,c_2\}\).
			By Lemma~\ref{lemma:no-cycles}, we have \(c_0 \neq b_0\) and \(b_1 \neq c_5\). 
			Therefore, $T_2'$ is an element of type~D.			
			Analogously to the case above, 
			If $b_1 \in \{c_3,c_4\}$, then $d(b_1)\geq7>5$, a contradiction.
			Therefore, $T_3'$ is an element of type~C.
			Once more, 
			analogously to the cases above, \(\D'\) is a complete decomposition 
			such that \(\tau(\D') = \tau(\D) -2<\tau(\D)\), a contradiction to the minimality of \(\D\).
		\end{proof}
		
			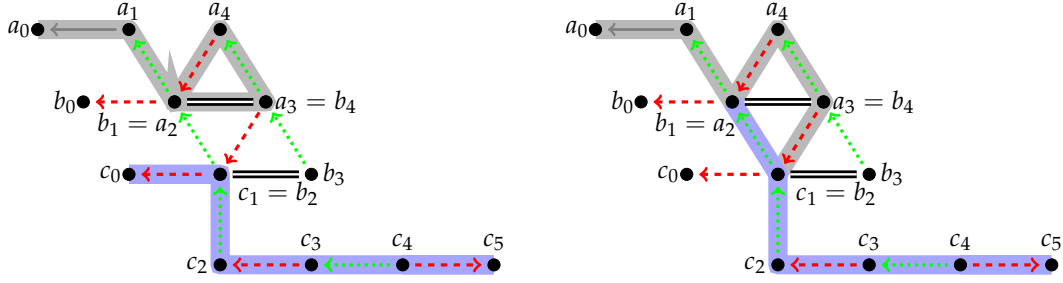
\begin{figure}
				\centering
				\begin{subfigure}{.45\textwidth}
					\centering
					\scalebox{.8}{\begin{tikzpicture}[scale=1.5]
						
							\draw[fatpath,backcolor1] (-0.5,0.8) -- (-1.5,0.8) -- (-0.5,0.8) -- (0,0) -- (0,0) -- (1,0) -- (0.5,0.8) -- (0,0);
							\draw[fatpath,backcolor2] (0.5,-0.8) -- (-0.5,-0.8)-- (0.5,-0.8) --(0.5,-1.8) -- (1.5,-1.8) -- (2.5,-1.8) -- (3.5,-1.8) -- (2.5,-1.8);

							\node (0) [black vertex] at (-1.5,0.8) {};
							\node (1) [black vertex] at (-0.5,0.8) {};
							\node (2) [black vertex] at (0,0) {};
							\node (3) [black vertex] at (1,0) {};
							\node (4) [black vertex] at (0.5,0.8) {};
							\node () []		 at ($(2,0)+(60:1)$) {};

							\node (4') [black vertex] at ($(1,0)$) {};
							\node (3') [black vertex] at ($(3)+(0.5,-0.8)$) {};	
							\node (2') [black vertex] at ($(2)+(0.5,-0.8)$) {};
							\node (1') [black vertex] at ($(0,0)$) {};
							\node (0') [black vertex] at ($(0,0)+(-1,0)$) {};
							\node (5') [] at ($(4,0)$) {};

							\node (1'') [black vertex] at ($(2')$) {};
							\node (0'') [black vertex] at ($(1'')+(-1,0)$) {};	
							\node (2'') [black vertex] at ($(1'')-(0,1)$) {};
							\node (3'') [black vertex] at ($(2'')+(1,0)$) {};
							\node (4'') [black vertex] at ($(3'')+(1,0)$) {};
							\node (5'') [black vertex] at ($(4'')+(1,0)$) {};

							\node () [] at ($(0)+(180:.2)$) {$a_0$};
							\node () [] at ($(1)+(90:.2)$) {$a_1$};
							\node () [] at ($(2)+(90:.3)$) {};
							\node () [] at ($(3)+(0:.55)$) {$a_3=b_4$};
							\node () [] at ($(4)+(90:.2)$) {$a_4$};

							\node () [] at ($(0')+(180:.2)$) {$b_0$};
							\node () [] at ($(1')+(-90:.25)-(0.4,0)$) {$b_1=a_2$};
							\node () [] at ($(2')+(90:.35)$) {};
							\node () [] at ($(3')+(0:.25)$) {$b_3$};
							\node () [] at ($(4')+(-90:.3)$) {};

							\node () [] at ($(0'')+(180:.2)$) {$c_0$};
							\node () [] at ($(1'')+(225:.3)+(0.85,0)$) {$c_1=b_2$};
							\node () [] at ($(2'')+(180:.25)$) {$c_2$};
							\node () [] at ($(3'')+(90:.25)$) {$c_3$};
							\node () [] at ($(4'')+(90:.25)$) {$c_4$};
							\node () [] at ($(5'')+(90:.25)$) {$c_5$};

							\draw[line width=1.3pt,color=gray,<-] (0) -- (1);
							\draw[line width=1.5pt,dotted,color=green,->] (2) -- (1);
							\draw[line width=1.5pt,dotted,color=green,->] (3) -- (4);
							\draw[line width=1.5pt,dashed,color=red,->] (4) -- (2);
							\draw[line width=1.3pt,M edge] (2) -- (3);

							\draw[line width=1.5pt,dashed,color=red,<-] (0') -- (1');
							\draw[line width=1.5pt,dashed,color=red,->] (4') -- (2');
							\draw[line width=1.5pt,dotted,color=green,->] (2') -- (1');
							\draw[line width=1.5pt,dotted,color=green,->] (3') -- (4');
							\draw[line width=1.3pt,M edge] (2') -- (3');

							\draw[line width=1.5pt,dashed,color=red,<-] (0'') -- (1'');
							\draw[line width=1.5pt,dashed,color=red,->] (4'') -- (5'');
							\draw[line width=1.5pt,dotted,color=green,->] (2'') -- (1'');
							\draw[line width=1.5pt,dotted,color=green,->] (4'') -- (3'');
							\draw[line width=1.5pt,dashed,color=red,->] (3'') -- (2'');
					\end{tikzpicture}}
				\end{subfigure}
				\begin{subfigure}{.45\textwidth}
					\centering
					\scalebox{.8}{\begin{tikzpicture}[scale=1.5]
						
							\draw[fatpath,backcolor1] (-0.5,0.8) -- (-1.5,0.8) -- (-0.5,0.8) -- (0,0) -- (0.5,0.8) -- (1,0) -- (0.5,-0.8);
							\draw[fatpath,backcolor2] (0,0)-- (0.5,-0.8) --(0.5,-1.8) -- (1.5,-1.8) -- (2.5,-1.8) -- (3.5,-1.8) -- (2.5,-1.8);

							\node (0) [black vertex] at (-1.5,0.8) {};
							\node (1) [black vertex] at (-0.5,0.8) {};
							\node (2) [black vertex] at (0,0) {};
							\node (3) [black vertex] at (1,0) {};
							\node (4) [black vertex] at (0.5,0.8) {};
							\node () []		 at ($(2,0)+(60:1)$) {};

							\node (4') [black vertex] at ($(1,0)$) {};
							\node (3') [black vertex] at ($(3)+(0.5,-0.8)$) {};	
							\node (2') [black vertex] at ($(2)+(0.5,-0.8)$) {};
							\node (1') [black vertex] at ($(0,0)$) {};
							\node (0') [black vertex] at ($(0,0)+(-1,0)$) {};
							\node (5') [] at ($(4,0)$) {};

							\node (1'') [black vertex] at ($(2')$) {};
							\node (0'') [black vertex] at ($(1'')+(-1,0)$) {};	
							\node (2'') [black vertex] at ($(1'')-(0,1)$) {};
							\node (3'') [black vertex] at ($(2'')+(1,0)$) {};
							\node (4'') [black vertex] at ($(3'')+(1,0)$) {};
							\node (5'') [black vertex] at ($(4'')+(1,0)$) {};

							\node () [] at ($(0)+(180:.2)$) {$a_0$};
							\node () [] at ($(1)+(90:.2)$) {$a_1$};
							\node () [] at ($(2)+(90:.3)$) {};
							\node () [] at ($(3)+(0:.55)$) {$a_3=b_4$};
							\node () [] at ($(4)+(90:.2)$) {$a_4$};

							\node () [] at ($(0')+(180:.2)$) {$b_0$};
							\node () [] at ($(1')+(-90:.25)-(0.4,0)$) {$b_1=a_2$};
							\node () [] at ($(2')+(90:.35)$) {};
							\node () [] at ($(3')+(0:.25)$) {$b_3$};
							\node () [] at ($(4')+(-90:.3)$) {};

							\node () [] at ($(0'')+(180:.2)$) {$c_0$};
							\node () [] at ($(1'')+(225:.3)+(0.85,0)$) {$c_1=b_2$};
							\node () [] at ($(2'')+(180:.25)$) {$c_2$};
							\node () [] at ($(3'')+(90:.25)$) {$c_3$};
							\node () [] at ($(4'')+(90:.25)$) {$c_4$};
							\node () [] at ($(5'')+(90:.25)$) {$c_5$};

							\draw[line width=1.3pt,color=gray,<-] (0) -- (1);
							\draw[line width=1.5pt,dotted,color=green,->] (2) -- (1);
							\draw[line width=1.5pt,dotted,color=green,->] (3) -- (4);
							\draw[line width=1.5pt,dashed,color=red,->] (4) -- (2);
							\draw[line width=1.3pt,M edge] (2) -- (3);

							\draw[line width=1.5pt,dashed,color=red,<-] (0') -- (1');
							\draw[line width=1.5pt,dashed,color=red,->] (4') -- (2');
							\draw[line width=1.5pt,dotted,color=green,->] (2') -- (1');
							\draw[line width=1.5pt,dotted,color=green,->] (3') -- (4');
							\draw[line width=1.3pt,M edge] (2') -- (3');

							\draw[line width=1.5pt,dashed,color=red,<-] (0'') -- (1'');
							\draw[line width=1.5pt,dashed,color=red,->] (4'') -- (5'');
							\draw[line width=1.5pt,dotted,color=green,->] (2'') -- (1'');
							\draw[line width=1.5pt,dotted,color=green,->] (4'') -- (3'');
							\draw[line width=1.5pt,dashed,color=red,->] (3'') -- (2'');
					\end{tikzpicture}}
				\end{subfigure}
				\caption{Exchange of edges between two elements of type~A
					and an element of type~C
					in the proof of Claim~\ref{claim:AAB,AAC}.
				}
				\label{fig:case7}
			\end{figure}	
		
		\begin{claim}\label{claim:no-free-element}
			There is no free element of type~A.
		\end{claim}
		
		\begin{proof}
			Suppose, for a contradiction, that \(\D\)
			contains a free element, say \(T_1\), of type~A.
			By Definition~\ref{def:complete-commutative}\eqref{def:complete-commutative-hanging-edge},
			there are two hanging edges \(e_2\) and \(e_2'\) at \(\cv_1(T_1)\).
			We may assume \(e_2 \notin E(T_1)\).
			Let \(T_2\) be the element of \(\D\) that contains \(e_2\).
			By Claim~\ref{claim:AB,AC}, \(T_2\) is not of type~B or C,
			and since \(M\) is a matching, \(T_2\) is not of type~D.
			Thus, \(T_2\) is of type~A.
			By Definition~\ref{def:complete-commutative}\eqref{def:complete-commutative-hanging-edge},
			there is a hanging edge \(e_3\) on \(\cv_2(T_2)\).
			Note that \(e_3\notin E(T_2)\).
			Let \(T_3\) be the element of \(\D\) that contains \(e_3\).
			By Claim~\ref{claim:AAB,AAC}, 
			\(T_3\) is of type~D,
			which implies that there are two edges of \(M\)
			incident to \(\cv_2(T_2)\), a contradiction. 
		\end{proof}	
		
		Now, consider the auxiliary directed graph \(D_\D\) in which \(V(D_\D) = \D\)
		and $(T_1,T_2)$ is an arc of \(D_\D\) if and only if \(\tr(T_2) = \cv_i(T_1)\) for some \(i\in\{1,2\}\).
		It is clear that the elements of type~A in \(\D\) are partitioned into A-chains
		if and only if \(D_\D\) consists of vertex-disjoint directed cycles and isolated vertices.	
		Since every vertex of \(G\) is a connection vertex of at most one element of \(\D\),
		by Claim~\ref{claim:no-free-element}, every vertex of \(D_\D\) has in degree precisely \(1\).
		
		Note also that given two elements \(T_1\) and \(T_2\) we have \(\tr(T_1) \neq \tr(T_2)\),
		otherwise there would be a vertex with two green in edges.
		This implies that every vertex of \(D_\D\) has out degree at most \(2\).
		Now, if \(T_1\) and \(T_2\) are two elements of type~A in \(\D\)
		such that \(\cv_1(T_1) = \tr(T_2) = u_1\),
		by Lemma~\ref{lemma:type12}, we have \(\aux(T_2) = \cv_2(T_1) = u_2\),
		which means that \(E(T_1)\cup E(T_2)\) contains the four edges in \(E(G)\) incident to \(u_1\) and five edges incident to $u_2$,
		and hence, no other element of \(\D\) contains $u_2$, 
		and no other element of \(\D\) has $u_1$ as its tricky vertex.
		This implies that every vertex of \(D_\D\) has out degree at most \(1\),
		and hence \(D_\D\) consists of vertex-disjoint directed cycles and isolated vertices as desired.
	\end{proof}
	
	\subsection{Admissible decompositions}
	
	In this section, we present a new decomposition invariant, which we call \emph{admissibility}, and conclude our proof.
	For that, we introduce an important object, the exceptional pair. Let \(G\) be a \(\{g,r\}\)-graph, and let \(\D\) be a decomposition of \(G\) into trails of length \(5\).
	We say that a pair \((T_1,T_2)\) of elements of \(\D\) is an \emph{exceptional pair}
	if \(T_1\) and \(T_2\) are elements of type~A and C,
	respectively, and can be written as \(T_1 = a_0a_1a_2a_3a_4a_5\) and 
	\(T_2 = b_0b_1b_2b_3b_4b_5\)
	such that \(a_2a_3\in M_{g,r}\),
	\(a_2=a_5 = b_3\),
	and
	\(a_2a_1,a_3a_4,b_2b_1,b_4b_3\in F_g\), 
	\(a_4a_5,b_3b_2,b_4b_5\in F_r\),
	\(a_1a_0,b_1b_0\in M_{g,r}\cup F_g\cup F_r\)
	(see Figure \ref{fig:typeE}).
	Note that since \(G\) is a simple graph, we have \(b_4\neq a_3\).
	Also, if \(2g + 2r \neq 0\), then we have \(b_1 \neq a_3\).
	This yields the following remark.
	
	\begin{remark}\label{remark:exceptional-hanging-edge}
		If \(G\) is a \(\{g,r\}\)-graph for which \(2g+2r\neq 0\) and \((T_1,T_2)\) is an exceptional pair,
		then \(T_2\) does not contain a hanging edge at \(\cv_1(T_1)\).
	\end{remark}

	\begin{figure}
		\centering
		\scalebox{.8}{\begin{tikzpicture}[scale=1.5]
					
				\node (4') [black vertex] at ($(1,0)$) {};
				\node (3') [black vertex] at ($(1.5,-0.8)$) {};	
				\node (2') [black vertex] at ($(0.5,-0.8)$) {};
				\node (1') [black vertex] at ($(0,0)$) {};
				\node (0') [black vertex] at ($(0,0)+(-1,0)$) {};
				\node (5') [] at ($(2,0)$) {};
				
				\node () [] at ($(0')+(90:.2)$) {$a_0$};
				\node () [] at ($(1')+(90:.2)$) {$a_1$};
				\node () [] at ($(2')+(180:.3)$) {$a_2$};
				\node () [] at ($(3')+(0:.3)$) {$a_3$};
				\node () [] at ($(4')+(90:.2)$) {$a_4$};
				
				\node (1'') [black vertex] at ($(0,0)+(-0.5,-0.8)$) {};
				\node (0'') [black vertex] at ($(1'')+(-1,0)$) {};	
				\node (2'') [black vertex] at ($(1'')+(0.5,-0.8)$) {};
				\node (3'') [black vertex] at ($(2')$) {};
				\node (4'') [black vertex] at ($(2')+(0.5,-0.8)$) {};
				\node (5'') [black vertex] at ($(4'')+(1,0)$) {};
				
				\node () [] at ($(0'')+(180:.3)$) {$b_0$};
				\node () [] at ($(1'')+(45:.3)$) {$b_1$};
				\node () [] at ($(2'')+(0:.3)$) {$b_2$};
				
				\node () [] at ($(4'')+(-90:.3)$) {$b_4$};
				\node () [] at ($(5'')+(0:.3)$) {$b_5$};

				\draw[line width=1.3pt,color=gray,<-] (0') -- (1');
				\draw[line width=1.5pt,dashed,color=red,->] (4') -- (2');
				\draw[line width=1.5pt,dotted,color=green,->] (2') -- (1');
				\draw[line width=1.5pt,dotted,color=green,->] (3') -- (4');
				\draw[line width=1.3pt,M edge] (2') -- (3');

				\draw[line width=1.3pt,color=gray,<-] (0'') -- (1'');
				\draw[line width=1.5pt,dashed,color=red,->] (4'') -- (5'');
				\draw[line width=1.5pt,dotted,color=green,->] (2'') -- (1'');
				\draw[line width=1.5pt,dotted,color=green,->] (4'') -- (3'');
				\draw[line width=1.5pt,dashed,color=red,->] (3'') -- (2'');
			\end{tikzpicture}
		}
		\caption{An exceptional pair.}
		\label{fig:typeE}
	\end{figure}
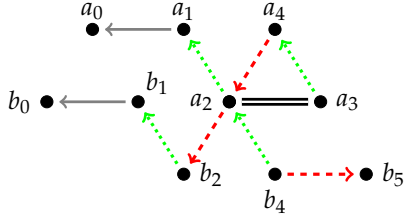
	An \emph{open chain} is a sequence $T_0, T_1,\ldots, T_{s-1}$ of \(s\geq 2\) elements of \(\D\)
	with the following properties.
	(i) 	\(T_0\) is a free element of type~A;
	(ii)	$T_j$ is an element of type~A and \({\tr(T_j) = \cv_i(T_{j-1})}\),
	for every $j \in\{0,\ldots,s-2\}$ and some \(i\in\{1,2\}\); and
	(iii)	$T_{s-1}$ is an element of type~C for which \((T_{s-2},T_{s-1})\) is an exceptional pair. 
	We remark that open chains are not A-chains.
	The next definition describes the invariant studied in this section.
	
	\begin{defi}\label{def:semi-complete-commutative}
		We say that a decomposition \(\D\) of a \(\{g,r\}\)-graph \(G\) 
		into trails of length \(5\) 		
		is \emph{admissible} if the following hold.
		\begin{enumerate}[(a)]
			\item\label{def:semi-complete-commutative-types}
			Every element in \(\D\) is either a path or an element of type~A;
			\item\label{def:semi-complete-commutative-hanging-edge}
			For every element \(T\in\D\) of type~A, we have \(\hang\big(\cv_1(T)\big)\geq 2\),
			and there is at most one element \(T\in \D\) of type~A for which
			\(\hang\big(\cv_2(T)\big) = 0\),
			and, in this case, there is an open chain \({S = T_0,\ldots, T_{s-2},T_{s-1}}\) in \(\D\), 
			for which \(T_{s-2} = T\);
			\item\label{def:semi-complete-free}
			The elements of type~A in \(\D\) are partitioned into A-chains and at most one open chain.
		\end{enumerate}
	\end{defi}
	
	It is not hard to check that the decomposition given 
	by Lemma~\ref{lemma:P5-decomposition}
	is an admissible decomposition.
	Therefore, every $\{g,r\}$-graph for which $2g+2r\neq 0$
	admits an admissible decomposition.
	By performing a few more exchanges of edges between the elements
	of a same A-chain of an admissible decomposition, we can show that
	an admissible decomposition that minimizes its number of elements of type~A
	is in fact a \(P_5\)-decomposition.

	\begin{theorem}\label{theorem:2g2r!=0}
		Every $\{g,r\}$-graph for which \(2g+2r\neq 0\) admits a $P_5$-decomposition.
	\end{theorem}
	
	\begin{proof}
		
		Let $g$ and $r$ be as in the statement, let \(G\) be a \(\{g,r\}\)-graph, and put \({M=M_{g,r}}\).
		By Lemma~\ref{lemma:P5-decomposition}, 
		\(G\) admits an admissible decomposition.
		Let \(\D\) be an admissible decomposition of \(G\) that minimizes \(\tau(\D)\).
		In what follows, we prove that \(\tau(\D) = 0\).
		Suppose, for a contradiction, that \(\tau(\D)>0\).
		We divide A-chains into three types, according to the connections between its elements.
		Given \(i\in\{1,2\}\), 
		we say that an A-chain $S=T_0, T_1,\ldots, T_{s-1}$ is of type~\(i\) if $\tr(T_j)=\cv_i(T_{j-1})$ 
		for every $j \in \{0,\ldots,s-1\}$;
		and we say that \(S\) is a \emph{mixed} A-chain if \(S\) is not of type~\(1\) or \(2\).
		
		Similarly to the proof of Lemma~\ref{lemma:P5-decomposition}, 
		in each step, we exchange edges between some elements of \(\D\) 
		and obtain an admissible decomposition \(\D'\) into trails of length~5 
		such that \(\tau(\D') <\tau(\D)\), which is a contradiction to the minimality of \(\D\).
		To check that \(\D'\) is an admissible decomposition, we observe the three following items:
		(i) The only connection vertex that has fewer hanging edges in \(\D'\) than in \(\D\)  
		is the secondary connection vertex of an element \(T_1\) of type~A, 
		and in this case there is an element \(T_2\) of type~C
		such that \((T_1,T_2)\) is an exceptional pair,
		and hence Definition~\ref{def:semi-complete-commutative}\eqref{def:semi-complete-commutative-hanging-edge} holds for \(\D'\);
		(ii) every element of \(\D'\) that is not an element of \(\D\), i.e., 
		the elements involved in the exchange of edges, is a path or an element of type~A,
		and hence \ref{def:semi-complete-commutative}\eqref{def:semi-complete-commutative-types} 
		holds for \(\D'\);
		(iii) either an open chain is shortened by at least one element, an A-chain is converted into an open chain,
		or all the elements of an A-chain are replaced by paths of length 5,
		and hence \ref{def:semi-complete-commutative}\eqref{def:semi-complete-free} 
		holds for \(\D'\).

		\begin{claim}\label{claim:mixed-A-chain}
			Every A-chain in $\D$ is mixed.		
		\end{claim}
		
		\begin{proof}
			Suppose, for a contradiction, that there is an A-chain \(S = T_0,T_1,\ldots, T_{s-1}\) of type~1 or~2.
			Let $T_j=a_{0,j}a_{1,j}a_{2,j}a_{3,j}a_{4,j}a_{5,j}$,
			where \(a_{5,j} = a_{2,j}\), \(a_{2,j}a_{3,j}\in M\), \({a_{2,j}a_{1,j}, a_{3,j}a_{4,j}\in F_g}\), 
			\(a_{4,j}a_{2,j}\in F_r\) and \(a_{1,j}a_{0,j}\in M\cup F_g \cup F_r\).
			For \(i\in\{1,2,3,4,5\}\), the edge \(a_{i-1,j}a_{i,j}\) is called the \emph{\(i\)th edge of \(T_j\)}.
			In what follows, we divide the proof according to the type of S.

			\smallskip	\noindent
			\textbf{\(\mathbf{S}\) is of type~1.}
			In this case, we have \({a_{3,j} = \cv_1(T_{j}) = \tr(T_{j+1}) = a_{4,j+1}}\) for each \({j\in \{0,\ldots,s-1}\} \),
			and hence, by Lemma~\ref{lemma:type12}, 
			we have $a_{2,j} = \cv_2(T_{j}) = \aux(T_{j+1})=a_{1,j+1}$.
			Now, for each \(j=0,\ldots,s-1\), 
			let \(T'_j = a_{2,j+1}a_{3,j}a_{4,j}a_{1,j}a_{2,j}a_{0,j+1}\) (see Figure~\ref{fig:case-1-2-cycles}).
			Note that 
			\(T'_j = T_j 
			- a_{1,j}a_{0,j} + a_{1,j+1}a_{0,j+1} 
			- a_{2,j}a_{3,j} + a_{2,j-1}a_{3,j-1} 
			- a_{4,j}a_{2,j} + a_{4,j+1}a_{2,j+1}\).
			More specifically, 
			\(a_{2,j+1}a_{3,j}		= a_{4,j+1}a_{5,j+1}\) 		is the \(5\)th edge of \(T_{j+1}\);
			\(a_{3,j}a_{4,j}\)									is the \(4\)th edge of \(T_{j}\);
			\(a_{4,j}a_{1,j}		= a_{2,j-1}a_{3,j-1}\)		is the \(3\)rd edge of \(T_{j-1}\);
			\(a_{1,j}a_{2,j}\)									is the \(2\)nd edge of \(T_{j}\);
			\(a_{2,j}a_{0,j+1}	= a_{1,j+1}a_{0,j+1}\)			is the \(1\)st edge of \(T_{j+1}\).
			Clearly, \(T'_j\) is a trail of length \(5\).
			Moreover, since, for each \(i\in\{1,2,3,4,5\}\), 
			the element \(T'_j\) contains the \(i\)th edge of an element of \(S\), 
			and, if \(j\neq j'\), the elements \(T'_j\) and \(T'_{j'}\) contain the \(i\)th edge of different elements of \(S\),
			the set \(\D'=\big(\D\setminus \{T_j\colon j=0,\ldots,s-1\}\big)\cup\{T'_j\colon j=0,\ldots,s-1\}\)
			is a decomposition of \(G\) into trails of length \(5\).
			We may regard \(\D'\) as the decomposition obtained by reversing the direction of two components of \(F_g\), namely, the green edges in \(S\), 
			and applying the same strategy used in Proposition~\ref{proposition:initial-decomposition}.
			
			\begin{figure}[h]
				\centering
				\begin{subfigure}{.45\textwidth}
					\centering
					\scalebox{.7}{\begin{tikzpicture}[scale=1.5]
							
							\draw[fatpath,backcolor2] (-0.5,-0.5) -- (-0.31,0.95) -- (0,0) -- (-0.5,-0.5) -- (-1.12,1.4) -- (-1.93,1.85) -- (-1.12,1.4);
						
							\node (1) [black vertex] at (0,0) {};
							\node (1') [black vertex] at (-0.5,-0.5) {};
							\node (1'') [black vertex] at (-1,-1) {};
							
							\node (2) [black vertex] at (1,0) {};
							\node (2') [black vertex] at (1.5,-0.5) {};
							\node (2'') [black vertex] at (2,-1) {};
							
							\node (3) [black vertex] at (1.31,0.95) {};
							\node (3') [black vertex] at (2.12,1.4) {};
							\node (3'') [black vertex] at (2.93,1.85) {};
							
							\node (4) [black vertex] at (0.5,1.54) {};
							\node (4') [black vertex] at (0.5,2.58) {};
							\node (4'') [black vertex] at (0.5,3.62) {};
							
							\node (5) [black vertex] at (-0.31,0.95) {};
							\node (5') [black vertex] at (-1.12,1.4) {};
							\node (5'') [black vertex] at (-1.93,1.85) {};

							\draw[line width=1.5pt,dotted,color=green,->] (2) -- (1);
							\draw[line width=1.5pt,dotted,color=green,->] (3) -- (2);
							\draw[line width=1.5pt,dotted,color=green,->] (4) -- (3);
							\draw[line width=1.5pt,dotted,color=green,->] (5) -- (4);
							\draw[line width=1.5pt,dotted,color=green,->] (1) -- (5);
							\draw[line width=1.5pt,dotted,color=green,->] (2') -- (1');
							\draw[line width=1.5pt,dotted,color=green,->] (3') -- (2');
							\draw[line width=1.5pt,dotted,color=green,->] (4') -- (3');
							\draw[line width=1.5pt,dotted,color=green,->] (5') -- (4');
							\draw[line width=1.5pt,dotted,color=green,->] (1') -- (5');
							\draw[line width=1.5pt,dashed,color=red,->] (5) -- (1');
							\draw[line width=1.5pt,dashed,color=red,->] (1) -- (2');
							\draw[line width=1.5pt,dashed,color=red,->] (2) -- (3');
							\draw[line width=1.5pt,dashed,color=red,->] (3) -- (4');
							\draw[line width=1.5pt,dashed,color=red,->] (4) -- (5');
							
							\draw[line width=1.5pt,dashed,color=red,->] (1') -- (1'');
							\draw[line width=1.5pt,dashed,color=red,->] (2') -- (2'');
							\draw[line width=1.5pt,dashed,color=red,->] (3') -- (3'');
							\draw[line width=1.5pt,dashed,color=red,->] (4') -- (4'');
							\draw[line width=1.5pt,dashed,color=red,->] (5') -- (5'');
							
							\draw[line width=1.3pt,M edge] (1) -- (1');
							\draw[line width=1.3pt,M edge] (2) -- (2');
							\draw[line width=1.3pt,M edge] (3) -- (3');
							\draw[line width=1.3pt,M edge] (4) -- (4');
							\draw[line width=1.3pt,M edge] (5) -- (5');
							
					\end{tikzpicture}}
					\label{fig:case-1-2-cycles-1}
				\end{subfigure}
				\begin{subfigure}{.45\textwidth}
					\centering
					\scalebox{.7}{\begin{tikzpicture}[scale=1.5]

							\draw[fatpath,backcolor2]  (1.31,0.95) --(0.5,2.58) -- (1.31,0.95) -- (1,0) -- (1.5,-0.5) -- (2.12,1.4) -- (2.93,1.85) -- (2.12,1.4);

							\node () [] at ($(0.5,2.58)+(0:.5)$) {$a_{2,j+1}$};
							\node () [] at ($(1.31,0.95)+(180:.5)$) {$a_{3,j}$};
							\node () [] at ($((1,0)+(135:.4)$) {$a_{4,j}$};
							\node () [] at ($(1.5,-0.5)+(225:.3)$) {$a_{1,j}$};
							\node () [] at ($(2.12,1.4)+(335:.4)$) {$a_{2,j}$};
							\node () [] at ($(2.93,1.85)+(90:.3)$) {$a_{0,j+1}$};
							
							\node (1) [black vertex] at (0,0) {};
							\node (1') [black vertex] at (-0.5,-0.5) {};
							\node (1'') [black vertex] at (-1,-1) {};
							
							\node (2) [black vertex] at (1,0) {};
							\node (2') [black vertex] at (1.5,-0.5) {};
							\node (2'') [black vertex] at (2,-1) {};
							
							\node (3) [black vertex] at (1.31,0.95) {};
							\node (3') [black vertex] at (2.12,1.4) {};
							\node (3'') [black vertex] at (2.93,1.85) {};
							
							\node (4) [black vertex] at (0.5,1.54) {};
							\node (4') [black vertex] at (0.5,2.58) {};
							\node (4'') [black vertex] at (0.5,3.62) {};
							
							\node (5) [black vertex] at (-0.31,0.95) {};
							\node (5') [black vertex] at (-1.12,1.4) {};
							\node (5'') [black vertex] at (-1.93,1.85) {};

							\draw[line width=1.5pt,dotted,color=green,->] (2) -- (1);
							\draw[line width=1.5pt,dotted,color=green,->] (3) -- (2);
							\draw[line width=1.5pt,dotted,color=green,->] (4) -- (3);
							\draw[line width=1.5pt,dotted,color=green,->] (5) -- (4);
							\draw[line width=1.5pt,dotted,color=green,->] (1) -- (5);
							\draw[line width=1.5pt,dotted,color=green,->] (2') -- (1');
							\draw[line width=1.5pt,dotted,color=green,->] (3') -- (2');
							\draw[line width=1.5pt,dotted,color=green,->] (4') -- (3');
							\draw[line width=1.5pt,dotted,color=green,->] (5') -- (4');
							\draw[line width=1.5pt,dotted,color=green,->] (1') -- (5');
							\draw[line width=1.5pt,dashed,color=red,->] (5) -- (1');
							\draw[line width=1.5pt,dashed,color=red,->] (1) -- (2');
							\draw[line width=1.5pt,dashed,color=red,->] (2) -- (3');
							\draw[line width=1.5pt,dashed,color=red,->] (3) -- (4');
							\draw[line width=1.5pt,dashed,color=red,->] (4) -- (5');
							
							\draw[line width=1.5pt,dashed,color=red,->] (1') -- (1'');
							\draw[line width=1.5pt,dashed,color=red,->] (2') -- (2'');
							\draw[line width=1.5pt,dashed,color=red,->] (3') -- (3'');
							\draw[line width=1.5pt,dashed,color=red,->] (4') -- (4'');
							\draw[line width=1.5pt,dashed,color=red,->] (5') -- (5'');
							
							\draw[line width=1.3pt,M edge] (1) -- (1');
							\draw[line width=1.3pt,M edge] (2) -- (2');
							\draw[line width=1.3pt,M edge] (3) -- (3');
							\draw[line width=1.3pt,M edge] (4) -- (4');
							\draw[line width=1.3pt,M edge] (5) -- (5');
							
						\end{tikzpicture}
					}
					\label{fig:case-1-2-cycles-2}
				\end{subfigure}
				\caption{Exchange of edges between the elements of an A-chain of type~1 with
					five elements in the proof of Claim~\ref{claim:mixed-A-chain}.}
				\label{fig:case-1-2-cycles}
			\end{figure}
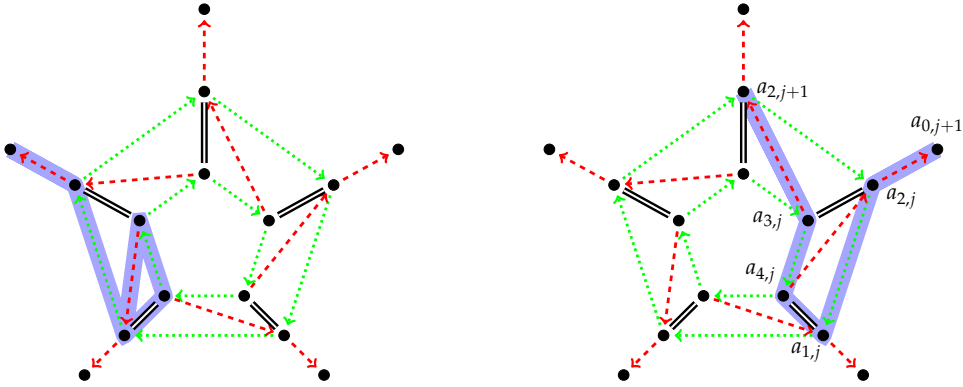
			
			In order to prove that $T'_j$ is a path,
			we show that \(a_{2,j+1},a_{0,j+1}\notin\{a_{3,j},a_{4,j},a_{1,j},a_{2,j}\}\).
			Note that, since for each \(j\in\{0,\ldots,s-1\}\), 
			\(T_j\) is a path, we have \(a_{i,j} \neq a_{i',j}\) for every \(i\neq i'\).
			Since \(G\) is a simple graph, we have \(a_{2,j+1}\notin\{a_{3,j},a_{4,j},a_{2,j}\}\) 
			and \(a_{0,j+1}\notin\{a_{3,j},a_{4,j},a_{1,j},a_{2,j}\}\);
			and if \(a_{2,j+1} = a_{1,j}\), then \(a_{4,j+1}a_{2,j+1}\) and \(a_{4,j-1,}a_{2,j-1}\) 
			are two distinct red in edges of \(a_{1,j}\), a contradiction.

			We claim that \(\D'\) is an admissible decomposition.
			Indeed, the only vertices of the elements of \(S\) that can be connection vertices of elements in \(\D'\)
			are the vertices \(a_{0,j}\), for \(j=0,\ldots,s-1\).
			But a hanging edge at a vertex \(a_{0,j}\) is in \(T_{j'}\in\D\) if and only if \(a_{0,j} = a_{3,j'}\) for some \(j'\neq j\),
			and, in this case \(a_{3,j'}\) is not a connection vertex in \(\D'\) because 
			all edges incident to it are in elements of \(\{T'_j\colon j=0,\ldots,s-1\}\).
			Therefore, Definition~\ref{def:semi-complete-commutative}\eqref{def:semi-complete-commutative-hanging-edge} holds for \(\D'\).
			Moreover, since \(T'_j\) is a path, for \(j=0,\ldots,s-1\), Definition~\ref{def:semi-complete-commutative}\eqref{def:complete-commutative-types} holds for \(\D'\).
			Finally, \(\D\) and \(\D'\) have the same number of open chains, 
			and hence Definition~\ref{def:semi-complete-commutative}\eqref{def:semi-complete-free} holds for \(\D'\).
			Therefore, \(\D'\) is an admissible decomposition of \(G\) such that \({\tau(\D') = \tau(\D) - s}\), 
			a contradiction to the minimality of \(\D\).
			
			\smallskip	\noindent
			\textbf{\(\mathbf{S}\) is of type~2.}
			In this case, for each \(j=0,\ldots,s-1\), we have \(a_{2,j} = \cv_2(T_{j}) = \tr(T_{j+1}) = a_{4,j+1}\).
			Now, for each \(j=0,\ldots,s-1\), 
			let \(T'_j = a_{0,j}a_{1,j}a_{2,j}a_{3,j}a_{4,j}a_{4,j-1}\) (see Figure~\ref{fig:case-1-1-cycles}).
			Clearly, \(T'_j\) is a trail of length \(5\).
			Note that 
			\(T'_j = T_j - a_{4,j}a_{2,j} + a_{4,j-1}a_{2,j-1}\),
			i.e., \(T'_j\) is the element obtained from \(T_j\) by exchanging its \(5\)th edge by the \(5\)th edge of \(T_{j-1}\).
			Thus, the set \(\D'=\big(\D\setminus \{T_j\colon j=0,\ldots,s-1\}\big)\cup\{T'_j\colon j=0,\ldots,s-1\}\)
			is a decomposition of \(G\) into trails of length \(5\).
			We may regard \(\D'\) as the decomposition obtained by reversing the direction of one component of \(F_r\)
			and applying the same strategy used in Proposition~\ref{proposition:initial-decomposition}.
			In order to prove that $T'_j$ is a path,
			we show that \(a_{4,j-1}\notin\{a_{0,j},a_{1,j},a_{2,j},a_{3,j},a_{4,j}\}\).
			Note that, since for each \(j\in\{0,\ldots,s-1\}\), 
			\(T_j\) is a path, we have \(a_{i,j} \neq a_{i',j}\) for every \(i\neq i'\).
			Since \(G\) is a simple graph, we have \(a_{4,j-1}\notin\{a_{2,j},a_{3,j},a_{4,j}\}\);
			also, by Lemma~\ref{lemma:no-cycles}, we have \(a_{4,j-1}\neq a_{0,j}\);
			and if \(a_{4,j-1} = a_{1,j}\), then \(a_{2,j}a_{1,j}\) and \(a_{3,j-1}a_{4,j-1}\) are two distinct green in edges of \(a_{4,j-1}\),
			a contradiction.

			\begin{figure}[H]
				\centering
				\begin{subfigure}{.45\textwidth}
					\centering
					\scalebox{.7}{\begin{tikzpicture}[scale=1.5]

							\draw[fatpath,backcolor2] (0.5,2.54) -- (1.5,2.54)-- (0.5,2.54) --(0.5,1.54) -- (-0.31,1.8) -- (-0.31,0.95) -- (0.5,1.54);

							\node (1) [black vertex] at (0,0) {};
							\node (2) [black vertex] at (1,0) {};
							\node (3) [black vertex] at (1.31,0.95) {};
							\node (4) [black vertex] at (0.5,1.54) {};
							\node (5) [black vertex] at (-0.31,0.95) {};
							
							\node (12) [black vertex] at (0.5,-0.69) {};
							\node (23) [black vertex] at (1.81,0.26) {};
							\node (34) [black vertex] at (1.31,1.8) {};
							\node (45) [black vertex] at (-0.31,1.8) {};
							\node (15) [black vertex] at (-0.81,0.26) {};
							
							\node (1') [black vertex] at (-0.59,-0.81) {};
							\node (3') [black vertex] at (2.26,1.26) {};
							\node (2') [black vertex] at (1.59,-0.81) {};
							\node (4') [black vertex] at (0.5,2.54) {};
							\node (5') [black vertex] at (-1.26,1.26) {};
							
							\node (1'') [black vertex] at (-1.4,-0.22) {};
							\node (2'') [black vertex] at (0.78,-1.4) {};
							\node (3'') [black vertex] at (2.57,0.31) {};
							\node (4'') [black vertex] at (1.5,2.54) {};
							\node (5'') [black vertex] at (-0.95,2.21) {};

							\draw[line width=1.5pt,dashed,color=red,->] (2) -- (1);
							\draw[line width=1.5pt,dashed,color=red,->] (3) -- (2);
							\draw[line width=1.5pt,dashed,color=red,->] (4) -- (3);
							\draw[line width=1.5pt,dashed,color=red,->] (5) -- (4);
							\draw[line width=1.5pt,dashed,color=red,->] (1) -- (5);
							
							\draw[line width=1.5pt,dotted,color=green,<-] (2) -- (12);
							\draw[line width=1.5pt,dotted,color=green,<-] (1) -- (15);
							\draw[line width=1.5pt,dotted,color=green,<-] (5) -- (45);
							\draw[line width=1.5pt,dotted,color=green,<-] (4) -- (34);
							\draw[line width=1.5pt,dotted,color=green,<-] (3) -- (23);
							
							\draw[line width=1.5pt,dotted,color=green,->] (2) -- (2');
							\draw[line width=1.5pt,dotted,color=green,->] (1) -- (1');
							\draw[line width=1.5pt,dotted,color=green,->] (5) -- (5');
							\draw[line width=1.5pt,dotted,color=green,->] (4) -- (4');
							\draw[line width=1.5pt,dotted,color=green,->] (3) -- (3');
							
							\draw[line width=1.3pt,M edge] (1) -- (12);
							\draw[line width=1.3pt,M edge] (2) -- (23);
							\draw[line width=1.3pt,M edge] (3) -- (34);
							\draw[line width=1.3pt,M edge] (4) -- (45);
							\draw[line width=1.3pt,M edge] (5) -- (15);
							
							\draw[line width=1.3pt,color=gray] (2'') -- (2');
							\draw[line width=1.3pt,color=gray] (1'') -- (1');
							\draw[line width=1.3pt,color=gray] (5'') -- (5');
							\draw[line width=1.3pt,color=gray] (4'') -- (4');
							\draw[line width=1.3pt,color=gray] (3'') -- (3');

						\end{tikzpicture}
					}
					\label{fig:case-1-1-cycles-1}
				\end{subfigure}
				\begin{subfigure}{.45\textwidth}
					\centering
					\scalebox{.7}{\begin{tikzpicture}[scale=1.5]

							\draw[fatpath,backcolor2] (0.5,2.54) -- (1.5,2.54) -- (0.5,2.54) -- (0.5,1.54) -- (-0.31,1.8) -- (-0.31,0.95) -- (0,0) -- (-0.31,0.95);
							
							\node () [] at ($(1.5,2.54)+(90:.25)$) {$a_{0,j}$};
							\node () [] at ($((0.5,2.54)+(90:.25)$) {$a_{1,j}$};
							\node () [] at ($(0.5,1.54)+(270:.35)$) {$a_{2,j}$};
							\node () [] at ($((-0.31,0.95)+(350:.45)$) {$a_{4,j}$};
							\node () [] at ($(0,0)+(30:.4)$) {$a_{4,j-1}$};
							\node () [] at ($(-0.31,1.8)+(90:.3)$) {$a_{3,j}$};

							\node (1) [black vertex] at (0,0) {};
							\node (2) [black vertex] at (1,0) {};
							\node (3) [black vertex] at (1.31,0.95) {};
							\node (4) [black vertex] at (0.5,1.54) {};
							\node (5) [black vertex] at (-0.31,0.95) {};
							
							\node (12) [black vertex] at (0.5,-0.69) {};
							\node (23) [black vertex] at (1.81,0.26) {};
							\node (34) [black vertex] at (1.31,1.8) {};
							\node (45) [black vertex] at (-0.31,1.8) {};
							\node (15) [black vertex] at (-0.81,0.26) {};
							
							\node (1') [black vertex] at (-0.59,-0.81) {};
							\node (3') [black vertex] at (2.26,1.26) {};
							\node (2') [black vertex] at (1.59,-0.81) {};
							\node (4') [black vertex] at (0.5,2.54) {};
							\node (5') [black vertex] at (-1.26,1.26) {};
							
							\node (1'') [black vertex] at (-1.4,-0.22) {};
							\node (2'') [black vertex] at (0.78,-1.4) {};
							\node (3'') [black vertex] at (2.57,0.31) {};
							\node (4'') [black vertex] at (1.5,2.54) {};
							\node (5'') [black vertex] at (-0.95,2.21) {};

							\draw[line width=1.5pt,dashed,color=red,->] (2) -- (1);
							\draw[line width=1.5pt,dashed,color=red,->] (3) -- (2);
							\draw[line width=1.5pt,dashed,color=red,->] (4) -- (3);
							\draw[line width=1.5pt,dashed,color=red,->] (5) -- (4);
							\draw[line width=1.5pt,dashed,color=red,->] (1) -- (5);
							
							\draw[line width=1.5pt,dotted,color=green,<-] (2) -- (12);
							\draw[line width=1.5pt,dotted,color=green,<-] (1) -- (15);
							\draw[line width=1.5pt,dotted,color=green,<-] (5) -- (45);
							\draw[line width=1.5pt,dotted,color=green,<-] (4) -- (34);
							\draw[line width=1.5pt,dotted,color=green,<-] (3) -- (23);
							
							\draw[line width=1.5pt,dotted,color=green,->] (2) -- (2');
							\draw[line width=1.5pt,dotted,color=green,->] (1) -- (1');
							\draw[line width=1.5pt,dotted,color=green,->] (5) -- (5');
							\draw[line width=1.5pt,dotted,color=green,->] (4) -- (4');
							\draw[line width=1.5pt,dotted,color=green,->] (3) -- (3');
							
							\draw[line width=1.3pt,M edge] (1) -- (12);
							\draw[line width=1.3pt,M edge] (2) -- (23);
							\draw[line width=1.3pt,M edge] (3) -- (34);
							\draw[line width=1.3pt,M edge] (4) -- (45);
							\draw[line width=1.3pt,M edge] (5) -- (15);
							
							\draw[line width=1.3pt,color=gray] (2'') -- (2');
							\draw[line width=1.3pt,color=gray] (1'') -- (1');
							\draw[line width=1.3pt,color=gray] (5'') -- (5');
							\draw[line width=1.3pt,color=gray] (4'') -- (4');
							\draw[line width=1.3pt,color=gray] (3'') -- (3');

						\end{tikzpicture}
					}
					\label{fig:case-1-1-cycles-2}
				\end{subfigure}
				\caption{Exchange of edges between the elements of an A-chain of type~2 with
					five elements in the proof of Claim~\ref{claim:mixed-A-chain}.}
				\label{fig:case-1-1-cycles}
			\end{figure}		
			
			We claim that \(\D'\) is an admissible decomposition.
			Indeed, the only vertices that have hanging edges in \(\D\) and may not have hanging edges in \(\D'\)
			are the vertices \(a_{3,j}\) and \(a_{4,j} = a_{2,j-1}\), for \(j=0,\ldots,s-1\),
			but these vertices are connection vertices of the elements in \(S\),
			and hence can't be connection vertices of elements in \(\D'\).
			Therefore, Definition~\ref{def:semi-complete-commutative}\eqref{def:semi-complete-commutative-hanging-edge} holds for \(\D'\).
			Moreover, since \(T'_j\) is a path, for \(j=0,\ldots,s-1\), Definition~\ref{def:semi-complete-commutative}\eqref{def:complete-commutative-types} holds for \(\D'\).
			Finally, \(\D\) and \(\D'\) have the same number of open chains, 
			and hence Definition~\ref{def:semi-complete-commutative}\eqref{def:semi-complete-free} holds for \(\D'\).
			Therefore, \(\D'\) is an admissible decomposition of \(G\) such that \(\tau(\D') = \tau(\D) - s\), 
			a contradiction to the minimality of \(\D\).
		\end{proof}
		
		\begin{claim}\label{claim:no-short-A-chain}
			Every A-chain contains at least four elements
		\end{claim}
		
		\begin{proof}
			First, note that if an A-chain consists of two elements, then \(G\) contains a parallel edge, which is a contradiction. 
			Thus, let \(S\) be an A-chain in $\D$ with precisely three elements, say \(T_1\), \(T_2\), and \(T_3\).  
			By Claim~\ref{claim:mixed-A-chain}, we may assume $\tr(T_1)=\cv_2(T_3)$,
			${\tr(T_2)=\cv_1(T_1)}$ and $\tr(T_3)=\cv_i(T_2)$, for $i \in \{1,2\}$. In what follows, we divide the proof depending on whether $i=1$ or $i=2$.
			
			Let $T_1=a_0a_1a_2a_3a_4a_5$, $T_2=b_0b_1b_2b_3b_4b_5$ and \(T_3 = c_0c_1c_2c_3c_4c_5\) be the elements of $S$ where $a_4=\tr(T_1)=\cv_2(T_3)=c_2$,
			$b_4=\tr(T_2)=\cv_1(T_1)=a_3$ and ${c_4=\tr(T_3)=\cv_i(T_2)}$.

			\smallskip	\noindent
			\textbf{Case $\mathbf{i=1}$.}
			In this case, $c_4=\tr(T_3)=\cv_1(T_2)=b_3$.
			Since $b_3=c_4$, we have $b_2=c_1$ and $c_0=a_1$.
			Put \(T_1' = a_0a_1a_2a_3a_4c_1\),
			\(T_2' = b_0b_1b_2b_4b_3c_2\),
			\(T_3' = c_0c_1c_4c_3c_2a_2\) (see Figure~\ref{fig:case4-similar4})  and
			\({\D' = \big(\D\setminus\{T_1,T_2,T_3\}\big)\cup\{T_1',T_2',T_3'\}}\).
			We claim that  $T_1'$, \(T_2'\) and \(T_3'\)  are paths.
			By Lemma~\ref{lemma:T'1-is-path}, \(T_2'\) is a path.
			Since $G$ is simple,
			\(c_1\notin \{a_2,a_3,a_4\}\) and
			\(a_2\notin \{c_1,c_2,c_3,c_4\}\).
			By Lemma~\ref{lemma:no-cycles}, we have \(c_1 \neq a_0\) and $a_2\neq c_0$. 
			Therefore, $T_3'$ is a path.	
			Finally, if \(c_1 = a_1\), then \(c_2c_1\) and \(a_2a_1\) are two green out edges at \(a_1\), a contradiction.					
			Therefore, $T_1'$ is a path, 
			and hence definition~\ref{def:semi-complete-commutative}\eqref{def:semi-complete-commutative-types} holds.
			Also, \(\hang_{\D'}(v) \geq \hang_{\D}(v)\) 
			for every \(v\in V(G)\setminus\{a_3,b_3,c_3\}\). 
			Thus, definition~\ref{def:semi-complete-commutative}\eqref{def:semi-complete-commutative-hanging-edge} holds for $\D$.
			Since $\D$ is admissible and the new elements are paths, the elements of type~A 
			are still partitioned into A-chains and at most one open chain,
			and hence \ref{def:semi-complete-commutative}\eqref{def:semi-complete-free}
			holds for~\(\D'\). 
			Therefore, \(\D'\) is an admissible decomposition of \(G\)
			such that \(\tau(\D')  = \tau(\D) - 3\), a contradiction to the minimality of~\(\D\).
			
			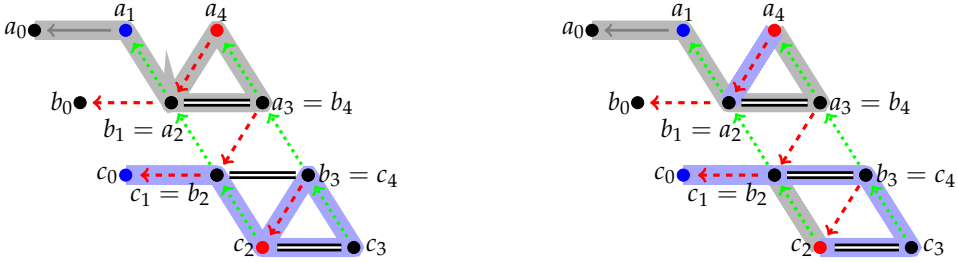
\begin{figure}[H]
				\centering
				\begin{subfigure}{.45\textwidth}
					\centering
					\scalebox{.8}{\begin{tikzpicture}[scale=1.5]
			
							\draw[fatpath,backcolor1] (-0.5,0.8) -- (-1.5,0.8) -- (-0.5,0.8) -- (0,0) -- (0,0) -- (1,0) -- (0.5,0.8) -- (0,0);
							\draw[fatpath,backcolor2] (0.5,-0.8) -- (-0.5,-0.8)-- (0.5,-0.8) -- (1,-1.6) -- (2,-1.6) -- (1.5,-0.8) -- (1,-1.6);

							\node (0) [black vertex] at (-1.5,0.8) {};
							\node (1) [blue vertex] at (-0.5,0.8) {};
							\node (2) [black vertex] at (0,0) {};
							\node (3) [black vertex] at (1,0) {};
							\node (4) [red vertex] at (0.5,0.8) {};
							\node () []		 at ($(1,0)+(60:1)$) {};

							\node (4') [black vertex] at ($(1,0)$) {};
							\node (3') [black vertex] at ($(3)+(0.5,-0.8)$) {};	
							\node (2') [black vertex] at ($(2)+(0.5,-0.8)$) {};
							\node (1') [black vertex] at ($(0,0)$) {};
							\node (0') [black vertex] at ($(0,0)+(-1,0)$) {};
							\node (5') [] at ($(3,0)$) {};

							\node (4'') [black vertex] at ($(3')$) {};
							\node (3'') [black vertex] at ($(3)+(1,-1.6)$) {};	
							\node (2'') [red vertex] at ($(2)+(1,-1.6)$) {};
							\node (1'') [black vertex] at ($(2')$) {};
							\node (0'') [blue vertex] at ($(1'')+(-1,0)$) {};
							\node (5'') [] at ($(3,0)$) {};

							\node () [] at ($(0)+(180:.2)$) {$a_0$};
							\node () [] at ($(1)+(90:.2)$) {$a_1$};
							\node () [] at ($(2)+(90:.3)$) {};
							\node () [] at ($(3)+(0:.55)$) {$a_3=b_4$};
							\node () [] at ($(4)+(90:.2)$) {$a_4$};

							\node () [] at ($(0')+(180:.2)$) {$b_0$};
							\node () [] at ($(1')+(-90:.3)-(0.3,0)$) {$b_1=a_2$};
							\node () [] at ($(2')+(90:.35)$) {};
							\node () [] at ($(3')+(0:.55)$) {$b_3=c_4$};
							\node () [] at ($(4')+(-90:.3)$) {};

							\node () [] at ($(0'')+(180:.2)$) {$c_0$};
							\node () [] at ($(1'')+(225:.3)-(0.3,0)$) {$c_1=b_2$};
							\node () [] at ($(2'')+(180:.2)$) {$c_2$};
							\node () [] at ($(3'')+(0:.25)$) {$c_3$};
							\node () [] at ($(4'')+(270:.35)$) {};

							\draw[line width=1.3pt,color=gray,<-] (0) -- (1);
							\draw[line width=1.5pt,dotted,color=green,->] (2) -- (1);
							\draw[line width=1.5pt,dotted,color=green,->] (3) -- (4);
							\draw[line width=1.5pt,dashed,color=red,->] (4) -- (2);
							\draw[line width=1.3pt,M edge] (2) -- (3);

							\draw[line width=1.5pt,dashed,color=red,<-] (0') -- (1');
							\draw[line width=1.5pt,dashed,color=red,->] (4') -- (2');
							\draw[line width=1.5pt,dotted,color=green,->] (2') -- (1');
							\draw[line width=1.5pt,dotted,color=green,->] (3') -- (4');
							\draw[line width=1.3pt,M edge] (2') -- (3');

							\draw[line width=1.5pt,dashed,color=red,<-] (0'') -- (1'');
							\draw[line width=1.5pt,dashed,color=red,->] (4'') -- (2'');
							\draw[line width=1.5pt,dotted,color=green,->] (2'') -- (1'');
							\draw[line width=1.5pt,dotted,color=green,->] (3'') -- (4'');
							\draw[line width=1.3pt,M edge] (2'') -- (3'');
						\end{tikzpicture}
					}
				\end{subfigure}
				\begin{subfigure}{.45\textwidth}
					\centering
					\scalebox{.8}{\begin{tikzpicture}[scale=1.5]
							
							\draw[fatpath,backcolor1] (-0.5,0.8) -- (-1.5,0.8) -- (-0.5,0.8) -- (0,0) -- (1,0) -- (0.5,0.8);
							\draw[fatpath,backcolor1] (1,-1.6) -- (0.5,-0.8) --(1,-1.6) -- (0.5,-0.8);
							\draw[fatpath,backcolor2] (0.5,-0.8) -- (-0.5,-0.8) -- (0.5,-0.8) -- (-0.5,-0.8); 
							\draw[fatpath,backcolor2] (2,-1.6) -- (1,-1.6) -- (2,-1.6) -- (1.5,-0.8) -- (0.5,-0.8) -- (1.5,-0.8);
							\draw[fatpath,backcolor2] (0,0) -- (0.5,0.8) -- (0,0) -- (0.5,0.8);

							\node (0) [black vertex] at (-1.5,0.8) {};
							\node (1) [blue vertex] at (-0.5,0.8) {};
							\node (2) [black vertex] at (0,0) {};
							\node (3) [black vertex] at (1,0) {};
							\node (4) [red vertex] at (0.5,0.8) {};
							\node () []		 at ($(1,0)+(60:1)$) {};

							\node (4') [black vertex] at ($(1,0)$) {};
							\node (3') [black vertex] at ($(3)+(0.5,-0.8)$) {};	
							\node (2') [black vertex] at ($(2)+(0.5,-0.8)$) {};
							\node (1') [black vertex] at ($(0,0)$) {};
							\node (0') [black vertex] at ($(0,0)+(-1,0)$) {};
							\node (5') [] at ($(3,0)$) {};

							\node (4'') [black vertex] at ($(3')$) {};
							\node (3'') [black vertex] at ($(3)+(1,-1.6)$) {};	
							\node (2'') [red vertex] at ($(2)+(1,-1.6)$) {};
							\node (1'') [black vertex] at ($(2')$) {};
							\node (0'') [blue vertex] at ($(1'')+(-1,0)$) {};
							\node (5'') [] at ($(3,0)$) {};

							\node () [] at ($(0)+(180:.2)$) {$a_0$};
							\node () [] at ($(1)+(90:.2)$) {$a_1$};
							\node () [] at ($(2)+(90:.3)$) {};
							\node () [] at ($(3)+(0:.55)$) {$a_3=b_4$};
							\node () [] at ($(4)+(90:.2)$) {$a_4$};

							\node () [] at ($(0')+(180:.2)$) {$b_0$};
							\node () [] at ($(1')+(-90:.3)-(0.3,0)$) {$b_1=a_2$};
							\node () [] at ($(2')+(90:.35)$) {};
							\node () [] at ($(3')+(0:.55)$) {$b_3=c_4$};
							\node () [] at ($(4')+(-90:.3)$) {};

							\node () [] at ($(0'')+(180:.2)$) {$c_0$};
							\node () [] at ($(1'')+(225:.3)-(0.3,0)$) {$c_1=b_2$};
							\node () [] at ($(2'')+(180:.2)$) {$c_2$};
							\node () [] at ($(3'')+(0:.25)$) {$c_3$};
							\node () [] at ($(4'')+(270:.35)$) {};

							\draw[line width=1.3pt,color=gray,<-] (0) -- (1);
							\draw[line width=1.5pt,dotted,color=green,->] (2) -- (1);
							\draw[line width=1.5pt,dotted,color=green,->] (3) -- (4);
							\draw[line width=1.5pt,dashed,color=red,->] (4) -- (2);
							\draw[line width=1.3pt,M edge] (2) -- (3);

							\draw[line width=1.5pt,dashed,color=red,<-] (0') -- (1');
							\draw[line width=1.5pt,dashed,color=red,->] (4') -- (2');
							\draw[line width=1.5pt,dotted,color=green,->] (2') -- (1');
							\draw[line width=1.5pt,dotted,color=green,->] (3') -- (4');
							\draw[line width=1.3pt,M edge] (2') -- (3');

							\draw[line width=1.5pt,dashed,color=red,<-] (0'') -- (1'');
							\draw[line width=1.5pt,dashed,color=red,->] (4'') -- (2'');
							\draw[line width=1.5pt,dotted,color=green,->] (2'') -- (1'');
							\draw[line width=1.5pt,dotted,color=green,->] (3'') -- (4'');
							\draw[line width=1.3pt,M edge] (2'') -- (3'');
						\end{tikzpicture}
					}
				\end{subfigure}
				\caption{Exchange performed in the proof of Claim~\ref{claim:no-short-A-chain}
					in the case \(\tr(T_3) = \cv_1(T_2)\). The red (resp. blue) circles illustrate the same vertex, i.e., $a_4=c_2$ (resp. $a_1=c_0$).}
				\label{fig:case4-similar4}
			\end{figure}

			\smallskip	\noindent
			\textbf{Case $\mathbf{i=2}$.}
			In this case, $c_4=\tr(T_3)=\cv_2(T_2)=b_2$.
			Put \(T_1' = a_0a_1a_2a_4a_3b_5\),
			\({T_2' = b_0b_1b_4b_3b_2c_2}\),
			\(T_3' = c_0c_1c_2c_3c_4b_1\)  (see Figure~\ref{fig:case5-similar5}) and let
			\(\D' = \big(\D\setminus\{T_1,T_2,T_3\}\big)\cup\{T_1',T_2',T_3'\}\).
			By Lemma~\ref{lemma:T'1-is-path}, \(T_1'\) is a path.
			Since \(G\) is simple, we have \(c_2\notin \{b_1,b_2,b_3,b_4\}\) and \(b_1\notin \{c_2,c_3,c_4\}\).
			By Lemma~\ref{lemma:no-cycles}, \(b_2 \neq a_0, ~c_2 \neq b_0,
			~b_1 \neq c_0\). Therefore, $T_2'$ is a path.					
			If \(b_1 = c_1\), then \(b_2b_1\) and \(c_2c_1\) are two green in edges at \(c_1\), a contradiction.
			Therefore, $T_3'$ is a path. 
			Analogously to the case above,
			\(\D'\) is an admissible decomposition of \(G\)
			such that \(\tau(\D')  = \tau(\D) - 3\), a contradiction to the minimality of \(\D\).
		\end{proof}
		
		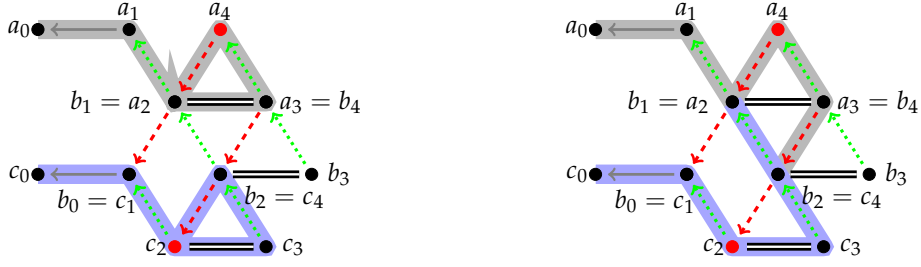
\begin{figure}
			\centering
			\begin{subfigure}{.45\textwidth}
				\centering
				\scalebox{.8}{\begin{tikzpicture}[scale=1.5]
						
						\draw[fatpath,backcolor1] (-0.5,0.8) -- (-1.5,0.8) -- (-0.5,0.8) -- (0,0) -- (0,0) -- (1,0) -- (0.5,0.8) -- (0,0);
						\draw[fatpath,backcolor2] (-0.5,-0.8) -- (-1.5,-0.8)-- (-0.5,-0.8) -- (0,-1.6) -- (1,-1.6) -- (0.5,-0.8) -- (0,-1.6);
					
						\node (0) [black vertex] at (-1.5,0.8) {};
						\node (1) [black vertex] at (-0.5,0.8) {};
						\node (2) [black vertex] at (0,0) {};
						\node (3) [black vertex] at (1,0) {};
						\node (4) [red vertex] at (0.5,0.8) {};
						\node () []		 at ($(0,0)+(60:1)$) {};

						\node (4') [black vertex] at ($(1,0)$) {};
						\node (3') [black vertex] at ($(3)+(0.5,-0.8)$) {};	
						\node (2') [black vertex] at ($(2)+(0.5,-0.8)$) {};
						\node (1') [black vertex] at ($(0,0)$) {};
						\node (0') [black vertex] at ($(2')-(1,0)$) {};
						\node (5') [] at ($(2,0)$) {};

						\node (4'') [black vertex] at ($(2')$) {};
						\node (3'') [black vertex] at ($(3)-(0,1.6)$) {};	
						\node (2'') [red vertex] at ($(2)-(0,1.6)$) {};
						\node (1'') [black vertex] at ($(4'')-(1,0)$) {};
						\node (0'') [black vertex] at ($(1'')+(-1,0)$) {};
						\node (5'') [] at ($(2,0)$) {};

						\node () [] at ($(0)+(180:.2)$) {$a_0$};
						\node () [] at ($(1)+(90:.2)$) {$a_1$};
						\node () [] at ($(2)+(90:.3)$) {};
						\node () [] at ($(3)+(0:.6)$) {$a_3=b_4$};
						\node () [] at ($(4)+(90:.2)$) {$a_4$};

						\node () [] at ($(0')+(-90:.3)-(0.35,0)$) {$b_0=c_1$};
						\node () [] at ($(1')+(180:.4)-(0.3,0)$) {$b_1=a_2$};
						\node () [] at ($(2')+(-90:.25)+(0.7,0)$) {$b_2=c_4$};
						\node () [] at ($(3')+(0:.3)$) {$b_3$};
						\node () [] at ($(4')+(-90:.3)$) {};

						\node () [] at ($(0'')+(180:.2)$) {$c_0$};
						\node () [] at ($(1'')+(225:.4)$) {};
						\node () [] at ($(2'')+(180:.2)$) {$c_2$};
						\node () [] at ($(3'')+(0:.3)$) {$c_3$};
						\node () [] at ($(4'')+(90:.3)$) {};

						\draw[line width=1.3pt,color=gray,<-] (0) -- (1);
						\draw[line width=1.5pt,dotted,color=green,->] (2) -- (1);
						\draw[line width=1.5pt,dotted,color=green,->] (3) -- (4);
						\draw[line width=1.5pt,dashed,color=red,->] (4) -- (2);
						\draw[line width=1.3pt,M edge] (2) -- (3);

						\draw[line width=1.5pt,dashed,color=red,<-] (0') -- (1');
						\draw[line width=1.5pt,dashed,color=red,->] (4') -- (2');
						\draw[line width=1.5pt,dotted,color=green,->] (2') -- (1');
						\draw[line width=1.5pt,dotted,color=green,->] (3') -- (4');
						\draw[line width=1.3pt,M edge] (2') -- (3');

						\draw[line width=1.3pt,color=gray,<-] (0'') -- (1'');
						\draw[line width=1.5pt,dashed,color=red,->] (4'') -- (2'');
						\draw[line width=1.5pt,dotted,color=green,->] (2'') -- (1'');
						\draw[line width=1.5pt,dotted,color=green,->] (3'') -- (4'');
						\draw[line width=1.3pt,M edge] (2'') -- (3'');
					\end{tikzpicture}
				}
			\end{subfigure}
			\begin{subfigure}{.45\textwidth}
				\centering
				\scalebox{.8}{\begin{tikzpicture}[scale=1.5]
					
						\draw[fatpath,backcolor1] (-0.5,0.8) -- (-1.5,0.8) -- (-0.5,0.8) -- (0,0) -- (0.5,0.8) -- (1,0) -- (0.5,-0.8);
						\draw[fatpath,backcolor2] (-0.5,-0.8) -- (-1.5,-0.8)-- (-0.5,-0.8) -- (0,-1.6) -- (1,-1.6) -- (0.5,-0.8) -- (0,0);
					
						\node (0) [black vertex] at (-1.5,0.8) {};
						\node (1) [black vertex] at (-0.5,0.8) {};
						\node (2) [black vertex] at (0,0) {};
						\node (3) [black vertex] at (1,0) {};
						\node (4) [red vertex] at (0.5,0.8) {};
						\node () []		 at ($(0,0)+(60:1)$) {};

						\node (4') [black vertex] at ($(1,0)$) {};
						\node (3') [black vertex] at ($(3)+(0.5,-0.8)$) {};	
						\node (2') [black vertex] at ($(2)+(0.5,-0.8)$) {};
						\node (1') [black vertex] at ($(0,0)$) {};
						\node (0') [black vertex] at ($(2')-(1,0)$) {};
						\node (5') [] at ($(2,0)$) {};

						\node (4'') [black vertex] at ($(2')$) {};
						\node (3'') [black vertex] at ($(3)-(0,1.6)$) {};	
						\node (2'') [red vertex] at ($(2)-(0,1.6)$) {};
						\node (1'') [black vertex] at ($(4'')-(1,0)$) {};
						\node (0'') [black vertex] at ($(1'')+(-1,0)$) {};
						\node (5'') [] at ($(2,0)$) {};

						\node () [] at ($(0)+(180:.2)$) {$a_0$};
						\node () [] at ($(1)+(90:.2)$) {$a_1$};
						\node () [] at ($(2)+(90:.3)$) {};
						\node () [] at ($(3)+(0:.6)$) {$a_3=b_4$};
						\node () [] at ($(4)+(90:.2)$) {$a_4$};

						\node () [] at ($(0')+(-90:.3)-(0.35,0)$) {$b_0=c_1$};
						\node () [] at ($(1')+(180:.4)-(0.3,0)$) {$b_1=a_2$};
						\node () [] at ($(2')+(-90:.25)+(0.7,0)$) {$b_2=c_4$};
						\node () [] at ($(3')+(0:.3)$) {$b_3$};
						\node () [] at ($(4')+(-90:.3)$) {};

						\node () [] at ($(0'')+(180:.2)$) {$c_0$};
						\node () [] at ($(1'')+(225:.4)$) {};
						\node () [] at ($(2'')+(180:.2)$) {$c_2$};
						\node () [] at ($(3'')+(0:.3)$) {$c_3$};
						\node () [] at ($(4'')+(90:.3)$) {};

						\draw[line width=1.3pt,color=gray,<-] (0) -- (1);
						\draw[line width=1.5pt,dotted,color=green,->] (2) -- (1);
						\draw[line width=1.5pt,dotted,color=green,->] (3) -- (4);
						\draw[line width=1.5pt,dashed,color=red,->] (4) -- (2);
						\draw[line width=1.3pt,M edge] (2) -- (3);

						\draw[line width=1.5pt,dashed,color=red,<-] (0') -- (1');
						\draw[line width=1.5pt,dashed,color=red,->] (4') -- (2');
						\draw[line width=1.5pt,dotted,color=green,->] (2') -- (1');
						\draw[line width=1.5pt,dotted,color=green,->] (3') -- (4');
						\draw[line width=1.3pt,M edge] (2') -- (3');

						\draw[line width=1.3pt,color=gray,<-] (0'') -- (1'');
						\draw[line width=1.5pt,dashed,color=red,->] (4'') -- (2'');
						\draw[line width=1.5pt,dotted,color=green,->] (2'') -- (1'');
						\draw[line width=1.5pt,dotted,color=green,->] (3'') -- (4'');
						\draw[line width=1.3pt,M edge] (2'') -- (3'');
					\end{tikzpicture}
				}
			\end{subfigure}
			\caption{Exchange performed in the proof of Claim~\ref{claim:no-short-A-chain}
				in the case \(\tr(T_3) = \cv_2(T_2)\). The red circles illustrate the same vertex.}
			\label{fig:case5-similar5}
		\end{figure}

		\begin{claim}\label{claim:open-chain}
			\(\D\) contains an open chain.
		\end{claim}	
		\begin{proof}
			Suppose, for a contradiction, that there is no open chain in \(\D\).
			Since \(\tau(\D)>0\), {\(\D\)~contains} an A-chain \(S = T_0,T_1,\ldots, T_{s-1}\).
			By Claim~\ref{claim:mixed-A-chain}, \(S\) is a mixed A-chain.
			Then we can find three consecutive elements in \(S\), say \(T_j, T_{j+1},T_{j+2}\),
			such that \(\cv_2(T_j)= \tr(T_{j+1})\) and \(\cv_1(T_{j+1}) = \tr(T_{j+2})\).
			By the cyclic structure of $S$, we may assume, without loss of generality, that \(j=0\).
			By Claim~\ref{claim:no-short-A-chain}, we have \(s\geq 4\), 
			and hence there is an element \(T_3\in\D\) such that
			\(\tr(T_3) = \cv_i(T_2)\), for some \(i\in\{1,2\}\).
			In what follows, the proof is divided according to \(i\).
			Let 
			\(T_0 = a_0a_1a_2a_3a_4a_5\),
			\(T_1 = b_0b_1b_2b_3b_4b_5\),
			\(T_2 = c_0c_1c_2c_3c_4c_5\), and
			\(T_3 = d_0d_1d_2d_3d_4d_5\),
			where \(a_5 = a_2\), \(b_5 = b_2\), \(c_5 = c_2\), \(d_5 = d_2\),
			and \(a_2a_3, b_2b_3, c_2c_3, d_2d_3 \in M\).
			By the choice of \(T_0\), \(T_1\), and \(T_2\), 
			we have \({b_4 = \tr(T_1) = \cv_2(T_0) = a_2}\),
			\({c_4 = \tr(T_2) = \cv_1(T_1) = b_3}\).
			The exchanges of edges performed here are analogous to the exchanges performed 
			on the proof of Claim~\ref{claim:AAB,AAC} of Lemma~\ref{lemma:P5-decomposition} for elements of type~A.
			
			\smallskip	\noindent 
		\textbf{Case} \(\mathbf{\btr\boldsymbol{(}T_3\boldsymbol{)} \boldsymbol{=} \bcv_1\boldsymbol{(}T_2\boldsymbol{)}}\)\textbf{.}
			In this case, we have \(d_4 = c_3\) and, by Lemma~\ref{lemma:type12}, \(c_2 = d_1\).
			Put \(T_1' = b_0b_1b_2b_4b_3c_2\),
			\(T_2' = c_0c_1c_4c_3c_2d_0\),
			\(T_3' = c_1d_1d_2d_3d_4d_2\) (see Figure~\ref{fig:case4-similar}),
			and let \(\D' = \big(\D\setminus\{T_1,T_2,T_3\}\big)\cup\{T_1',T_2',T_3'\}\).
			By Lemma~\ref{lemma:T'1-is-path}, \(T'_1\) is an element of type~C.		
			In what follows, we prove that \(T_2'\) is a path and \(T_3'\) is an element of type~A.
			Since \(G\) is simple, we have \(d_0\notin \{c_1,c_2,c_3,c_4\}\),
			and \(c_1\notin \{d_1,d_2,d_4\}\).
			By Lemma~\ref{lemma:no-cycles}, \(d_0 \neq c_0\),
			and hence, $T_2'$ is a path.
			If \(c_1 = d_3\),
			then \(b_2b_3\) and \(d_2d_3\) are two edges of \(M\) incident to \(c_1\),
			a contradiction.
			Therefore, $T_3'$ is an element of type~A,
			and hence definition~\ref{def:semi-complete-commutative}\eqref{def:semi-complete-commutative-types} holds.
			Also, \(\hang_{\D'}(v) \geq \hang_{\D}(v)\) 
			for every \(v\in V(G)\setminus\{a_2=b_4,b_3,c_3\}\). 
			Note also that \(b_4b_2\) is a hanging edge at \(a_2 = \cv_2(T_0)\) in \(\D\),
			but not in \(\D'\).
			However, \((T_0,T_1')\) is an exceptional pair.
			Also, $b_3$ and $c_3$ are not connection vertices of $\D'$.
			Since \(c_3\) is not a connection vertex in \(\D'\),
			the element \(T'_3\) is free.
			Therefore, \(S'=T'_3,\ldots, T_{s-1},T_1,T'_2\) is an open chain,
			and hence Definition~\ref{def:semi-complete-commutative}\eqref{def:semi-complete-commutative-hanging-edge} 
			holds for \(\D'\). 
			Finally, note that an element \(T\) of type~A in \(\D\setminus\{T_1,T_2,T_3\}\) 
			is either in an A-chain of \(\D\) different from \(S\),
			which implies that \(T\) is in an A-chain of \(\D'\),
			or is in \(S\), which implies that \(T\) is in \(S'\).
			Thus, Definition~\ref{def:semi-complete-commutative}\eqref{def:semi-complete-free} holds for \(\D'\).
			Therefore, \(\D'\) is an admissible decomposition of \(G\)
			such that \(\tau(\D')  = \tau(\D) - 2\), a contradiction to the minimality of \(\D\).	
			
			\begin{figure}[h]
				\centering
				\begin{subfigure}{.45\textwidth}
					\centering
					\scalebox{.8}{\begin{tikzpicture}[scale=1.5]
						
							\draw[fatpath,backcolor1] (-0.5,0.8) -- (-1.5,0.8) -- (-0.5,0.8) -- (0,0) -- (0,0) -- (1,0) -- (0.5,0.8) -- (0,0);
							\draw[fatpath,backcolor2] (0.5,-0.8) -- (-0.5,-0.8)-- (0.5,-0.8) -- (1,-1.6) -- (2,-1.6) -- (1.5,-0.8) -- (1,-1.6);

							\node (0) [black vertex] at (-1.5,0.8) {};
							\node (1) [black vertex] at (-0.5,0.8) {};
							\node (2) [black vertex] at (0,0) {};
							\node (3) [black vertex] at (1,0) {};
							\node (4) [black vertex] at (0.5,0.8) {};
							\node () []		 at ($(1,0)+(60:1)$) {};

							\node (4') [black vertex] at ($(1,0)$) {};
							\node (3') [black vertex] at ($(3)+(0.5,-0.8)$) {};	
							\node (2') [black vertex] at ($(2)+(0.5,-0.8)$) {};
							\node (1') [black vertex] at ($(0,0)$) {};
							\node (0') [black vertex] at ($(0,0)+(-1,0)$) {};
							\node (5') [] at ($(3,0)$) {};

							\node (4'') [black vertex] at ($(3')$) {};
							\node (3'') [black vertex] at ($(3)+(1,-1.6)$) {};	
							\node (2'') [black vertex] at ($(2)+(1,-1.6)$) {};
							\node (1'') [black vertex] at ($(2')$) {};
							\node (0'') [black vertex] at ($(1'')+(-1,0)$) {};
							\node (5'') [] at ($(3,0)$) {};

							\node () [] at ($(0)+(180:.2)$) {$b_0$};
							\node () [] at ($(1)+(90:.2)$) {$b_1$};
							\node () [] at ($(2)+(90:.3)$) {};
							\node () [] at ($(3)+(0:.55)$) {$b_3=c_4$};
							\node () [] at ($(4)+(90:.2)$) {$b_4=a_2$};

							\node () [] at ($(0')+(180:.2)$) {$c_0$};
							\node () [] at ($(1')+(-90:.3)-(0.3,0)$) {$c_1=b_2$};
							\node () [] at ($(2')+(90:.35)$) {};
							\node () [] at ($(3')+(0:.55)$) {$c_3=d_4$};
							\node () [] at ($(4')+(-90:.3)$) {};

							\node () [] at ($(0'')+(180:.2)$) {$d_0$};
							\node () [] at ($(1'')+(225:.3)-(0.3,0)$) {$d_1=c_2$};
							\node () [] at ($(2'')+(180:.2)$) {$d_2$};
							\node () [] at ($(3'')+(0:.25)$) {$d_3$};
							\node () [] at ($(4'')+(270:.35)$) {};

							\draw[line width=1.3pt,color=gray,<-] (0) -- (1);
							\draw[line width=1.5pt,dotted,color=green,->] (2) -- (1);
							\draw[line width=1.5pt,dotted,color=green,->] (3) -- (4);
							\draw[line width=1.5pt,dashed,color=red,->] (4) -- (2);
							\draw[line width=1.3pt,M edge] (2) -- (3);

							\draw[line width=1.5pt,dashed,color=red,<-] (0') -- (1');
							\draw[line width=1.5pt,dashed,color=red,->] (4') -- (2');
							\draw[line width=1.5pt,dotted,color=green,->] (2') -- (1');
							\draw[line width=1.5pt,dotted,color=green,->] (3') -- (4');
							\draw[line width=1.3pt,M edge] (2') -- (3');

							\draw[line width=1.5pt,dashed,color=red,<-] (0'') -- (1'');
							\draw[line width=1.5pt,dashed,color=red,->] (4'') -- (2'');
							\draw[line width=1.5pt,dotted,color=green,->] (2'') -- (1'');
							\draw[line width=1.5pt,dotted,color=green,->] (3'') -- (4'');
							\draw[line width=1.3pt,M edge] (2'') -- (3'');
						\end{tikzpicture}
					}
				\end{subfigure}
				\begin{subfigure}{.45\textwidth}
					\centering
					\scalebox{.8}{\begin{tikzpicture}[scale=1.5]
							
							\draw[fatpath,backcolor1] (-0.5,0.8) -- (-1.5,0.8) -- (-0.5,0.8) -- (0,0) -- (0.5,0.8) -- (1,0) -- (0.5,-0.8);
							\draw[fatpath,backcolor2] (0,0)-- (0.5,-0.8) --(1,-1.6) -- (2,-1.6) -- (1.5,-0.8) -- (1,-1.6);

							\node (0) [black vertex] at (-1.5,0.8) {};
							\node (1) [black vertex] at (-0.5,0.8) {};
							\node (2) [black vertex] at (0,0) {};
							\node (3) [black vertex] at (1,0) {};
							\node (4) [black vertex] at (0.5,0.8) {};
							\node () []		 at ($(1,0)+(60:1)$) {};

							\node (4') [black vertex] at ($(1,0)$) {};
							\node (3') [black vertex] at ($(3)+(0.5,-0.8)$) {};	
							\node (2') [black vertex] at ($(2)+(0.5,-0.8)$) {};
							\node (1') [black vertex] at ($(0,0)$) {};
							\node (0') [black vertex] at ($(0,0)+(-1,0)$) {};
							\node (5') [] at ($(3,0)$) {};

							\node (4'') [black vertex] at ($(3')$) {};
							\node (3'') [black vertex] at ($(3)+(1,-1.6)$) {};	
							\node (2'') [black vertex] at ($(2)+(1,-1.6)$) {};
							\node (1'') [black vertex] at ($(2')$) {};
							\node (0'') [black vertex] at ($(1'')+(-1,0)$) {};
							\node (5'') [] at ($(3,0)$) {};

							\node () [] at ($(0)+(180:.2)$) {$b_0$};
							\node () [] at ($(1)+(90:.2)$) {$b_1$};
							\node () [] at ($(2)+(90:.3)$) {};
							\node () [] at ($(3)+(0:.55)$) {$b_3=c_4$};
							\node () [] at ($(4)+(90:.2)$) {$b_4=a_2$};

							\node () [] at ($(0')+(180:.2)$) {$c_0$};
							\node () [] at ($(1')+(-90:.3)-(0.3,0)$) {$c_1=b_2$};
							\node () [] at ($(2')+(90:.35)$) {};
							\node () [] at ($(3')+(0:.55)$) {$c_3=d_4$};
							\node () [] at ($(4')+(-90:.3)$) {};

							\node () [] at ($(0'')+(180:.2)$) {$d_0$};
							\node () [] at ($(1'')+(225:.3)-(0.3,0)$) {$d_1=c_2$};
							\node () [] at ($(2'')+(180:.2)$) {$d_2$};
							\node () [] at ($(3'')+(0:.25)$) {$d_3$};
							\node () [] at ($(4'')+(270:.35)$) {};

							\draw[line width=1.3pt,color=gray,<-] (0) -- (1);
							\draw[line width=1.5pt,dotted,color=green,->] (2) -- (1);
							\draw[line width=1.5pt,dotted,color=green,->] (3) -- (4);
							\draw[line width=1.5pt,dashed,color=red,->] (4) -- (2);
							\draw[line width=1.3pt,M edge] (2) -- (3);

							\draw[line width=1.5pt,dashed,color=red,<-] (0') -- (1');
							\draw[line width=1.5pt,dashed,color=red,->] (4') -- (2');
							\draw[line width=1.5pt,dotted,color=green,->] (2') -- (1');
							\draw[line width=1.5pt,dotted,color=green,->] (3') -- (4');
							\draw[line width=1.3pt,M edge] (2') -- (3');

							\draw[line width=1.5pt,dashed,color=red,<-] (0'') -- (1'');
							\draw[line width=1.5pt,dashed,color=red,->] (4'') -- (2'');
							\draw[line width=1.5pt,dotted,color=green,->] (2'') -- (1'');
							\draw[line width=1.5pt,dotted,color=green,->] (3'') -- (4'');
							\draw[line width=1.3pt,M edge] (2'') -- (3'');
						\end{tikzpicture}
					}
				\end{subfigure}
				\caption{Exchange performed in the proof of Claim~\ref{claim:open-chain}
					in the case \(\tr(T_3) = \cv_1(T_2)\).}
				\label{fig:case4-similar}
			\end{figure}
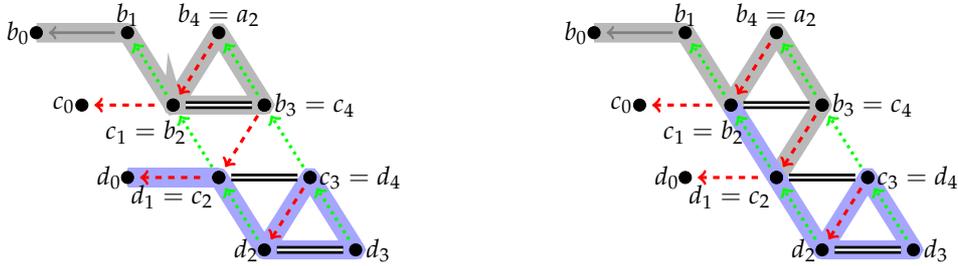

			\smallskip	\noindent
			\textbf{Case} \(\mathbf{\btr\boldsymbol{(}T_3\boldsymbol{)} \boldsymbol{=} \bcv_2\boldsymbol{(}T_2\boldsymbol{)}}\)\textbf{.}
			Put \(T_1' = b_0b_1b_2b_4b_3c_2\),
			\(T_2' = c_0c_1c_4c_3c_2d_2\),
			\(T_3' = d_0d_1d_2d_3d_4c_1\) (see Figure~\ref{fig:case5-similar}),
			and let \(\D' = \big(\D\setminus\{T_1,T_2,T_3\}\big)\cup\{T_1',T_2',T_3'\}\).
			By Lemma~\ref{lemma:T'1-is-path}, \(T'_1\) is an element of type~C.		
			In what follows, we prove that \(T_2'\) and \(T_3'\) are paths.
			Since \(G\) is simple, we have \(d_2\notin \{c_1,c_2,c_3,c_4\}\),
			and \(c_1\notin \{d_2,d_3,d_4\}\).
			By Lemma~\ref{lemma:no-cycles}, \(d_2 \neq c_0\), and \(c_1 \neq d_0\). 
			Therefore, $T_2'$ is a path.
			If \(c_1 = d_1\),
			then \(d_2d_1\) and \(c_2c_1\) are two green in edges of \(c_1\).
			Therefore, $T_3'$ is a path,
			and hence definition~\ref{def:semi-complete-commutative}\eqref{def:semi-complete-commutative-types} holds.
			Also, 
			\(\hang_{\D'}(v) \geq \hang_{\D}(v)\) 
			for every \(v\in V(G)\setminus\{a_2=b_4,b_3,c_3,d_3\}\). 
			Note also that \(b_4b_2\) is a hanging edge at \(a_2\) in \(\D\)
			but not in \(\D'\).
			However, \((T_0,T_1')\) is an exceptional pair.
			Also, $b_3$, $c_3$ and $d_3$ are not connection vertices of $\D'$.
			Thus, since \(d_2\) and \(d_3\) are not connection vertices in \(\D'\),
			the element \(T_4\) (or \(T_1\), if \(s=4\)) is free.
			Therefore, \(S'=T_4,\ldots, T_{s-1},T_1,T'_2\) is an open chain,
			and hence Definition~\ref{def:semi-complete-commutative}\eqref{def:semi-complete-commutative-hanging-edge} 
			holds for \(\D'\). 	
			Finally, note that an element \(T\) of type~A in \(\D\setminus\{T_1,T_2,T_3\}\) 
			is either in an A-chain of \(\D\) different from \(S\),
			which implies that \(T\) is in an A-chain of \(\D'\),
			or is in \(S\), which implies that \(T\) is in \(S'\).
			Thus, Definition~\ref{def:semi-complete-commutative}\eqref{def:semi-complete-free} holds for \(\D'\).
			Therefore, \(\D'\) is an admissible decomposition of \(G\)
			such that \(\tau(\D')  = \tau(\D) - 3\), a contradiction to the minimality of \(\D\).
		\end{proof}
		
		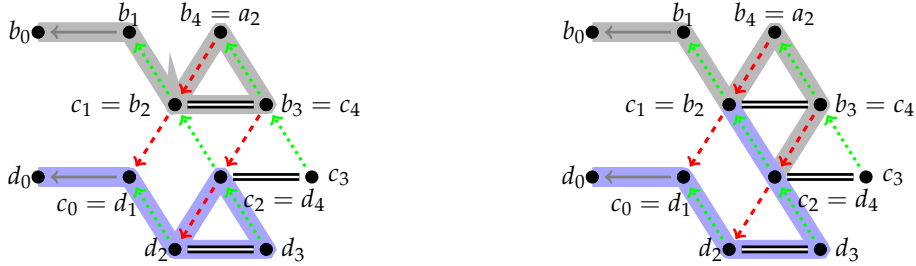
\begin{figure}[H]
			\centering
			\begin{subfigure}{.45\textwidth}
				\centering
				\scalebox{.8}{	\begin{tikzpicture}[scale=1.5]
					
						\draw[fatpath,backcolor1] (-0.5,0.8) -- (-1.5,0.8) -- (-0.5,0.8) -- (0,0) -- (0,0) -- (1,0) -- (0.5,0.8) -- (0,0);
						\draw[fatpath,backcolor2] (-0.5,-0.8) -- (-1.5,-0.8)-- (-0.5,-0.8) -- (0,-1.6) -- (1,-1.6) -- (0.5,-0.8) -- (0,-1.6);
					
						\node (0) [black vertex] at (-1.5,0.8) {};
						\node (1) [black vertex] at (-0.5,0.8) {};
						\node (2) [black vertex] at (0,0) {};
						\node (3) [black vertex] at (1,0) {};
						\node (4) [black vertex] at (0.5,0.8) {};
						\node () []		 at ($(0,0)+(60:1)$) {};

						\node (4') [black vertex] at ($(1,0)$) {};
						\node (3') [black vertex] at ($(3)+(0.5,-0.8)$) {};	
						\node (2') [black vertex] at ($(2)+(0.5,-0.8)$) {};
						\node (1') [black vertex] at ($(0,0)$) {};
						\node (0') [black vertex] at ($(2')-(1,0)$) {};
						\node (5') [] at ($(2,0)$) {};

						\node (4'') [black vertex] at ($(2')$) {};
						\node (3'') [black vertex] at ($(3)-(0,1.6)$) {};	
						\node (2'') [black vertex] at ($(2)-(0,1.6)$) {};
						\node (1'') [black vertex] at ($(4'')-(1,0)$) {};
						\node (0'') [black vertex] at ($(1'')+(-1,0)$) {};
						\node (5'') [] at ($(2,0)$) {};

						\node () [] at ($(0)+(180:.2)$) {$b_0$};
						\node () [] at ($(1)+(90:.2)$) {$b_1$};
						\node () [] at ($(2)+(90:.3)$) {};
						\node () [] at ($(3)+(0:.6)$) {$b_3=c_4$};
						\node () [] at ($(4)+(90:.2)$) {$b_4=a_2$};

						\node () [] at ($(0')+(-90:.3)-(0.35,0)$) {$c_0=d_1$};
						\node () [] at ($(1')+(180:.4)-(0.3,0)$) {$c_1=b_2$};
						\node () [] at ($(2')+(-90:.25)+(0.7,0)$) {$c_2=d_4$};
						\node () [] at ($(3')+(0:.3)$) {$c_3$};
						\node () [] at ($(4')+(-90:.3)$) {};

						\node () [] at ($(0'')+(180:.2)$) {$d_0$};
						\node () [] at ($(1'')+(225:.4)$) {};
						\node () [] at ($(2'')+(180:.2)$) {$d_2$};
						\node () [] at ($(3'')+(0:.3)$) {$d_3$};
						\node () [] at ($(4'')+(90:.3)$) {};

						\draw[line width=1.3pt,color=gray,<-] (0) -- (1);
						\draw[line width=1.5pt,dotted,color=green,->] (2) -- (1);
						\draw[line width=1.5pt,dotted,color=green,->] (3) -- (4);
						\draw[line width=1.5pt,dashed,color=red,->] (4) -- (2);
						\draw[line width=1.3pt,M edge] (2) -- (3);

						\draw[line width=1.5pt,dashed,color=red,<-] (0') -- (1');
						\draw[line width=1.5pt,dashed,color=red,->] (4') -- (2');
						\draw[line width=1.5pt,dotted,color=green,->] (2') -- (1');
						\draw[line width=1.5pt,dotted,color=green,->] (3') -- (4');
						\draw[line width=1.3pt,M edge] (2') -- (3');

						\draw[line width=1.3pt,color=gray,<-] (0'') -- (1'');
						\draw[line width=1.5pt,dashed,color=red,->] (4'') -- (2'');
						\draw[line width=1.5pt,dotted,color=green,->] (2'') -- (1'');
						\draw[line width=1.5pt,dotted,color=green,->] (3'') -- (4'');
						\draw[line width=1.3pt,M edge] (2'') -- (3'');
					\end{tikzpicture}
				}
			\end{subfigure}
			\begin{subfigure}{.45\textwidth}
				\centering
				\scalebox{.8}{\begin{tikzpicture}[scale=1.5]
						
						\draw[fatpath,backcolor1] (-0.5,0.8) -- (-1.5,0.8) -- (-0.5,0.8) -- (0,0) -- (0.5,0.8) -- (1,0) -- (0.5,-0.8);
						\draw[fatpath,backcolor2] (-0.5,-0.8) -- (-1.5,-0.8)-- (-0.5,-0.8) -- (0,-1.6) -- (1,-1.6) -- (0.5,-0.8) -- (0,0);
						
						\node (0) [black vertex] at (-1.5,0.8) {};
						\node (1) [black vertex] at (-0.5,0.8) {};
						\node (2) [black vertex] at (0,0) {};
						\node (3) [black vertex] at (1,0) {};
						\node (4) [black vertex] at (0.5,0.8) {};
						\node () []		 at ($(0,0)+(60:1)$) {};

						\node (4') [black vertex] at ($(1,0)$) {};
						\node (3') [black vertex] at ($(3)+(0.5,-0.8)$) {};	
						\node (2') [black vertex] at ($(2)+(0.5,-0.8)$) {};
						\node (1') [black vertex] at ($(0,0)$) {};
						\node (0') [black vertex] at ($(2')-(1,0)$) {};
						\node (5') [] at ($(2,0)$) {};

						\node (4'') [black vertex] at ($(2')$) {};
						\node (3'') [black vertex] at ($(3)-(0,1.6)$) {};	
						\node (2'') [black vertex] at ($(2)-(0,1.6)$) {};
						\node (1'') [black vertex] at ($(4'')-(1,0)$) {};
						\node (0'') [black vertex] at ($(1'')+(-1,0)$) {};
						\node (5'') [] at ($(2,0)$) {};

						\node () [] at ($(0)+(180:.2)$) {$b_0$};
						\node () [] at ($(1)+(90:.2)$) {$b_1$};
						\node () [] at ($(2)+(90:.3)$) {};
						\node () [] at ($(3)+(0:.6)$) {$b_3=c_4$};
						\node () [] at ($(4)+(90:.2)$) {$b_4=a_2$};

						\node () [] at ($(0')+(-90:.3)-(0.35,0)$) {$c_0=d_1$};
						\node () [] at ($(1')+(180:.4)-(0.3,0)$) {$c_1=b_2$};
						\node () [] at ($(2')+(-90:.25)+(0.7,0)$) {$c_2=d_4$};
						\node () [] at ($(3')+(0:.3)$) {$c_3$};
						\node () [] at ($(4')+(-90:.3)$) {};

						\node () [] at ($(0'')+(180:.2)$) {$d_0$};
						\node () [] at ($(1'')+(225:.4)$) {};
						\node () [] at ($(2'')+(180:.2)$) {$d_2$};
						\node () [] at ($(3'')+(0:.3)$) {$d_3$};
						\node () [] at ($(4'')+(90:.3)$) {};

						\draw[line width=1.3pt,color=gray,<-] (0) -- (1);
						\draw[line width=1.5pt,dotted,color=green,->] (2) -- (1);
						\draw[line width=1.5pt,dotted,color=green,->] (3) -- (4);
						\draw[line width=1.5pt,dashed,color=red,->] (4) -- (2);
						\draw[line width=1.3pt,M edge] (2) -- (3);

						\draw[line width=1.5pt,dashed,color=red,<-] (0') -- (1');
						\draw[line width=1.5pt,dashed,color=red,->] (4') -- (2');
						\draw[line width=1.5pt,dotted,color=green,->] (2') -- (1');
						\draw[line width=1.5pt,dotted,color=green,->] (3') -- (4');
						\draw[line width=1.3pt,M edge] (2') -- (3');

						\draw[line width=1.3pt,color=gray,<-] (0'') -- (1'');
						\draw[line width=1.5pt,dashed,color=red,->] (4'') -- (2'');
						\draw[line width=1.5pt,dotted,color=green,->] (2'') -- (1'');
						\draw[line width=1.5pt,dotted,color=green,->] (3'') -- (4'');
						\draw[line width=1.3pt,M edge] (2'') -- (3'');
					\end{tikzpicture}
				}
			\end{subfigure}
			\caption{Exchange performed in the proof of Claim~\ref{claim:open-chain}
				in the case \(\tr(T_3) = \cv_2(T_2)\).}
			\label{fig:case5-similar}
		\end{figure}

		Now, let \(S = T_0, T_1,\ldots, T_{s-1}\) be an open chain in \(\D\).
		Let \(T_j = a_{0,j}a_{1,j}a_{2,j}a_{3,j}a_{4,j}a_{5,j}\), for \(j\in\{0,\ldots,s-1\}\), where \(a_{2,j}a_{3,j}\in M\) and \(a_{5,j}=a_{2,j}\) for \(j\in\{0,\ldots,s-2\}\).
		
		\begin{claim}\label{claim:open-chains-1}
			\(T_1\) is an element of type~A and \(\tr(T_1) = \cv_1(T_0)\).
		\end{claim}
		
		\begin{proof}
			Suppose, for a contradiction, that \(T_1\) is not an element of type~A or \({\tr(T_1) = \cv_2(T_0)}\).
			We claim that \(T_1\) does not contain a hanging edge at \(\cv_1(T_0)\).
			Indeed, if \(T_1\) is not an element of type~A, then, by the definition of open chain, \(T_1\) is an element of type~C,
			and hence, by Remark~\ref{remark:exceptional-hanging-edge}, \(T_1\) does not have a hanging edge at \(\cv_1(T_0)\);
			and if \(T_1\) is an element of type~A for which \(\tr(T_1) = \cv_2(T_0)\), 
			then we have \(a_{4,1} = \tr(T_1) = \cv_2(T_0) = a_{2,0}\),
			and hence, if \(a_{1,1} = a_{3,0}\), then we have \(a_{4,1} + r + g = a_{1,1} = a_{3,0} = a_{2,0} - r - g\),
			which implies that \(2g + 2r = 0\), a contradiction.
			Therefore, \(T_1\) does not contain a hanging edge at \(\cv_1(T_0)\).
			By Definition~\ref{def:semi-complete-commutative}\eqref{def:semi-complete-commutative-hanging-edge},
			there are two hanging edges at \(\cv_1(T_0)\).
			Thus, there is an element \(T = u_0u_1u_2u_3u_4u_5\) in \(\D\setminus\{T_0,T_1\}\) that contains a hanging edge, say \(u_1u_0\), at \(\cv_1(T_0)\).
			Note that all the edges incident to \(a_{2,0}\) are in \(E(T_0)\cup E(T_1)\).
			Let \(T_0' = a_{0,0}a_{1,0}a_{2,0}a_{4,0}a_{3,0}u_0\) and \(T' = a_{2,0}u_1u_2u_3u_4u_5\) and put \(\D' = \big(\D\setminus\{T_0,T\}\big)\cup\{T_0',T'\}\).
			By Lemma~\ref{lemma:T'1-is-path}, \(T_0'\) is a path;
			and since all the edges incident to \(a_{2,0}\) are in \(E(T_0)\cup E(T_1)\), we have \(a_{2,0}\notin \{u_1,u_2,u_3,u_4,u_5\}\), 
			and hence if \(T\) is a path (resp. an element of type~A), then \(T'\) is a path (resp. an element of type~A).
			Thus Definition~\ref{def:semi-complete-commutative}\eqref{def:complete-commutative-types} holds for \(\D'\).	
			
			To check that $\D'$ is an admissible decomposition first
			note that \(\hang_{\D'}(v) \geq \hang_{\D}(v)\) 
			for every \(v\in V(G)\setminus\{a_{4,0}\}\). 
			Thus, since \(T_0\) is a free element, \(a_{4,0}\) is not a connection vertex in \(\D\),
			and hence \(a_{4,0}\) is not a connection vertex in \(\D'\).
			Note also that \(T_1\) is either an element of type~C or a free element of type~A,
			and hence \(S' = T_1,\ldots, T_{s-1}\) is an open chain.
			Thus, Definition~\ref{def:semi-complete-commutative}\eqref{def:semi-complete-commutative-hanging-edge} 
			holds for \(\D'\). 
			Analogously to the cases above, every element of type~A in $\D'$ is in an A-chain of $\D'$.
			Thus, Definition~\ref{def:semi-complete-commutative}\eqref{def:semi-complete-free} holds for \(\D'\).
			Therefore, \(\D'\) is an admissible decomposition of \(G\)
			such that \(\tau(\D')  = \tau(\D) - 1\), a contradiction to the minimality of \(\D\).	
		\end{proof}
		
		By Claim~\ref{claim:open-chains-1}, we have \(s\geq 3\),
		and hence, there is an element \(T_2\) in \(S\).
		Note that, by Lemma~\ref{lemma:type12}, since \(\tr(T_1) = \cv_1(T_0)\), we have \(\aux(T_1) = \cv_2(T_0)\).
		This implies that \(a_{1,1}a_{0,1}\in F_r\) because all the edges incident to \(a_{1,1} = \cv_2(T_0)\) are in \(E(T_0)\cup E(T_1)\).
		
		\begin{claim}\label{claim:open-chains-2}
			\(T_2\) is of type~A.
		\end{claim}
		
		\begin{proof}
			Suppose, for a contradiction, that \(T_2\) is not of type~A,
			then \(T_2\) is an element of type~C
			and \((T_1,T_2)\) is an exceptional pair.
			Thus, we can write \(T_2 = a_{0,2}a_{1,2}a_{2,2}a_{3,2}a_{4,2}a_{5,2}\)
			such that \(a_{2,1}=a_{5,1} = a_{3,2}\),
			and \(a_{2,2}a_{1,2},a_{4,2}a_{3,2}\in F_g\), 
			\(a_{3,2}a_{2,2},a_{4,2}a_{5,2}\in F_r\),
			\(a_{1,2}a_{0,2}\in M_{g,r}\cup F_g\cup F_r\).
			We claim that \(a_{0,1} = a_{1,2}\).
			Indeed, since \({a_{1,1}a_{0,1}\in F_r}\), we have \({a_{0,1} = a_{1,1} + r = a_{2,1} + g + r}\),
			but by the definition of type~C, we have \({a_{1,2} = a_{2,2} + g = a_{3,2} + r + g}\).
			Thus, since \(a_{3,2} = a_{2,1}\), we obtain \(a_{0,1} = a_{1,2}\).	
			Now, put 
			\(T'_0 = a_{0,0}a_{1,0}a_{2,0}a_{4,0}a_{3,0}a_{2,1}\),
			\(T'_1 = a_{1,1}a_{4,1}a_{3,1}a_{2,1}a_{2,2}a_{0,1}\), and
			\(T'_2 = a_{0,2}a_{1,2}a_{1,1}a_{3,2}a_{4,2}a_{5,2}\)
			(see Figure~\ref{fig:claim5}),
			and put \(\D' = \big(\D\setminus\{T_0,T_1,T_2\}\big)\cup\{T'_0,T'_1,T'_2\}\).
			By Lemma~\ref{lemma:T'1-is-path}, \(T_0'\) is a path;
			since \(G\) is a simple graph, 
			\(a_{2,2}\notin \{a_{1,1},a_{4,1},a_{3,1},a_{2,1},a_{0,1}\}\), 
			and hence \(T'_1\) is a path;
			and since all edges incident to \(a_{1,1}\) are in \(E(T_0)\cup E(T_1)\),
			we have \(a_{1,1} \notin V(T_2)\), which implies that \(T'_2\) is a path.
			Thus Definition~\ref{def:semi-complete-commutative}\eqref{def:complete-commutative-types} holds for \(\D'\).	
			
			\begin{figure}[H]
				\centering
				\begin{subfigure}{.45\textwidth}
					\centering
					\scalebox{.8}{\begin{tikzpicture}[scale=1.5]
						
							\draw[fatpath,backcolor1] (-0.5,0.8) -- (-1.5,0.8) -- (-0.5,0.8) -- (0,0) -- (0,0) -- (1,0) -- (0.5,0.8) -- (0,0);
							\draw[fatpath,backcolor2] (0,0) -- (-0.5,-0.8) -- (0,0) -- (0.5,-0.8) -- (1,0) -- (1.5,-0.8) -- (0.5,-0.8);
						
							\node (0) [black vertex] at (-1.5,0.8) {};
							\node (1) [black vertex] at (-0.5,0.8) {};
							\node (2) [black vertex] at (0,0) {};
							\node (3) [black vertex] at (1,0) {};
							\node (4) [black vertex] at (0.5,0.8) {};
							\node () []		 at ($(0,0)+(60:1)$) {};

							\node (4') [black vertex] at ($(1,0)$) {};
							\node (3') [black vertex] at ($(3)+(0.5,-0.8)$) {};	
							\node (2') [black vertex] at ($(2)+(0.5,-0.8)$) {};
							\node (1') [black vertex] at ($(0,0)$) {};
							\node (0') [black vertex] at ($(0,0)+(-0.5,-0.8)$) {};
							\node (5') [] at ($(2,0)$) {};

							\node (1'') [black vertex] at ($(0')$) {};
							\node (0'') [black vertex] at ($(1'')+(-0.5,-0.8)$) {};	
							\node (2'') [black vertex] at ($(1'')+(0.5,-0.8)$) {};
							\node (3'') [black vertex] at ($(2')$) {};
							\node (4'') [black vertex] at ($(2')+(0.5,-0.8)$) {};
							\node (5'') [black vertex] at ($(4'')+(1,0)$) {};

							\node () [] at ($(0)+(180:.35)$) {$a_{0,0}$};
							\node () [] at ($(1)+(90:.2)$) {$a_{1,0}$};
							\node () [] at ($(2)+(180:.3)-(0.5,0)$) {$a_{2,0}=a_{1,1}$};
							\node () [] at ($(3)+(0:.35)+(0.4,0)$) {$a_{3,0}=a_{4,1}$};
							\node () [] at ($(4)+(90:.2)$) {$a_{4,0}$};

							\node () [] at ($(0')+(180:.35)-(0.4,.)$) {$a_{0,1}=a_{1,2}$};
							\node () [] at ($(1')+(-90:.3)$) {};
							\node () [] at ($(2')+(-90:.35)+(0.9,0)$) {$a_{2,1}=a_{3,2}$};
							\node () [] at ($(3')+(0:.35)$) {$a_{3,1}$};
							\node () [] at ($(4')+(-90:.3)$) {};

							\node () [] at ($(0'')+(180:.35)$) {$a_{0,2}$};
							\node () [] at ($(1'')+(270:.3)$) {};
							\node () [] at ($(2'')+(180:.35)$) {$a_{2,2}$};
							\node () [] at ($(3'')+(270:.25)$) {};
							\node () [] at ($(4'')+(180:.35)$) {$a_{4,2}$};
							\node () [] at ($(5'')+(0:.35)$) {$a_{5,2}$};

							\draw[line width=1.3pt,color=gray,<-] (0) -- (1);
							\draw[line width=1.5pt,dotted,color=green,->] (2) -- (1);
							\draw[line width=1.5pt,dotted,color=green,->] (3) -- (4);
							\draw[line width=1.5pt,dashed,color=red,->] (4) -- (2);
							\draw[line width=1.3pt,M edge] (2) -- (3);

							\draw[line width=1.5pt,dashed,color=red,<-] (0') -- (1');
							\draw[line width=1.5pt,dashed,color=red,->] (4') -- (2');
							\draw[line width=1.5pt,dotted,color=green,->] (2') -- (1');
							\draw[line width=1.5pt,dotted,color=green,->] (3') -- (4');
							\draw[line width=1.3pt,M edge] (2') -- (3');

							\draw[line width=1.5pt,color=gray,<-] (0'') -- (1'');
							\draw[line width=1.5pt,dashed,color=red,->] (4'') -- (5'');
							\draw[line width=1.5pt,dotted,color=green,->] (2'') -- (1'');
							\draw[line width=1.5pt,dotted,color=green,->] (4'') -- (3'');
							\draw[line width=1.5pt,dashed,color=red,->] (3'') -- (2'');
						\end{tikzpicture}
					}
				\end{subfigure}
				\begin{subfigure}{.45\textwidth}
					\centering
					\scalebox{.8}{\begin{tikzpicture}[scale=1.5]
						
							\draw[fatpath,backcolor1] (-0.5,0.8) -- (-1.5,0.8) -- (-0.5,0.8) -- (0,0) -- (0.5,0.8) -- (1,0) -- (0.5,-0.8);
							\draw[fatpath,backcolor2] (0,0) -- (1,0) -- (1.5,-0.8) -- (0.5,-0.8) -- (0,-1.6) -- (-0.5,-0.8) -- (0,-1.6);

							\node (0) [black vertex] at (-1.5,0.8) {};
							\node (1) [black vertex] at (-0.5,0.8) {};
							\node (2) [black vertex] at (0,0) {};
							\node (3) [black vertex] at (1,0) {};
							\node (4) [black vertex] at (0.5,0.8) {};
							\node () []		 at ($(0,0)+(60:1)$) {};

							\node (4') [black vertex] at ($(1,0)$) {};
							\node (3') [black vertex] at ($(3)+(0.5,-0.8)$) {};	
							\node (2') [black vertex] at ($(2)+(0.5,-0.8)$) {};
							\node (1') [black vertex] at ($(0,0)$) {};
							\node (0') [black vertex] at ($(0,0)+(-0.5,-0.8)$) {};
							\node (5') [] at ($(2,0)$) {};

							\node (1'') [black vertex] at ($(0')$) {};
							\node (0'') [black vertex] at ($(1'')+(-0.5,-0.8)$) {};	
							\node (2'') [black vertex] at ($(1'')+(0.5,-0.8)$) {};
							\node (3'') [black vertex] at ($(2')$) {};
							\node (4'') [black vertex] at ($(2')+(0.5,-0.8)$) {};
							\node (5'') [black vertex] at ($(4'')+(1,0)$) {};

							\node () [] at ($(0)+(180:.35)$) {$a_{0,0}$};
							\node () [] at ($(1)+(90:.2)$) {$a_{1,0}$};
							\node () [] at ($(2)+(180:.3)-(0.5,0)$) {$a_{2,0}=a_{1,1}$};
							\node () [] at ($(3)+(0:.35)+(0.4,0)$) {$a_{3,0}=a_{4,1}$};
							\node () [] at ($(4)+(90:.2)$) {$a_{4,0}$};

							\node () [] at ($(0')+(180:.35)-(0.4,.)$) {$a_{0,1}=a_{1,2}$};
							\node () [] at ($(1')+(-90:.3)$) {};
							\node () [] at ($(2')+(-90:.35)+(0.9,0)$) {$a_{2,1}=a_{3,2}$};
							\node () [] at ($(3')+(0:.35)$) {$a_{3,1}$};
							\node () [] at ($(4')+(-90:.3)$) {};

							\node () [] at ($(0'')+(180:.35)$) {$a_{0,2}$};
							\node () [] at ($(1'')+(270:.3)$) {};
							\node () [] at ($(2'')+(180:.35)$) {$a_{2,2}$};
							\node () [] at ($(3'')+(270:.25)$) {};
							\node () [] at ($(4'')+(180:.35)$) {$a_{4,2}$};
							\node () [] at ($(5'')+(0:.35)$) {$a_{5,2}$};

							\draw[line width=1.3pt,color=gray,<-] (0) -- (1);
							\draw[line width=1.5pt,dotted,color=green,->] (2) -- (1);
							\draw[line width=1.5pt,dotted,color=green,->] (3) -- (4);
							\draw[line width=1.5pt,dashed,color=red,->] (4) -- (2);
							\draw[line width=1.3pt,M edge] (2) -- (3);

							\draw[line width=1.5pt,dashed,color=red,<-] (0') -- (1');
							\draw[line width=1.5pt,dashed,color=red,->] (4') -- (2');
							\draw[line width=1.5pt,dotted,color=green,->] (2') -- (1');
							\draw[line width=1.5pt,dotted,color=green,->] (3') -- (4');
							\draw[line width=1.3pt,M edge] (2') -- (3');

							\draw[line width=1.5pt,color=gray,<-] (0'') -- (1'');
							\draw[line width=1.5pt,dashed,color=red,->] (4'') -- (5'');
							\draw[line width=1.5pt,dotted,color=green,->] (2'') -- (1'');
							\draw[line width=1.5pt,dotted,color=green,->] (4'') -- (3'');
							\draw[line width=1.5pt,dashed,color=red,->] (3'') -- (2'');
						\end{tikzpicture}
					}
				\end{subfigure}
				\caption{Exchange of edges between two elements of type~A
					and the elements of an exceptional pair
					in the proof of Claim~\ref{claim:open-chains-2}.}
				\label{fig:claim5}
			\end{figure}
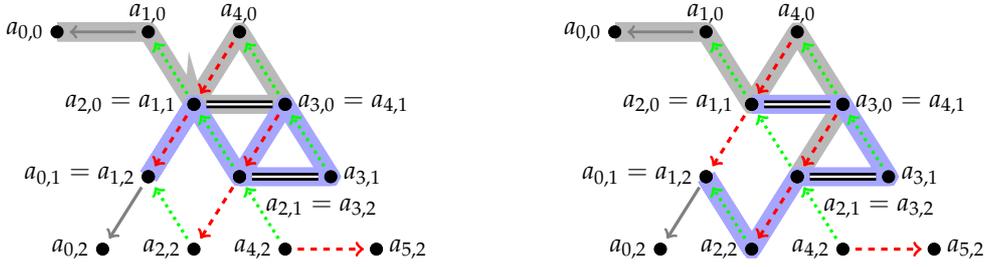
			
			To check that $\D'$ is an admissible decomposition first
			note that \(\hang_{\D'}(v) \geq \hang_{\D}(v)\) 
			for every \(v\in V(G)\setminus\{a_{4,0},a_{1,1},a_{3,1}\}\). 
			Thus, since \(T_0\) is a free element, \(a_{4,0}\) is not a connection vertex in \(\D\),
			and hence \(a_{4,0}\) is not a connection vertex in \(\D'\);
			and since the edges of \(M\) incident to \(a_{1,1}\) and \(a_{3,1}\) are in \(T'_1\),
			the vertices \(a_{1,1}\) and \(a_{3,1}\) are not connection vertices in \(\D'\).
			Note also that no element of \(S\) is in \(\D'\), and hence there are no open chains in \(\D'\).
			Thus, Definitions~\ref{def:semi-complete-commutative}\eqref{def:semi-complete-commutative-hanging-edge} and~\ref{def:semi-complete-commutative}\eqref{def:semi-complete-free}
			hold for~\(\D'\). 
			Therefore, \(\D'\) is an admissible decomposition of \(G\)
			such that \(\tau(\D')  = \tau(\D) - 2\), a contradiction to the minimality of~\(\D\).	
		\end{proof}
		
		Now, by Claim~\ref{claim:open-chains-2}, we have \(s\geq 4\).
		In what follows, we divide the proof depending on whether \(\tr(T_2) = \cv_1(T_1)\) or \(\tr(T_2) = \cv_2(T_1)\).
		
		\smallskip	\noindent
		\textbf{Case} \(\mathbf{\btr\boldsymbol{(}T_2\boldsymbol{)} \boldsymbol{=} \bcv_1\boldsymbol{(}T_1\boldsymbol{)}}\)\textbf{.}
		By Lemma~\ref{lemma:type12}, we have \(a_{1,2}=\aux(T_2) = \cv_2(T_1)=a_{2,1}\).
		Put \(T_0' = a_{0,0}a_{1,0}a_{2,0}a_{4,0}a_{3,0}a_{2,1}\),
		\(T_1' = a_{0,1}a_{1,1}a_{4,1}a_{3,1}a_{2,1}a_{0,2}\),
		\(T_2' = a_{1,1}a_{1,2}a_{2,2}a_{3,2}a_{4,2}a_{2,2}\) (see Figure~\ref{fig:case4-similar2}),
		let \(\D' = \big(\D\setminus\{T_0,T_1,T_2\}\big)\cup\{T_0',T_1',T_2'\}\),
		and let \(S' = T'_2, T_3, \dots, T_{s-1}\).
		By Lemma~\ref{lemma:T'1-is-path}, \(T'_0\) is an element of type~C.		
		In what follows, we prove that \(T_1'\) is a path and \(T_2'\) is an element of type~A.
		Since \(G\) is simple, we have \(a_{0,2}\notin \{a_{1,1},a_{2,1},a_{3,1},a_{4,1}\}\),
		and \(a_{1,1}\notin \{a_{1,2},a_{2,2},a_{4,2}\}\).
		By Lemma~\ref{lemma:no-cycles}, \(a_{0,2} \neq a_{0,1}\),
		and hence, $T_1'$ is a path.
		If \(a_{1,1} = a_{3,2}\),
		then \(a_{2,0}a_{3,0}\) and \(a_{2,2}a_{3,2}\) are two edges of \(M\) incident to \(a_{1,1}\),
		a contradiction.
		Therefore,~$T_2'$ is an element of type~A.
		
		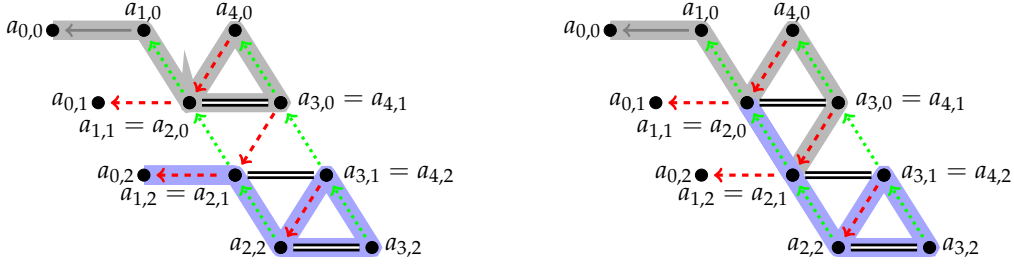
\begin{figure}[H]
			\centering
			\begin{subfigure}{.45\textwidth}
				\centering
				\scalebox{.8}{\begin{tikzpicture}[scale=1.5]
					
						\draw[fatpath,backcolor1] (-0.5,0.8) -- (-1.5,0.8) -- (-0.5,0.8) -- (0,0) -- (0,0) -- (1,0) -- (0.5,0.8) -- (0,0);
						\draw[fatpath,backcolor2] (0.5,-0.8) -- (-0.5,-0.8) -- (0.5,-0.8) -- (1,-1.6) -- (2,-1.6) -- (1.5,-0.8) -- (1,-1.6);

						\node (0) [black vertex] at (-1.5,0.8) {};
						\node (1) [black vertex] at (-0.5,0.8) {};
						\node (2) [black vertex] at (0,0) {};
						\node (3) [black vertex] at (1,0) {};
						\node (4) [black vertex] at (0.5,0.8) {};
						\node () []		 at ($(1,0)+(60:1)$) {};

						\node (4') [black vertex] at ($(1,0)$) {};
						\node (3') [black vertex] at ($(3)+(0.5,-0.8)$) {};	
						\node (2') [black vertex] at ($(2)+(0.5,-0.8)$) {};
						\node (1') [black vertex] at ($(0,0)$) {};
						\node (0') [black vertex] at ($(0,0)+(-1,0)$) {};
						\node (5') [] at ($(3,0)$) {};

						\node (4'') [black vertex] at ($(3')$) {};
						\node (3'') [black vertex] at ($(3)+(1,-1.6)$) {};	
						\node (2'') [black vertex] at ($(2)+(1,-1.6)$) {};
						\node (1'') [black vertex] at ($(2')$) {};
						\node (0'') [black vertex] at ($(1'')+(-1,0)$) {};
						\node (5'') [] at ($(3,0)$) {};

						\node () [] at ($(0)+(180:.3)$) {$a_{0,0}$};
						\node () [] at ($(1)+(90:.2)$) {$a_{1,0}$};
						\node () [] at ($(2)+(90:.2)+(0.5,0)$) {};
						\node () [] at ($(3)+(0:.2)+(0.6,0)$) {$a_{3,0}=a_{4,1}$};
						\node () [] at ($(4)+(90:.2)$) {$a_{4,0}$};

						\node () [] at ($(0')+(180:.3)$) {$a_{0,1}$};
						\node () [] at ($(1')+(-90:.3)-(0.6,0)$) {$a_{1,1}=a_{2,0}$};
						\node () [] at ($(2')+(90:.35)$) {};
						\node () [] at ($(3')+(0:.2)+(0.6,0)$) {$a_{3,1}=a_{4,2}$};
						\node () [] at ($(4')+(-90:.2)+(0.4,0)$) {};

						\node () [] at ($(0'')+(180:.3)$) {$a_{0,2}$};
						\node () [] at ($(1'')+(225:.35)-(0.4,0)$) {$a_{1,2}=a_{2,1}$};
						\node () [] at ($(2'')+(180:.35)$) {$a_{2,2}$};
						\node () [] at ($(3'')+(0:.35)$) {$a_{3,2}$};
						\node () [] at ($(4'')+(-90:.35)+(0.3,0)$) {};

						\draw[line width=1.3pt,color=gray,<-] (0) -- (1);
						\draw[line width=1.5pt,dotted,color=green,->] (2) -- (1);
						\draw[line width=1.5pt,dotted,color=green,->] (3) -- (4);
						\draw[line width=1.5pt,dashed,color=red,->] (4) -- (2);
						\draw[line width=1.3pt,M edge] (2) -- (3);

						\draw[line width=1.5pt,dashed,color=red,<-] (0') -- (1');
						\draw[line width=1.5pt,dashed,color=red,->] (4') -- (2');
						\draw[line width=1.5pt,dotted,color=green,->] (2') -- (1');
						\draw[line width=1.5pt,dotted,color=green,->] (3') -- (4');
						\draw[line width=1.3pt,M edge] (2') -- (3');

						\draw[line width=1.5pt,dashed,color=red,<-] (0'') -- (1'');
						\draw[line width=1.5pt,dashed,color=red,->] (4'') -- (2'');
						\draw[line width=1.5pt,dotted,color=green,->] (2'') -- (1'');
						\draw[line width=1.5pt,dotted,color=green,->] (3'') -- (4'');
						\draw[line width=1.3pt,M edge] (2'') -- (3'');
					\end{tikzpicture}
				}
			\end{subfigure}
			\begin{subfigure}{.45\textwidth}
				\centering
				\scalebox{.8}{\begin{tikzpicture}[scale=1.5]
					
						\draw[fatpath,backcolor1] (-0.5,0.8) -- (-1.5,0.8) -- (-0.5,0.8) -- (0,0) -- (0.5,0.8) -- (1,0) -- (0.5,-0.8);
						\draw[fatpath,backcolor2] (0.5,-0.8) -- (0,0) -- (0.5,-0.8) -- (1,-1.6) -- (2,-1.6) -- (1.5,-0.8) -- (1,-1.6);

						\node (0) [black vertex] at (-1.5,0.8) {};
						\node (1) [black vertex] at (-0.5,0.8) {};
						\node (2) [black vertex] at (0,0) {};
						\node (3) [black vertex] at (1,0) {};
						\node (4) [black vertex] at (0.5,0.8) {};
						\node () []		 at ($(1,0)+(60:1)$) {};

						\node (4') [black vertex] at ($(1,0)$) {};
						\node (3') [black vertex] at ($(3)+(0.5,-0.8)$) {};	
						\node (2') [black vertex] at ($(2)+(0.5,-0.8)$) {};
						\node (1') [black vertex] at ($(0,0)$) {};
						\node (0') [black vertex] at ($(0,0)+(-1,0)$) {};
						\node (5') [] at ($(3,0)$) {};

						\node (4'') [black vertex] at ($(3')$) {};
						\node (3'') [black vertex] at ($(3)+(1,-1.6)$) {};	
						\node (2'') [black vertex] at ($(2)+(1,-1.6)$) {};
						\node (1'') [black vertex] at ($(2')$) {};
						\node (0'') [black vertex] at ($(1'')+(-1,0)$) {};
						\node (5'') [] at ($(3,0)$) {};

						\node () [] at ($(0)+(180:.3)$) {$a_{0,0}$};
						\node () [] at ($(1)+(90:.2)$) {$a_{1,0}$};
						\node () [] at ($(2)+(90:.2)+(0.5,0)$) {};
						\node () [] at ($(3)+(0:.2)+(0.6,0)$) {$a_{3,0}=a_{4,1}$};
						\node () [] at ($(4)+(90:.2)$) {$a_{4,0}$};

						\node () [] at ($(0')+(180:.3)$) {$a_{0,1}$};
						\node () [] at ($(1')+(-90:.3)-(0.6,0)$) {$a_{1,1}=a_{2,0}$};
						\node () [] at ($(2')+(90:.35)$) {};
						\node () [] at ($(3')+(0:.2)+(0.6,0)$) {$a_{3,1}=a_{4,2}$};
						\node () [] at ($(4')+(-90:.2)+(0.4,0)$) {};

						\node () [] at ($(0'')+(180:.3)$) {$a_{0,2}$};
						\node () [] at ($(1'')+(225:.35)-(0.4,0)$) {$a_{1,2}=a_{2,1}$};
						\node () [] at ($(2'')+(180:.35)$) {$a_{2,2}$};
						\node () [] at ($(3'')+(0:.35)$) {$a_{3,2}$};
						\node () [] at ($(4'')+(-90:.35)+(0.3,0)$) {};

						\draw[line width=1.3pt,color=gray,<-] (0) -- (1);
						\draw[line width=1.5pt,dotted,color=green,->] (2) -- (1);
						\draw[line width=1.5pt,dotted,color=green,->] (3) -- (4);
						\draw[line width=1.5pt,dashed,color=red,->] (4) -- (2);
						\draw[line width=1.3pt,M edge] (2) -- (3);

						\draw[line width=1.5pt,dashed,color=red,<-] (0') -- (1');
						\draw[line width=1.5pt,dashed,color=red,->] (4') -- (2');
						\draw[line width=1.5pt,dotted,color=green,->] (2') -- (1');
						\draw[line width=1.5pt,dotted,color=green,->] (3') -- (4');
						\draw[line width=1.3pt,M edge] (2') -- (3');

						\draw[line width=1.5pt,dashed,color=red,<-] (0'') -- (1'');
						\draw[line width=1.5pt,dashed,color=red,->] (4'') -- (2'');
						\draw[line width=1.5pt,dotted,color=green,->] (2'') -- (1'');
						\draw[line width=1.5pt,dotted,color=green,->] (3'') -- (4'');
						\draw[line width=1.3pt,M edge] (2'') -- (3'');
					\end{tikzpicture}
				}
			\end{subfigure}
			\caption{Exchange of edges between the three first elements in an open A-chain with at least four elements, in the case \(\tr(T_3) = \cv_1(T_2)\).}
			\label{fig:case4-similar2}
		\end{figure}
		
		To check that $\D'$ is an admissible decomposition first
		note that \(\hang_{\D'}(v) \geq \hang_{\D}(v)\) 
		for every \(v\in V(G)\setminus\{a_{3,0}, a_{3,1}, a_{4,0}\}\),
		but since \(T_0\) is free, \(a_{4,0}\) is not a connection vertex in \(\D\),
		and hence \(a_{3,0}, a_{3,1},a_{4,0}\) are not connection vertices in \(\D'\).
		Thus, since \(a_{3,1}\) is not a connection vertex in \(\D'\),
		the element \(T'_2\) is free.
		Therefore, \(S'\) is an open chain,
		and hence Definition~\ref{def:semi-complete-commutative}\eqref{def:semi-complete-commutative-hanging-edge} 
		holds for \(\D'\). 
		Since, $T'_0$ and $T'_1$ are paths and $T'_2$ is an element of type~A, 
		Definition~\ref{def:semi-complete-commutative}\eqref{def:complete-commutative-types} holds for \(\D'\).	
		Analogously to the cases above, every element of type~A in $\D'$ is in an A-chain of $\D'$.
		Thus, Definition~\ref{def:semi-complete-commutative}\eqref{def:semi-complete-free} holds for \(\D'\).
		Therefore, \(\D'\) is an admissible decomposition of \(G\)
		such that \(\tau(\D')  = \tau(\D) - 2\), a contradiction to the minimality of \(\D\).	
		
		\smallskip	\noindent
		\textbf{Case} \(\mathbf{\btr\boldsymbol{(}T_2\boldsymbol{)} \boldsymbol{=} \bcv_2\boldsymbol{(}T_1\boldsymbol{)}}\)\textbf{.}
		Put \(T_0' = a_{0,0}a_{1,0}a_{2,0}a_{4,0}a_{3,0}a_{2,1}\),
		\(T_1' = a_{0,1}a_{1,1}a_{4,1}a_{3,1}a_{2,1}a_{2,2}\),
		\(T_2' = a_{0,2}a_{1,2}a_{2,2}a_{3,2}a_{4,2}a_{1,1}\) (see Figure~\ref{fig:case5-similar2}),
		let \(\D' = \big(\D\setminus\{T_0,T_1,T_2\}\big)\cup\{T_0',T_1',T_2'\}\),
		and let \(S' = T_3, \dots, T_{s-1}\).
		By Lemma~\ref{lemma:T'1-is-path}, \(T'_1\) is an element of type~C.		
		In what follows, we prove that \(T_1'\) and \(T_2'\) are paths.
		Since \(G\) is simple, we have \(a_{2,2}\notin \{a_{1,1},a_{2,1},a_{3,1},a_{4,1}\}\),
		and \(a_{1,1}\notin \{a_{2,2},a_{3,2},a_{4,2}\}\).
		By Lemma~\ref{lemma:no-cycles}, \(a_{2,2} \neq a_{0,1}\), and \(a_{1,1} \neq a_{0,2}\). 
		Therefore, $T_2'$ is a path.
		If \(a_{1,1} = a_{1,2}\),
		then \(a_{2,2}a_{1,2}\) and \(a_{2,1}a_{1,1}\) are two green in edges of \(a_{1,1}\).
		Therefore, $T_2'$ is a path.
		
		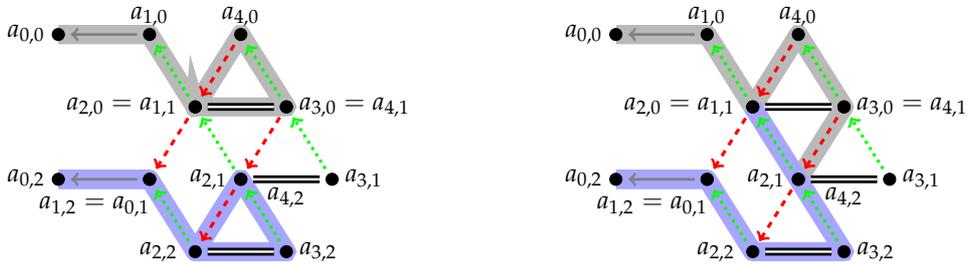
\begin{figure}[h]
			\centering
			\begin{subfigure}{.45\textwidth}
				\centering
				\scalebox{.8}{\begin{tikzpicture}[scale=1.5]
					
						\draw[fatpath,backcolor1] (-0.5,0.8) -- (-1.5,0.8) -- (-0.5,0.8) -- (0,0) -- (0,0) -- (1,0) -- (0.5,0.8) -- (0,0);
						\draw[fatpath,backcolor2] (-0.5,-0.8) -- (-1.5,-0.8)-- (-0.5,-0.8) -- (0,-1.6) -- (1,-1.6) -- (0.5,-0.8) -- (0,-1.6);
					
						\node (0) [black vertex] at (-1.5,0.8) {};
						\node (1) [black vertex] at (-0.5,0.8) {};
						\node (2) [black vertex] at (0,0) {};
						\node (3) [black vertex] at (1,0) {};
						\node (4) [black vertex] at (0.5,0.8) {};
						\node () []		 at ($(0,0)+(60:1)$) {};

						\node (4') [black vertex] at ($(1,0)$) {};
						\node (3') [black vertex] at ($(3)+(0.5,-0.8)$) {};	
						\node (2') [black vertex] at ($(2)+(0.5,-0.8)$) {};
						\node (1') [black vertex] at ($(0,0)$) {};
						\node (0') [black vertex] at ($(2')-(1,0)$) {};
						\node (5') [] at ($(2,0)$) {};

						\node (4'') [black vertex] at ($(2')$) {};
						\node (3'') [black vertex] at ($(3)-(0,1.6)$) {};	
						\node (2'') [black vertex] at ($(2)-(0,1.6)$) {};
						\node (1'') [black vertex] at ($(4'')-(1,0)$) {};
						\node (0'') [black vertex] at ($(1'')+(-1,0)$) {};
						\node (5'') [] at ($(2,0)$) {};

						\node () [] at ($(0)+(180:.35)$) {$a_{0,0}$};
						\node () [] at ($(1)+(90:.2)$) {$a_{1,0}$};
						\node () [] at ($(2)+(180:.15)-(0.65,0)$) {$a_{2,0}=a_{1,1}$};
						\node () [] at ($(3)+(0:.35)+(0.4,0)$) {$a_{3,0}=a_{4,1}$};
						\node () [] at ($(4)+(90:.2)$) {$a_{4,0}$};

						\node () [] at ($(0')+(180:.35)$) {};
						\node () [] at ($(1')+(240:.3)$) {};
						\node () [] at ($(2')+(180:.35)$) {$a_{2,1}$};
						\node () [] at ($(3')+(0:.35)$) {$a_{3,1}$};
						\node () [] at ($(4')+(-90:.3)-(0.4,0)$) {};

						\node () [] at ($(0'')+(180:.35)$) {$a_{0,2}$};
						\node () [] at ($(1'')+(-90:.3)-(0.6,0)$) {$a_{1,2}=a_{0,1}$};
						\node () [] at ($(2'')+(180:.4)$) {$a_{2,2}$};
						\node () [] at ($(3'')+(0:.35)$) {$a_{3,2}$};
						\node () [] at ($(4'')+(-90:.2)+(0.5,0)$) {$a_{4,2}$};

						\draw[line width=1.3pt,color=gray,<-] (0) -- (1);
						\draw[line width=1.5pt,dotted,color=green,->] (2) -- (1);
						\draw[line width=1.5pt,dotted,color=green,->] (3) -- (4);
						\draw[line width=1.5pt,dashed,color=red,->] (4) -- (2);
						\draw[line width=1.3pt,M edge] (2) -- (3);

						\draw[line width=1.5pt,dashed,color=red,<-] (0') -- (1');
						\draw[line width=1.5pt,dashed,color=red,->] (4') -- (2');
						\draw[line width=1.5pt,dotted,color=green,->] (2') -- (1');
						\draw[line width=1.5pt,dotted,color=green,->] (3') -- (4');
						\draw[line width=1.3pt,M edge] (2') -- (3');

						\draw[line width=1.3pt,color=gray,<-] (0'') -- (1'');
						\draw[line width=1.5pt,dashed,color=red,->] (4'') -- (2'');
						\draw[line width=1.5pt,dotted,color=green,->] (2'') -- (1'');
						\draw[line width=1.5pt,dotted,color=green,->] (3'') -- (4'');
						\draw[line width=1.3pt,M edge] (2'') -- (3'');
					\end{tikzpicture}
				}
			\end{subfigure}
			\begin{subfigure}{.45\textwidth}
				\centering
				\scalebox{.8}{\begin{tikzpicture}[scale=1.5]
					
						\draw[fatpath,backcolor1] (-0.5,0.8) -- (-1.5,0.8) -- (-0.5,0.8) -- (0,0) -- (0.5,0.8) -- (1,0) -- (0.5,-0.8);
						\draw[fatpath,backcolor2] (-0.5,-0.8) -- (-1.5,-0.8)-- (-0.5,-0.8) -- (0,-1.6) -- (1,-1.6) -- (0.5,-0.8) -- (0,0);
					
						\node (0) [black vertex] at (-1.5,0.8) {};
						\node (1) [black vertex] at (-0.5,0.8) {};
						\node (2) [black vertex] at (0,0) {};
						\node (3) [black vertex] at (1,0) {};
						\node (4) [black vertex] at (0.5,0.8) {};
						\node () []		 at ($(0,0)+(60:1)$) {};

						\node (4') [black vertex] at ($(1,0)$) {};
						\node (3') [black vertex] at ($(3)+(0.5,-0.8)$) {};	
						\node (2') [black vertex] at ($(2)+(0.5,-0.8)$) {};
						\node (1') [black vertex] at ($(0,0)$) {};
						\node (0') [black vertex] at ($(2')-(1,0)$) {};
						\node (5') [] at ($(2,0)$) {};

						\node (4'') [black vertex] at ($(2')$) {};
						\node (3'') [black vertex] at ($(3)-(0,1.6)$) {};	
						\node (2'') [black vertex] at ($(2)-(0,1.6)$) {};
						\node (1'') [black vertex] at ($(4'')-(1,0)$) {};
						\node (0'') [black vertex] at ($(1'')+(-1,0)$) {};
						\node (5'') [] at ($(2,0)$) {};

						\node () [] at ($(0)+(180:.35)$) {$a_{0,0}$};
						\node () [] at ($(1)+(90:.2)$) {$a_{1,0}$};
						\node () [] at ($(2)+(180:.15)-(0.65,0)$) {$a_{2,0}=a_{1,1}$};
						\node () [] at ($(3)+(0:.35)+(0.4,0)$) {$a_{3,0}=a_{4,1}$};
						\node () [] at ($(4)+(90:.2)$) {$a_{4,0}$};

						\node () [] at ($(0')+(180:.35)$) {};
						\node () [] at ($(1')+(240:.3)$) {};
						\node () [] at ($(2')+(180:.35)$) {$a_{2,1}$};
						\node () [] at ($(3')+(0:.35)$) {$a_{3,1}$};
						\node () [] at ($(4')+(-90:.3)-(0.4,0)$) {};

						\node () [] at ($(0'')+(180:.35)$) {$a_{0,2}$};
						\node () [] at ($(1'')+(-90:.3)-(0.6,0)$) {$a_{1,2}=a_{0,1}$};
						\node () [] at ($(2'')+(180:.4)$) {$a_{2,2}$};
						\node () [] at ($(3'')+(0:.35)$) {$a_{3,2}$};
						\node () [] at ($(4'')+(-90:.2)+(0.5,0)$) {$a_{4,2}$};

						\draw[line width=1.3pt,color=gray,<-] (0) -- (1);
						\draw[line width=1.5pt,dotted,color=green,->] (2) -- (1);
						\draw[line width=1.5pt,dotted,color=green,->] (3) -- (4);
						\draw[line width=1.5pt,dashed,color=red,->] (4) -- (2);
						\draw[line width=1.3pt,M edge] (2) -- (3);

						\draw[line width=1.5pt,dashed,color=red,<-] (0') -- (1');
						\draw[line width=1.5pt,dashed,color=red,->] (4') -- (2');
						\draw[line width=1.5pt,dotted,color=green,->] (2') -- (1');
						\draw[line width=1.5pt,dotted,color=green,->] (3') -- (4');
						\draw[line width=1.3pt,M edge] (2') -- (3');

						\draw[line width=1.3pt,color=gray,<-] (0'') -- (1'');
						\draw[line width=1.5pt,dashed,color=red,->] (4'') -- (2'');
						\draw[line width=1.5pt,dotted,color=green,->] (2'') -- (1'');
						\draw[line width=1.5pt,dotted,color=green,->] (3'') -- (4'');
						\draw[line width=1.3pt,M edge] (2'') -- (3'');
					\end{tikzpicture}
				}
			\end{subfigure}
			\caption{Exchange of edges between the three first elements in an open A-chain with at least four elements, in the case \(\tr(T_3) = \cv_2(T_2)\).}
			\label{fig:case5-similar2}
		\end{figure}
		
		To check that $\D'$ is an admissible decomposition first
		note that \(\hang_{\D'}(v) \geq \hang_{\D}(v)\) 
		for every \(v\in V(G)\setminus\{a_{3,0}, a_{3,1}, a_{3,2},a_{4,0}\}\),
		but since \(T_0\) is free, \(a_{3,0}, a_{3,1}, a_{3,2},a_{4,0}\) are not connection vertices in \(\D\),
		and hence \(a_{4,0}\) is not a connection vertex in \(\D'\).
		Thus, since $a_{2,2}$ and $a_{3,2}$ are not connection vertices in~$\D'$,
		the element $T_3$ is free.
		Therefore, \(S'\) is either an open chain or contains only one element,
		which is of type~C,
		and hence Definition~\ref{def:semi-complete-commutative}\eqref{def:semi-complete-commutative-hanging-edge} 
		holds for~\(\D'\). 
		Since, $T'_0$, $T'_1$ and $T'_2$ are paths, 
		Definition~\ref{def:semi-complete-commutative}\eqref{def:complete-commutative-types} holds for \(\D'\).	
		Analogously to the cases above, every element of type~A in $\D'$ is in an A-chain of $\D'$.
		Thus, Definition~\ref{def:semi-complete-commutative}\eqref{def:semi-complete-free} holds for \(\D'\).
		Therefore, \(\D'\) is an admissible decomposition of \(G\)
		such that \(\tau(\D')  = \tau(\D) - 3\), a contradiction to the minimality of \(\D\).
		This concludes the proof.
	\end{proof}
	
	Recall that a \(\{g,r\}\)-graph \(G\) is a \(5\)-regular graph that contains the Cayley graph \(X(\Gamma,S)\), where \(S = \{g,-g,r,-r\}\).
	Thus, since \(S\) is closed under taking inverses, \(G\) is also a \(\{g,-r\}\)-, \(\{-g,r\}\)-, \(\{-g,-r\}\)-graph,
	which yields the following corollary of Theorem~\ref{theorem:2g2r!=0}.
	
	\begin{corollary}\label{corollary:hard-equation}
		Every \(\{g,r\}\)-graph for which \(2g+2r \neq 0\) or \(2g-2r\neq 0\)
		admits a \(P_5\)-decomposition.
	\end{corollary}

	The main result of this paper is a straightforward consequence of Corollary~\ref{corollary:hard-equation} and Theorem~\ref{theorem:no-g2r2}.
	
	\begin{theorem}\label{thm:main}
		Every \(\{g,r\}\)-graph admits a \(P_5\)-decomposition.
	\end{theorem}

\section{Conclusion and future works} \label{sec:concluding}

In this paper, we verified Conjecture \ref{conj:favaron} for (i) 
$(2k +1)$-regular graphs containing a spanning $2k$-regular power of a cycle, 
and (ii) $5$-regular graphs containing special spanning $4$-regular Cayley graphs. 
We believe that the techniques developed here can be extended for a more general class of graphs, such as \emph{Schreier Coset Graphs} (see \cite{Gross77}).

Let $G$ be a group and let $H$ be a subgroup of $G$. 
For $s \in G$, the \emph{right coset} of $H$ corresponding to $s$ is the set $Hs=\{hs\colon h \in H\}$. 
\emph{Left cosets} can be defined analogously. 
Let $g_1,\dots,g_r$ be a sequence in $G$ whose members generate $G$, the \emph{Schreier Right Coset Graph} (SRCG) is defined as follows. 
Its vertex set is the set of right cosets of $H$ in $G$, 
for each coset $H'$ and each generator $g_i$ there is an edge from $H'$ to the right coset $H'g_i$.
In particular, a Cayley graph is an SRCG where $H=\{0\}$. 
Schreier coset graphs are generalization of Cayley ``color'' graphs using cosets of some specified subgroup as vertices instead of group elements. 
In 1977, Gross~\cite{Gross77} showed that every connected regular graph of even degree is an SRCG. 
This implies that, if we extend our result for 5-regular graphs that contain any spanning $4$-regular SRCG, 
then we verify the conjecture for $k=2$.

Finally, we can also explore others graphs containing special spanning Cayley graphs. For instance,  a natural step is to examine 7-regular graphs containing a spanning 4- or 6-regular Cayley graph. Also, note that the definitions of simple commutative generator and $\{g,r\}$-graph are equivalent to Cayley graphs under the restriction of the equation $g+r = r+g$ for every pair of generators. Therefore, we plan to explore other restrictions, such as $g+r \neq r+g$, which would extend our result for 5-regular graphs containing every spanning 4-regular Cayley graph.

\section*{Acknowlegments}
This study was financed in part by the Coordenação de Aperfeiçoamento de Pessoal de Nível Superior - Brasil (CAPES) - Finance Code 001.
F. Botler is partially supported by CNPq (Grant 423395/2018-1) and by FAPERJ (Grant 211.305/2019).

	\bibliographystyle{amsplain}

\begin{thebibliography}{10}
		
		\bibitem{bondy1976graph}
		John~A. Bondy and Uppaluri S.~R. Murty.
		\newblock {\em Graph theory with applications}.
		\newblock American Elsevier Publishing Co., Inc., New York, 1976.
		
		\bibitem{BoMoOsWa17}
		F{\'a}bio Botler, Guilherme~O. Mota, Marcio T.~I. Oshiro, and Yoshiko
		Wakabayashi.
		\newblock Decomposing regular graphs with prescribed girth into paths of given
		length.
		\newblock {\em European J. Combin.}, 66:28--36, 2017.
		
		\bibitem{BoMoWa15}
		F{\'a}bio Botler, Guilherme~O. Mota, and Yoshiko Wakabayashi.
		\newblock Decompositions of triangle-free 5-regular graphs into paths of length
		five.
		\newblock {\em Discrete Math.}, 338(11):1845--1855, 2015.
		
		\bibitem{BoFo83}
		Andr\'{e} Bouchet and Jean-Luc Fouquet.
		\newblock Trois types de d\'{e}compositions d'un graphe en cha\^{\i}nes.
		\newblock In {\em Combinatorial mathematics ({M}arseille-{L}uminy, 1981)},
		vol.~75 of {\em North-Holland Math. Stud.}, pages 131--141. North-Holland,
		Amsterdam, 1983.
		
		\bibitem{Ca78}
		Arthur Cayley.
		\newblock Desiderata and {S}uggestions: {N}o. 2. {T}he {T}heory of {G}roups:
		{G}raphical {R}epresentation.
		\newblock {\em Amer. J. Math.}, 1(2):174--176, 1878.
		
		\bibitem{FaGeKo10}
		Odile Favaron, Fran\c{c}ois Genest, and Mekkia Kouider.
		\newblock Regular path decompositions of odd regular graphs.
		\newblock {\em J. Graph Theory}, 63(2):114--128, 2010.
		
		\bibitem{godsil13}
		Chris Godsil and Gordon Royle.
		\newblock {\em Algebraic graph theory}, volume 207 of {\em Graduate Texts in
			Mathematics}.
		\newblock Springer-Verlag, New York, 2001.
		
		\bibitem{Gross77}
		Jonathan~L. Gross.
		\newblock Every connected regular graph of even degree is a {S}chreier coset
		graph.
		\newblock {\em J. Combinatorial Theory Ser. B}, 22(3):227--232, 1977.
		
		\bibitem{Ko57}
		Anton Kotzig.
		\newblock Aus der {T}heorie der endlichen regul\"{a}ren {G}raphen dritten und
		vierten {G}rades.
		\newblock {\em \v{C}asopis P\v{e}st. Mat.}, 82:76--92, 1957.
		
		\bibitem{Pe1891}
		Julius Petersen.
		\newblock Die {T}heorie der regul\"{a}ren graphs.
		\newblock {\em Acta Math.}, 15(1):193--220, 1891.
		
	\end{thebibliography}

\end{document}